\documentclass{article}

\title{Twisted Hilbert modular surfaces, arithmetic intersections and the Jacquet-Langlands correspondence}

\date{}

\usepackage{geometry}

\usepackage[all]{xy}
\usepackage{stmaryrd}
\usepackage[usenames, dvipsnames]{color}
\usepackage[pdfborder={0 0 0}, colorlinks = true, linkcolor=blue, citecolor=OliveGreen]{hyperref}
\usepackage{url}
\usepackage{graphicx}
\usepackage{amssymb}
\usepackage{amsthm}
\usepackage{amsmath, amscd}
\usepackage{mathrsfs}
\usepackage{paralist}
\usepackage{framed}
\usepackage{float}
\usepackage{mathpazo}
\usepackage{authblk}

\usepackage{tikz, wrapfig}
\usetikzlibrary{trees}

\usepackage{sidecap}
\sidecaptionvpos{figure}{c}

\newcommand{\Abf}{\mathbf A}
\newcommand{\Gbf}{\mathbf G}
\newcommand{\Vbf}{\mathbf V}
\newcommand{\zbf}{\mathbf z}

\newcommand{\Fcal}{\mathcal F}
\newcommand{\Lcal}{\mathcal L}
\newcommand{\Mcal}{\mathcal M}
\newcommand{\Ocal}{\mathcal O}
\newcommand{\bbU}{\mathbb U}
\newcommand{\Xcal}{\mathcal{X}}
\newcommand{\Ecal}{\mathcal{E}}

\newcommand{\Zed}{\mathfrak{Z}}

\newcommand{\Lbb}{\mathbb L}
\newcommand{\Dbb}{\mathbb D}

\newcommand{\Kstar}{\mathbb K}
\newcommand{\KstarN}{{\Kstar(N)}}

\usepackage{amsfonts} 
\newcommand{\adeg}{\widehat{\deg}\;}
\DeclareMathOperator{\Q}{\mathbb{Q}}
\DeclareMathOperator{\Z}{\mathbb{Z}}
\DeclareMathOperator{\C}{\mathbb{C}}
\DeclareMathOperator{\A}{\mathbb{A}}
\DeclareMathOperator{\R}{\mathbb{R}}

\DeclareMathOperator\Aut{Aut}
\DeclareMathOperator\End{End}
\DeclareMathOperator\Hom{Hom}
\DeclareMathOperator\Spec{Spec}

\DeclareMathOperator{\cpd}{\overline{\partial}}

\DeclareMathOperator{\ach}{\widehat{ch}}
\DeclareMathOperator{\atd}{\widehat{td}}
\DeclareMathOperator{\ch}{ch}
\DeclareMathOperator{\td}{td}
\DeclareMathOperator{\Real}{\texttt{Re}}
\DeclareMathOperator{\Imag}{\texttt{Im}}
\DeclareMathOperator{\Gm}{\mathbb{G}}

\newcommand{\ChernHatOne}{\widehat{c_1}}
\newcommand{\ChernHatTwo}{\widehat{c_2}}
\newcommand{\ac}{\widehat{c}}
\newcommand{\ov}{\overline}
\newcommand{\Scal}{\mathcal{S}}
\newcommand{\OF}{\mathcal O_F}
\newcommand{\dF}{\mathbf d_F}
\newcommand{\kbf}{\mathbf k}
\newcommand{\cz}{\overline{z}}

\DeclareMathOperator{\Nrd}{Nrd}
\DeclareMathOperator{\Trd}{Trd}

\DeclareMathOperator{\vol}{vol}
\DeclareMathOperator{\Ta}{\mathtt{Ta}}
\DeclareMathOperator{\Lie}{\mathsf{Lie}}

\DeclareMathOperator{\Id}{\mathtt{Id}}
\DeclareMathOperator\Fil{\mathtt{Fil}}
\DeclareMathOperator{\tr}{tr}
\newcommand{\Ch}[1]{\widehat{\mathsf{CH}}{}^{\, #1}}

\newcommand{\la}{\langle} 	
\newcommand{\ra}{\rangle} 	
\newcommand{\lie}[1]{\mathfrak{#1} } 
\newcommand{\isomto}{\overset{\sim}{\longrightarrow}} 	

\newcommand{\taut}{\mathsf{taut}}

\newcommand{ \refThm}[1]{\hyperref[#1]{Theorem \ref{#1}}}
\newcommand{ \refSec}[1]{\hyperref[#1]{Section \ref{#1}}}
\newcommand{ \refProp}[1]{\hyperref[#1]{Proposition \ref{#1}}}
\newcommand{ \refLemma}[1]{\hyperref[#1]{Lemma \ref{#1}}}
\newcommand{\refRemark}[1]{\hyperref[#1]{Remark \ref{#1}}}
\newcommand{\refDef}[1]{\hyperref[#1]{Definition \ref{#1}}}
\newcommand{\refCor}[1]{\hyperref[#1]{Corollary \ref{#1}}}

\theoremstyle{plain}
\newtheorem{theorem}{Theorem}[section]
\newtheorem*{theorem*}{Theorem}
\newtheorem*{maintheorem*}{Main Theorem}

\newtheorem{lemma}[theorem]{Lemma}
\newtheorem{proposition}[theorem]{Proposition}
\newtheorem{definition}[theorem]{Definition}
\newtheorem{corollary}[theorem]{Corollary}

\newtheorem*{claim*}{Claim}

\numberwithin{equation}{section}

\theoremstyle{remark}
\newtheorem{remark}[theorem]{Remark}

\begin{document}

\author{Gerard Freixas i Montplet\thanks{Centre National de la Recherche Scientifique (CNRS), France. \url{gerard.freixas@imj-prg.fr}} \  and Siddarth Sankaran\thanks{Department of Mathematics, University of Manitoba, Canada. \url{siddarth.sankaran@umanitoba.ca}}}

\maketitle

\begin{abstract}
We study arithmetic intersections on quaternionic Hilbert modular surfaces and Shimura curves over a real quadratic field. 
Our first main result is the determination of the degree of the top arithmetic Todd class of an arithmetic twisted Hilbert modular surface. This quantity is then related to the arithmetic volume of a Shimura curve, via the arithmetic Grothendieck-Riemann-Roch theorem and the Jacquet-Langlands correspondence.
\end{abstract}

\tableofcontents

\section{Introduction}
The aim of this paper is to compare arithmetic intersection numbers on Shimura varieties attached to inner forms of $GL_2$ over a real quadratic field $F$. 

To illustrate our approach, first consider the arithmetic volume of  a Shimura variety attached to an inner form of $GL_2$ over $\Q$. In the case of a modular curve, this volume was computed by Bost and K\"uhn, see  \cite{Kuhn-thesis}, while the case of Shimura curves was handled by Kudla-Rapoport-Yang \cite{KRY-book}; comparing these formulas, one finds that up to some contributions at primes of bad reduction, they are essentially identical. In \cite{Freixas-Q}, the first author gave an explanation for this coincidence via his arithmetic Grothendieck-Riemann-Roch theorem for pointed stable curves \cite{Freixas:ARR}, which expresses the arithmetic volumes in terms of holomorphic modular forms and Maa\ss\ forms. The automorphic contributions can be matched in the two cases via the Jacquet-Langlands correspondence, yielding an identity of arithmetic volumes. 

In the present work, we begin with a division quaternion algebra $B$ over $F$ that is split at both infinite places, and whose  discriminant 
\[
D_B \ := \ (p_1 \cdots p_r) \cdot \OF
\]
is a nonempty product of split rational primes. Let $\Gbf$ denote the algebraic group over $\Q$ such that
\begin{equation}\label{def:G}
	\Gbf(A) = \{ b \in B \otimes_{\Q} A \ | \ \Nrd(b) \in A^{\times}  \}
\end{equation}
for any $\Q$-algebra $A$; here $\Nrd$ is the reduced norm on $B$. Given a sufficiently small compact open subgroup $K \subset \Gbf(\A_{\Q,f})$, we obtain the \emph{twisted Hilbert modular surface}
\[
M \ := \  \Gbf(\Q) \big\backslash X \times \Gbf(\A_{\Q,f}) / K 
\]
where $X := \{ (z_1, z_2) \in \C^2  | \Imag(z_1) \Imag(z_2) > 0 \}$. This is a Shimura variety with a proper canonical model over $\Q$. Moreover, it admits a PEL interpretation, as described by Kudla-Rapoport \cite{KR-HB}, giving rise to a smooth, proper integral model $\Mcal$ over $\Spec \Z[1/N]$ for a suitable integer $N$, cf.\ Section \ref{sec:HMSIntModel} below.

We may also construct a \emph{Shimura curve} as follows: fix an embedding $v_1 \colon F \hookrightarrow \R$ and an inert prime $\ell$ that is relatively prime to $D_B$. Let $B_1$ denote the quaternion algebra over $F$ whose discriminant is $D_{B_1} = \ell D_B$, and is ramified at $v_1$, and define
\[
	\Gbf_1 \ := \ \mathrm{Res}_{F/\Q} (B_1^{\times})
\]
viewed as an algebraic group over $\Q$. For a fixed sufficiently small $K_1 \subset \Gbf_1(\A_{\Q,f})$, we have the \emph{Shimura curve}
\[
 S_1  \ = \  \Gbf_1(\Q) \big\backslash \lie H^{\pm} \times \Gbf_1(\A_{\Q,f}) / K_1.
\]
This Shimura variety, which is not of PEL type, admits a proper canonical model over $F$. Although the choice of an integral model will not play a role in the present work, for concreteness fix a model $\Scal_1$ that is smooth and proper over $\Spec \OF[1/N]$. 


The arithmetic Grothendieck-Riemann-Roch theorem, due to Gillet-Soul\'e \cite{Gillet-Soule:ARR}, relates the coherent cohomology of a hermitian vector bundle on an arithmetic variety to certain intersections of arithmetic characteristic classes. For our particular Shimura varieties, and applied to the trivial bundle, it takes the following form: for the twisted Hilbert modular surface $\Mcal$, which is an arithmetic three-fold, the formula reads
\begin{align}
	\widehat{\deg} (\det H^{\bullet}(\Mcal,& \Ocal_{\Mcal})_Q) \  \notag \\ 
  \label{eqn:IntroARRHMS}	&\equiv \ -\frac{1}{24} \adeg\ac_{1}(\overline{\Omega_{\Mcal}})\ac_{2}(\overline{\Omega_{\Mcal}})-\frac{1}{4}\left[2\zeta'(-1)+\zeta(-1) \right] 
  \deg \left( c_{1}(\Omega_{\Mcal(\C)})^{2}\right)
   \ \in \ \R_N;
\end{align}
here $\overline \Omega_{\Mcal}$ is the cotangent bundle $\Omega_{\Mcal} = \Omega_{\Mcal / \Z[1/N]}$ equipped with an invariant metric (canonical up to scaling), and the cohomology on the left hand side is equipped with the \emph{Quillen metric}  (see Section \ref{subsec:hol-an-torsion} below) which incorporates both the Petersson norm of twisted (i.e.\ quaternionic) Hilbert modular forms and the holomorphic analytic torsion of $\Mcal(\C)$. Furthermore, the equation \eqref{eqn:IntroARRHMS} takes place in the group 
\[ \R_N := \R \Big/ \bigoplus_{p|N} \Q \cdot \log p. \]

Similarly, for the Shimura curve $\Scal_1$, we have
\begin{align}
	\widehat{\deg}( \det H^{\bullet}(\Scal_1,& \Ocal_{\Scal_1})_{Q}) \notag \\
	\label{eqn:IntroARRSC} & \equiv  \ \frac{1}{12}\adeg\ac_{1}(\overline{\Omega}_{\Scal_{1}})^2 \ +\frac{1}{2}\left[2\zeta'(-1)+\zeta(-1)\right]  \, \deg c_{1}(\Omega_{\Scal_{1}(\C)} ) \ \in \ \R_{N}.
\end{align}

There are two main themes in this paper. The first is an explicit evaluation of the right hand side of \eqref{eqn:IntroARRHMS}, whose rough form is in the spirit of the conjectures of Maillot-R\"ossler \cite{MR}.
\begin{theorem} \label{thm:introc1c2}
	\[  \ \frac{ \widehat \deg \ \ChernHatOne( \overline \Omega_{\Mcal}) \cdot \ChernHatTwo( \overline \Omega_{\Mcal})} {\deg \left( c_1(\Omega_{\Mcal(\C)})^2 \right)} \ \equiv\   -4 \log \pi - 2\gamma + 1  +  \frac{\zeta'_F(2) }{\zeta_F(2)}   \ \in \ \R_N. \]
Here $\gamma = - \Gamma'(1)$ is the Euler-Mascheroni constant.
\end{theorem}
The proof uses the fortunate coincidence that $\Mcal$ can be simultaneously realized as a PEL Shimura variety and a $GSpin$ Shimura variety; the first interpretation allows us to decompose the cotangent bundle in a convenient way, while the second allows us to express $\ac_1(\ov\Omega_{\Mcal})$ in terms of Borcherds forms.  One technical obstacle is that the absence of cusps on $\Mcal$ renders the arithmetic properties of Borcherds forms less accessible; to circumvent this, we embed $\Mcal$ in an arithmetic twisted Siegel modular threefold and work there. In the end, the proof combines a formula of H\"ormann \cite{Hor} for the arithmetic volume of $\Mcal$, a formula for the integral of a Borcherds form in the style of Kudla \cite{Kudla-IBF}, and work of Kudla-Yang \cite{KY} on the Fourier coefficients of Eisenstein series for $GL_2$.

The second main result in this paper is a comparison of the left hand sides of the arithmetic Grothendieck-Riemann-Roch formulas \eqref{eqn:IntroARRHMS} and \eqref{eqn:IntroARRSC}.
\begin{theorem} \label{thm:intro-main-comparison}
\begin{equation}\label{eq:intro-main-comparison}
	\frac{\adeg\ac_{1}(\ov{\Omega}_{\Mcal})\ac_{2}(\ov{\Omega}_{\Mcal})}{\deg \left( c_{1}(\Omega_{\Mcal(\C)})^{2}\right)} \ \equiv \ 	\frac{\adeg \left(\ac_{1}(\ov{\Omega}_{S_{1}})^{2}\right)}{\deg c_{1}(\Omega_{\Scal_{1}(\C)})}\quad\text{in }\R/\log |\ov{\Q}^{\times}|,
\end{equation}
where $\ov{\Q}$ is the algebraic closure of $\Q$ in $\C$.   As a corollary,
	\begin{equation} \label{eq:intro-main-scvol}
	 \frac{	\adeg \left( \ac_1(\overline \Omega_{\Scal_1})^2 \right) }{ \deg c_1(\Omega_{\Scal_{1}(\C)})  }\ \equiv  \  -4 \log \pi - 2\gamma + 1  +  \frac{\zeta'_F(2) }{\zeta_F(2)}  \in \R / \log|\ov\Q^{\times}|.
	\end{equation}
\end{theorem}
The arithmetic Grothendieck-Riemann-Roch formula expresses both arithmetic intersection numbers in terms of automorphic representations. In particular the holomorphic analytic torsion encodes the contributions from Maa\ss\ forms. After some preliminary reduction steps, the proof amounts to identifying those representations that give non-trivial contributions to the arithmetic Grothendieck-Riemann-Roch for $\Mcal$, and comparing, via Jacquet-Langlands, with the analogous contributions for the Shimura curve. The Shimura period relations provide the matching of the contributions from holomorphic forms; our comparison of holomorphic analytic torsions, which is the technical heart of this section, can be viewed as a non-holomorphic counterpart to Shimura's result.


Observe that \eqref{eq:intro-main-comparison} takes place in $\R / \log|\ov \Q^{\times}|$. Though one  hopes for a relation in $\R$, or slightly less ambitiously in $\R_N$, there are serious obstacles:

\begin{enumerate}
	\item the ambiguity $\log|\ov{\Q}^{\times}|$ could be strengthened to $\log |\Q^{\times}|$ if the Shimura period relations were compatible with Galois action; to the authors' knowledge, this has not been established.
	\item Even given this  compatibility, a refinement from $\R / \log |\Q^{\times}|$ to $\R_N$ for a suitable $N$ (depending, say, on $F$, $D_{B}$ , $K$ and $K_1$), would require a significantly more difficult treatment of integrality questions. In particular, one needs a compatibility statement of Jacquet-Langlands on quaternionic holomorphic modular forms with integral models of  Shimura varieties, analagous to the work of Prasanna \cite{Prasanna} in the case $F=\Q$. The corresponding statement for real quadratic $F$ is not known, although there are partial results due to Hida \cite{Hida} and Ichino-Prasanna \cite{Ichino-Prasanna}.

	\item If there were a version of this theorem valued in $\R$, it would necessarily be sensitive to the choices of compact open  subgroups $K$ and $K_1$, as well as the bad reduction of integral models.
	
	On the other hand, the theorem holds, as stated,  for all sufficiently small $K$ and $K_1$. This provides evidence for the conjecture that the arithmetic degree of the top arithmetic Todd class (normalized by the geometric degree, and viewed in $\R / \log|\ov{\Q}^{\times}|$) should be the same for any positive dimensional Shimura variety attached to an inner form of $GL_2/F$, and is therefore a quantity intrinsic to that group.
	
	\item Yuan-Zhang-Zhang \cite{YZZ} have computed the full arithmetic volume (i.e.\ in $\R$) of a Shimura curve for a particular choice of level structure $K_1$ and integral model. The formula \eqref{eq:intro-main-scvol} is consistent with their computation; comparing their formula with Theorem \ref{thm:introc1c2} gives some measure of evidence for the presence of an integral Jacquet-Langlands in this particular case.
		
\end{enumerate}
	

Finally, note that we have excluded the case $B=M_2(F)$ throughout. The corresponding Shimura variety, which is a classical Hilbert modular surface, is not compact; passing to a toroidal compactification introduces logarithmic singularities for the metric on the cotangent bundle, and so the arithmetic Grothendieck-Riemann-Roch theorem does not apply. On the other hand, the methods of this paper suggest an automorphic approach to a conjectural Grothendieck-Riemann-Roch formula such that Theorem \ref{thm:intro-main-comparison} continues to hold, and in particular, a conjectural formula for the boundary contributions in terms of derivatives of Shimizu $L$-functions.  The authors, together with D.\ Eriksson, will address this question in forthcoming work.

%

\paragraph{Acknowledgements} The authors thank Dennis Eriksson for many  discussions on the subject of this paper, in particular his insightful comments on the non-compact case.

We also thank G.\ Chenevier, M.-H. Nicole, S.\ Kudla,  K.\ Prasanna, M.\ Raum and X.\ Yuan for helpful conversations, and  X.\ Yuan, S.-w.\ Zhang, and W.\ Zhang  for generously sharing a preliminary version of their article.

Work on this article was carried out during a Junior Trimester Program at the Hausdorff Institute in Bonn and at the CRM in Montr\'eal; the authors are grateful to these institutions for the hospitality. G.F.\ acknowledges financial support from UMI-CNRS in Montr\'eal and the ANR, and S.S. financial support from NSERC.

\section{Complex Shimura varieties}\label{section:complex}
In this section we briefly describe the complex points of the Shimura varieties to be studied in this article.
\subsection{Twisted Hilbert modular surfaces}\label{subsec:complex-twisted}
Recalling that $B$ is a totally indefinite division quaternion algebra, we may fix  isomorphisms
\[
	B\otimes_{F, v_1}\R  \ \simeq \ M_2(\R) \ \simeq \ B \otimes_{F,v_2} \R,
\]
so that $B\otimes_{\Q} \R \simeq M_2(\R) \oplus M_2(\R)$ and
\[
	\Gbf(\R) \ \simeq \ \{  (A_1,A_2) \in GL_2(\R) \times GL_2(\R) \ | \ \det A_1 = \det A_2    \}.
\]
Let
\[ X \ := \ \left\{ (z_1, z_2) \in \C^2 \ | \ \Imag(z_1) \cdot \Imag(z_2) > 0 \right\}, \]
which admits an action of $\Gbf(\R)$   by fractional linear transformations in each component:
\[
	(A_1,A_2) \cdot (z_1, z_2) \ = \ \left( \frac{ a_1 z_1 + b_1 }{c_1 z_1 + d_1}, \ \frac{a_2 z_2 + b_2}{c_2 z_2 + d_2} \right)
\]
for all $(z_1, z_2) \in X$ and $(A_1, A_2) = ( (\begin{smallmatrix}a_1&b_1 \\ c_1 & d_1 \end{smallmatrix}), (\begin{smallmatrix}a_2&b_2 \\ c_2 & d_2 \end{smallmatrix})) \in \Gbf(\R)$.
%
\begin{definition}
The \emph{complex twisted Hilbert modular surface} attached to $B$ and of level $K \subset \Gbf(\A_{\Q,f})$ is the quotient
\begin{displaymath}
	M_K \ := \ \Gbf(\Q) \Big\backslash  \left[X \times \left( \Gbf(\A_{\Q,f }) / K \right) \right]
\end{displaymath}
where  $\Gbf(\Q)$ acts on $X$ via the inclusion $\Gbf(\Q) \subset \Gbf(\R)$, and on $\Gbf(\A_{\Q,f})/K$ by left multiplication. The quotient $M_K$ has a canonical structure of compact complex orbifold of dimension two in general, and is a manifold when $K$ is sufficiently small. 
 \end{definition}
 This can be written in a perhaps more familiar way: fix a set of elements $\{ h_1, \dots , h_t \} \subset \Gbf(\A_{\Q,f})$ such that
 \[ 	
 	\Gbf(\A_{\Q,f}) \ = \ \coprod_i \Gbf(\Q)^+ \, h_i \, K
 \]
 where $\Gbf(\Q)^+ = \{ b \in \Gbf(\Q) \ | \ {}^{\iota} b  \cdot b > 0  \}$, and define $\Gamma_i := \Gbf(\Q)^+ \cap h_i K h_i^{-1}$. Then it is straightforward to check that
 \begin{equation} \label{eqn:ShBUnif}
M_K \ \simeq \  \Gbf(\Q)^+ \Big\backslash \left[ \lie H^2 \ \times \ \Gbf(\A_{\Q,f}) / K \right] \ \simeq \ \coprod_{i} \Gamma_i \backslash \lie H^2;
 \end{equation}
 the first identification follows from the existence of elements of $\Gbf(\Q)$ of negative norm.
 
It will also be useful to explicitly realize the pair $(X, \Gbf)$ as a Shimura datum in the sense of Deligne's axioms, as follows: for a point $z = (z_1, z_2) \in X$ with $z_j = x_j + i y_j$, let 
\begin{equation}\label{eqn:JzDef}
	 J_z \ = \ \left( \frac1{y_1} \begin{pmatrix} x_1 & - x_1^2 - y_1^2 \\ 1 & -x_1  \end{pmatrix} ,  \  \frac1{y_2} \begin{pmatrix} x_2 & - x_2^2 - y_2^2 \\ 1 & -x_2  \end{pmatrix} \right)  \ \in B \otimes_{\Q} \R \ = \ M_2(\R) \oplus M_2(\R). 
\end{equation}
Then the map $h_z \colon \C^{\times} \to \Gbf(\R)$, determined by setting 
\[ h_z(a+bi) = a+bJ_z, \]
is induced from an algebraic map from Deligne's torus $\mathrm{Res}_{\C/\R} \mathbb G_m$ to $\Gbf_{/\R}$, and $X$ is identified with a $\Gbf(\R)$-conjugacy class of such maps. The reflex field of this Shimura datum is $\Q$, and so the theory of canonical models implies that $M_K$ is the set of complex points of a smooth projective surface over $\Q$ whenever $K$ is sufficiently small.

\subsection{Quaternionic Shimura curves}

As in the introduction, let  $\Gbf_{1}=\mathrm{Res}_{F/Q} B_1^{\times}$, so that 
\begin{displaymath}
	\Gbf_{1}(\R)\simeq \mathbb H \times GL_2(\R)
\end{displaymath}
where $\mathbb{H}$ are the Hamiltonians.

Let $\lie H^{\pm} = \{ z \in \C \ | \  \Imag(z) \neq 0 \}$ denote the union of the upper and lower half-planes, which admits an action of $\Gbf_1(\R)$ via the second factor:
\[
	\left( h, (\begin{smallmatrix} a & b \\ c& d\end{smallmatrix}) \right)  \cdot z \ = \ \frac{az + b}{cz + d}.
\]
\begin{definition}
The \emph{complex Shimura curve} attached to $B_1$ of level $K_1 \subset \Gbf_1(\A_{\Q,f})$ is the curve
\[
	S_1 = S_{1,K_1} \ := \ \Gbf_1(\Q) \big\backslash \left[ \lie H^{\pm} \times (\Gbf_1(\A_{\Q,f}) / K_1) \right].
\]
If $K_1$ is sufficiently small, then $S_1$ is a complex compact manifold of dimension one.
\end{definition}

We view this construction in terms of Deligne's axioms for a Shimura variety as follows: for a point $z = x + i y \in \lie H^{\pm}$, consider the map $h_{1,z} \colon \mathrm{Res}_{\C/\R}\mathbb G_m \to (\Gbf_1)_{\R}$ determined, on $\R$-points, by sending $a + bi \in \C^{\times}$ to 
\[
	h_{1,z}(a + bi) \ = \ (1, a + b \cdot j_z )\in \Gbf_1(\R), \qquad \text{where} \qquad j_z = \frac{1}{y}   \begin{pmatrix} x & - x^2 - y^2 \\ 1 & -x  \end{pmatrix}.
\]
In this way, we may identify $\lie H^{\pm}$ with a $\Gbf_1(\R)$-conjugacy class of such maps. The reflex field for the Shimura datum $(\lie H^{\pm}, \Gbf_1)$ in this case is $F$, which we view as a subfield of $\C$ via the first embedding $v_1\colon F \to \C$. The theory of canonical models implies that there is a smooth projective curve $\Scal_1 = \Scal_1(K_1)$ over $F$ such that $\Scal_{1,v_1}(\C) = S_1$, assuming $K_1$ is sufficiently small. 
%
%
%

%


\section{Arithmetic Chow groups and arithmetic Grothendieck-Riemann-Roch}
In this section, which is intended primarily to set up notation, we briefly review the theory of arithmetic Chow groups of Gillet-Soul\'e, as well as the arithmetic Grothendieck-Riemann-Roch theorem; further details can be found in \cite{Gillet-Soule:AIT, Gillet-Soule:ACC, Gillet-Soule:ARR}.

\paragraph{Arithmetic varieties and arithmetic Chow groups}
Let $\kbf$ be a number field with ring of integers $\Ocal_{\kbf}$. In the following, let $S$ denote either $\Spec(\kbf)$, or a Zariski open subscheme of $\Spec(\Ocal_{\kbf})$. 

An \emph{arithmetic variety} $\Xcal$ over $S$ is a regular scheme $\Xcal$ together with a flat and projective map $f \colon \Xcal \to S$. For each complex embedding $\sigma\colon \kbf \to \C$, let $\Xcal_{\sigma} := \Xcal\times_{\sigma} \Spec \C$, and fix, once and for all, a K\"ahler form on the compact complex manifold
\[ \Xcal(\C) := \coprod_{\sigma \colon \kbf \to \C} \Xcal_{\sigma}(\C) ,\]
which we require to be invariant under the natural action of complex conjugation $F_{\infty}$ on $\Xcal(\C)$.

In \cite{Gillet-Soule:AIT}, the authors define the \emph{arithmetic Chow groups} $\Ch{p}(\Xcal)$
attached to $\Xcal$. Elements of $\Ch{p}(\Xcal)$ are represented by pairs 
\[ \widehat Z = (Z, g) \]
where $Z$ is a codimension $p$ cycle on $\Xcal$, and $g$ is a real $(p-1,p-1)$-current on $\Xcal(\C)$ such that the current
\[
	dd^c \, g \ + \ \delta_{Z(\C)} 
\]
is represented by a smooth $(p,p)$ form on $\Xcal(\C)$; here $\delta_{Z(\C)}$ is the current defined by integration over $Z(\C)$. In $\Ch{p}(\Xcal)$, the \emph{rational arithmetic cycles} are deemed to be zero; these are cycles  of the form $(\mathtt{div}(f),-\log|f|^{2}\delta_{W(\C)})$, where  $W$ is a codimension $p-1$ integral subscheme of $\Xcal$, and $f\in \kbf(W)^{\times}$, as well as those of the form $(0, \partial \eta + \overline \partial \eta')$ for currents $\eta$ and $\eta'$ on $\Xcal(\C)$. 

There is a natural map
\[
a \colon	A^{(p-1,p-1)}(\Xcal(\C)) \to \Ch{p}(\Xcal) , \qquad  \eta \mapsto (0,\eta)
\]
where $A^{(p-1,p-1)}(\Xcal(\C))$ is the space of smooth $(p-1,p-1)$ differential forms satisfying $F_{\infty}^{\ast}(g)=(-1)^{p-1}g$; cycles in the image of this map can be thought of as ``purely archimedean".

After tensoring with $\Q$, the arithmetic Chow groups can be equipped with a product structure
\[
	\Ch{p}(\Xcal)_{\Q} \times \Ch{q}(\Xcal)_{\Q} \to \Ch{p+q}(\Xcal)_{\Q}, \qquad \text{denoted by }  (\widehat Z_1, \widehat Z_2) \mapsto \widehat Z_1 \cdot \widehat Z_2
\]
as well as a proper pushforward map
\[
	f_* \colon \Ch{\dim\Xcal}(\Xcal)_{\Q} \to \Ch{1}(S)_{\Q}.
\]
Finally, if $S \subset \Spec(\Ocal_{\kbf})$ is a Zariski open subset, and $N$ is an integer such that $\Spec(\Ocal_{\kbf}[1/N]) \subset S$, then there is an \emph{arithmetic degree map}
\[
	\widehat\deg \colon \Ch{1}(S) \ \to \ \R_N := \R / \oplus_{p | N} \Q \cdot \log p.
\]
More concretely, let $\widehat{Z}$ be an arithmetic divisor on $\Spec\Ocal_{\kbf}[1/N]$; it can be represented by a couple $(\sum_{\mathscr{p}}n_{\mathfrak{p}}(\mathfrak{p}),g)$, where the $\mathfrak{p}$ are the maximal ideals of $\Ocal_{\kbf}[1/N]$ and $g=(g_{\sigma})_{\sigma\colon\kbf\to\C}$ is a tuple of complex numbers, invariant under the action of $F_{\infty}$. Its arithmetic degree is given by the formula
\begin{displaymath}
	\adeg\widehat{Z}=\sum_{\mathfrak{p}}n_{\mathfrak{p}}\log\sharp(\Ocal_{K}/\mathfrak{p})+\frac{1}{2}\sum_{\sigma\colon \kbf\to\C}g_{\sigma}.
\end{displaymath}
Similarly, when $S = \Spec(\kbf)$, one has an arithmetic degree map $\widehat\deg$ valued in the group $\R / \log| \Q^{\times}|$:
\[
	\widehat\deg (a(g)) = \widehat\deg((0,g)) \ = \ \frac12 \sum_{\sigma \colon \kbf \to \C} g_{\sigma}.
\]

By a customary abuse of notation, if $\widehat Z_1$ and $\widehat Z_2$ are cycles of codimension $p$ and $\dim \Xcal - p$ respectively, we will often write $\widehat Z_1 \cdot \widehat Z_2 $ to denote the quantity $\widehat\deg\; f_{*} \left( \widehat Z_1 \cdot \widehat Z_2 \right)$, in either $\R_N$ or $\R/\log|\Q^{\times}|$ as above.

\paragraph{Arithmetic characteristic classes}
A \emph{hermitian vector bundle} on $\Xcal$ is a pair
\[
	\overline \Ecal \ := \ (\Ecal, \la\cdot, \cdot \ra) 
\]
where $\Ecal$ is a vector bundle over $\Xcal$, and $\la\cdot, \cdot \ra$ is a smooth hermitian form on the holomorphic vector bundle $\Ecal_{\C}$ over $\Xcal(\C)$, i.e.\ a family of smooth hermitian forms $(\la\cdot, \cdot \ra_{\sigma})_{\sigma\colon\kbf\to\C}$ where each $\la\cdot, \cdot \ra_{\sigma}$ is a hermitian form on the holomorphic vector bundle $\Ecal_{\sigma}$ on $\Xcal_{\sigma}(\C)$. Furthermore, we require $(\la\cdot, \cdot \ra_{\sigma})$ to be invariant under the natural action of complex conjugation.

Given a hermitian vector bundle, Gillet-Soul\'e construct an element
\[
	\widehat c(\overline \Ecal) \ = \  1 + \widehat c_1(\overline \Ecal) +  \widehat c_2(\overline \Ecal) + \dots \ \in \ \Ch{\bullet}(\Xcal)  \ = \ \oplus \Ch{p}(\Xcal)
\]
called the \emph{total Chern class}, and whose components
\[ \widehat c_p(\overline \Ecal) \in  \Ch{p}(\Xcal)  \]
are \emph{arithmetic Chern classes}.
For example, if $\overline \Lcal $ is a hermitian line bundle and $s$ is any meromorphic section of $\Lcal$, then
\[ \widehat c_1(\overline \Lcal) \ = \ \left( \mathtt{div} (s) , \, - \log || s_{\C} ||^2 \right), \]
and $\widehat c_1(\overline \Ecal) = \widehat c_1(\det \overline \Ecal)$ for a hermitian vector bundle $\overline \Ecal$. 

From the total Chern class, and using the product structure on $\Ch{\bullet}(\Xcal)$, one can build further characteristic classes: for example, by mirroring the usual constructions in differential geometry, we have the \emph{arithmetic Chern character}
\[
	\ach(\ov{\Ecal})=\mathrm{rk}\,\Ecal+\ac_{1}(\ov{\Ecal})+\frac{1}{2} \left(\ac_{1}(\ov{\Ecal})^{2}-2\ac_{2}(\ov{\Ecal}) \right)+..., 
\]
which is additive on orthogonal direct sums of vector bundles, and is multiplicative on tensor products. Similarly, we have the \emph{arithmetic Todd class}
\[
	\atd(\ov{\Ecal})=1+\frac{1}{2}\ac_{1}(\ov{\Ecal})+\frac{1}{12}(\ac_{1}(\ov{\Ecal})^{2}+\ac_{2}(\ov{\Ecal}))+\frac{1}{24}\ac_{1}(\ov{\Ecal})\ac_{2}(\ov{\Ecal})+\ldots
\]
which is multiplicative on orthogonal direct sums. Moreover, all these constructions are functorial with respect to pullbacks.

In this article, we compute top degree arithmetic intersection products of the form $\widehat{c}_{1}(\ov{\Lcal})\cdot\widehat{Z}$. Representing $\widehat{c}_{1}(\Lcal)$ as $(\mathtt{div} (s),-\log\|s_{\C}\|^{2})$ and $\widehat{Z}$ as $(Z,g)$, there is a recurrence formula
\begin{displaymath}
	\widehat{c}_{1}(\ov{\Lcal})\cdot\widehat{Z}=\left(\widehat{Z}\mid\mathtt{div} (s)\right)-\int_{\Xcal(\C)}\log\|s_{\C}\| \omega.
\end{displaymath}
Here, $\left(\widehat{Z}\mid\mathtt{div} (s)\right)$ is the so-called height pairing of Bost-Gillet-Soul\'e \cite{BGS}, and $\omega$ is the smooth differential form representing the current $dd^{c}g+\delta_{Z(\C)}$. The height pairing is bilinear and enjoys of several functoriality properties. As a special case, suppose $\widehat{Z}$ is of the form $\widehat{c}_{n}(\ov{\Ecal})$, and let $j\colon\widetilde{W}\to W$ be a proper and generically finite morphism from an arithmetic variety $\widetilde{W}$ to a closed integral subscheme $W$ of $\Xcal$, of dimension $n$. Then
\begin{displaymath}
	(\widehat{c}_{n}(\ov{\Ecal})\mid W)=\frac{1}{\deg j}\adeg \widehat{c}_{n}(j^{\ast}\ov{\Ecal}),
\end{displaymath}
where $\deg j$ is the generic degree of $j$.

\subsection{The arithmetic Grothendieck-Riemann-Roch theorem}
The arithmetic Grothendieck-Riemann-Roch theorem describes the behaviour of the arithmetic Chern character under push-forwards. First, recall the usual Grothendieck-Riemann-Roch formula in degree 1: if $f\colon\Xcal\to S$ is a projective morphism between regular noetherian schemes, then there is an equality
\[
c_1( \det Rf_* \Ecal) = f_* (\ch( \Ecal) \cdot \td(T_f))^{(1)}  \in \mathrm{CH}^1(S)_{\Q}
\]
where $T_f$ is the tangent complex of $f \colon \Xcal \to S$. For an arithmetic variety $\Xcal\to S$ as above, the equality is trivial: indeed, $\mathrm{CH}^1(S)_{\Q}=0$, because the class group is finite. However, the Grothendieck-Riemann-Roch equality can be upgraded to the level of arithmetic Chow groups, which carry more information and do not vanish. For this, we first need to describe an appropriate metric for the left hand side.

\subsubsection{Holomorphic analytic torsion and the Quillen metric}\label{subsec:hol-an-torsion}
Fix an embedding $\sigma \colon \kbf \to \C$, and let $E = \Ecal_{\sigma}$, viewed as a hermitian holomorphic vector bundle over the compact K\"ahler manifold $X = \Xcal_{\sigma}(\C)$. Attached to $E$ is the \emph{Dolbeault complex}
\begin{displaymath}
	\cdots\longrightarrow A^{0,k}(X,E)\overset{\cpd_{E}}\longrightarrow A^{0,k+1}(X,E)\longrightarrow\cdots
\end{displaymath}
where $A^{0,k}(X,E)$ is the space of $(0,k)$ differential forms valued in $E$. 

The choice of a K\"ahler form $\omega$ on $X$ and the hermitian metric on $E$ determine a pointwise hermitian product, also denoted $\la \cdot, \cdot \ra$  on $A^{(0,k)}(X,E) = A^{(0,k)}(X) \otimes_{C^{\infty}(X)}A^0(X,E)$. This in turn determines a global $L^2$ pairing $h_{L^2}$ on  $A^{0,k}(X,E)$, defined by setting
\begin{displaymath}
	h_{L^2}(s,t)=\int_{X}\langle s_{x},t_{x}\rangle_{x}\frac{\omega^{n}}{n!}(x), \qquad \text{ for } s,t\in A^{0,k}(X,E).
\end{displaymath}


\begin{definition} The $\cpd_{E}$ laplacian, which depends on the choices of metrics as above, is defined to be 
\begin{displaymath}
	\Delta_{\cpd_{E}}^{0,k} \ := \ \cpd_{E}\cpd_{E}^{\ast}+\cpd_{E}^{\ast}\cpd_{E},
\end{displaymath}
where $\cpd_{E}^{\ast}$ is the adjoint to $\cpd_{E}$ with respect to the $L^{2}$ pairing. 
\end{definition}The Laplace operator $\Delta_{\cpd_{E}}^{0,k}$ is an elliptic, positive, essentially self-adjoint differential operator of order 2. It has a discrete spectrum, and the eigenspaces of $\Delta_{\cpd_{E}}^{0,k}$ are finite dimensional. 

Note that Hodge theory provides a canonical isomorphism of finite dimensional vector spaces
\begin{equation}\label{eq:hodge-theory}
	\ker\Delta_{\cpd_{E}}^{0,k}\overset{\sim}{\longrightarrow} H^{0,k}_{\cpd}(X,E)\simeq H^{k}(X,E).
\end{equation}
In particular, the coherent cohomology of $E$ can be realized in spaces of $E$ valued differential forms and thereby inherits the $L^{2}$ metric. 

While 0-eigenspaces (i.e. harmonic forms) are identified with cohomology, the geometric meaning of the strictly positive spectrum is encoded in the \emph{spectral zeta function}, defined as follows. For $s\in\C$ with $\Real(s)\gg 0$, the operator 
\[ \left(\Delta_{\cpd_{E}}^{0,k}\mid_{(\ker\Delta_{\cpd_E}^{0,k})^{\perp}}\right)^{-s} \]
is trace class, and
\begin{displaymath}
	\zeta_{E,k}(s):=\tr\left((\Delta_{\cpd_{E}}^{0,k}\mid_{(\ker\Delta_{\cpd_E}^{0,k})^{\perp}})^{-s}\right)
\end{displaymath}
is holomorphic on its domain. We can equivalently write $\zeta_{E,k}(s)$ as the absolutely convergent sum
\begin{displaymath}
	\zeta_{E,k}(s)=\sum_{\lambda>0}\frac{1}{\lambda^{s}},
\end{displaymath}
where $\lambda$ runs over the discrete strictly positive part of the spectrum of $\Delta_{\cpd_E}^{0,k}$, counted with multiplicity. The theory of the heat operator $e^{-t\Delta_{\cpd_E}^{0,k}}$ shows that the function $\zeta_{E,k}(s)$ admits a meromorphic continuation to the whole complex plane $\C$, and is holomorphic at $s=0$. We then define\footnote{Elsewhere in the literature, the symbol $\det$ is often decorated with a prime symbol, to indicate that 0-eigenvalues were removed.}
\begin{displaymath}
	\det\Delta_{\cpd_{E}}^{0,k}:=\exp(-\zeta_{E,k}^{\prime}(0));
\end{displaymath}
it is a real number, which can informally be thought of as the  ``product over all the strictly positive eigenvalues"
\begin{displaymath}
	``\prod_{\lambda>0}\lambda".
\end{displaymath}
\begin{definition} \label{def:holAnTor}
The \emph{holomorphic analytic torsion} of $\overline{E}$, with respect to the K\"ahler form $\omega$, is defined as
\begin{displaymath}
	T(\overline{E},\omega):=\sum_{k=0}^{n}(-1)^{k}k\log\det\Delta_{\cpd_{E}}^{0,k}=\sum_{k=0}^{n}(-1)^{k+1}k\zeta_{E,k}^{\prime}(0).
\end{displaymath}
\end{definition}
If the choice of K\"ahler metric on $X$ is implicit, we may write $T(X,\ov{E})$. 

We now return to the arithmetic setting. In order to define the Quillen metric on $\det Rf_* \Ecal$, first note that this bundle inherits an $L^2$ metric, denoted $h_{L^2}$, and obtained via Hodge theory, cf.\ \eqref{eq:hodge-theory}.
\begin{definition} The Quillen metric $h_Q = (h_{Q,\sigma})$ is the metric on $\det Rf_* \Ecal$ obtained by rescaling the $L^2$ metric by the analytic torsion. More precisely,  for an embedding $\sigma \colon \kbf \to \C$, it is defined by
\begin{displaymath}
	h_{Q,\sigma}=h_{L^2,\sigma} \cdot \exp(T(\overline{\Ecal}_{\sigma},\omega_{\sigma})).
\end{displaymath}
Let 
\[ \overline{ \det R f_* \Ecal}_Q \ = \  (\det R f_* \Ecal, h_Q ) \]
denote the corresponding hermitian vector bundle.
\end{definition}

\subsubsection{Arithmetic Grothendieck-Riemann-Roch}
We can now state the arithmetic Grothendieck-Riemann-Roch theorem:
\begin{theorem}[Gillet-Soul\'e \cite{Gillet-Soule:ARR}]
There is an equality in $\widehat{\mathsf{CH}}^{1}(S)_{\Q}$
\begin{displaymath}
	\ac_{1}(\overline{\det Rf_{\ast}\Ecal}_{Q})=f_{\ast}(\ach(\overline{\Ecal})\atd(\overline{T}_{f}))^{(1)}-a\left(\int_{\Xcal(\C)}\ch(\Ecal_{\C})\td(T_{\Xcal(\C)})R(T_{\Xcal(\C)})\right).
\end{displaymath}
Here $R$ is the additive genus determined by the formal power series with real coefficients
\begin{displaymath}
	R(x)=\sum_{\substack{m \text{ odd}\\ m\geq 1}}\left(2\zeta'(-m)+\zeta(-m)\sum_{k=1}^{m}\frac{1}{k}\right)\frac{x^{m}}{m!}.
\end{displaymath}
\end{theorem}
We will be interested in arithmetic surfaces and threefolds with $\overline{\Ecal}$ the trivial bundle $\overline{\Ocal}$. Also, instead of the tangent complex, it will be more convenient to work with the cotangent complex $\Omega_{f}$, which, in our applications, will be a vector bundle. Specializing the above theorem, we obtain:
\begin{corollary}\label{cor:ARR}
\begin{enumerate}
	\item If $\Xcal$ is an arithmetic surface, i.e. has Krull dimension 2, then
	\begin{displaymath}
		\ac_{1}(\overline{\det Rf_{\ast}\Ocal}_{Q})=\frac{1}{12}f_{\ast}(\ac_{1}(\overline{\Omega}_{f})^{2})+a\left((2\zeta'(-1)+\zeta(-1))\int_{\Xcal(\C)}c_{1}(\Omega_{f_{\C}})\right)
		.
	\end{displaymath}
	\item If $\Xcal$ is an arithmetic threefold, i.e. has Krull dimension 3, then
	\begin{displaymath}
		\ac_{1}(\overline{\det Rf_{\ast}\Ocal}_{Q})=-\frac{1}{24}f_{\ast}(\ac_{1}(\overline{\Omega}_{f})\ac_{2}(\overline{\Omega}_{f}))-a\left(\frac{1}{2}(2\zeta'(-1)+\zeta(-1))\int_{\Xcal(\C)}c_{1}(\Omega_{f_{\C}})^{2}\right).
	\end{displaymath} 
\end{enumerate}
\end{corollary}
Applying the arithmetic degree map, we have 
\begin{displaymath}
	\widehat{\deg}\ \ac_{1}(\overline{\det Rf_{\ast}\Ocal}_{Q})=\frac{1}{12}\ac_{1}(\overline{\Omega}_{f})^2 \ +\frac{1}{2}[2\zeta'(-1)+\zeta(-1)]\int_{\Xcal(\C)}c_{1}(\Omega_{f_{\C}})^{2}
\end{displaymath}
for an arithmetic surface, and
\begin{displaymath}
	\widehat{\deg}\ \ac_{1}(\overline{\det Rf_{\ast}\Ocal}_{Q})=-\frac{1}{24}\ac_{1}(\overline{\Omega}_{f})\ac_{2}(\overline{\Omega}_{f})-\frac{1}{4}[2\zeta'(-1)+\zeta(-1)]\int_{\Xcal(\C)}c_{1}(\Omega_{f_{\C}})^{2}.
\end{displaymath}
for an arithmetic threefold; here we have abused notation, writing, for instance
\[
	\ac_{1}(\ov{\Omega}_{f})\ac_{2}(\ov{\Omega}_{f})=\widehat{\deg}\ f_{\ast}(\ac_{1}(\ov{\Omega}_{f})\ac_{2}(\ov{\Omega}_{f})).
\]

\section{Integral models of twisted Hilbert modular surfaces} \label{sec:HMSIntModel}
The aim of this section is to review the construction of integral models for twisted Hilbert modular surfaces in terms of PEL data, following \cite{KR-HB}; of crucial importance is the interpretation of $M_K$ as a $GSpin$ Shimura variety. In addition, we introduce and study the automorphic vector bundles of interest in this article.

\paragraph{Notation:} We continue with the setting in the introduction, so that $B$ is a totally indefinite quaternion algebra whose discriminant $D_B$ is a product of split rational primes, and $\Gbf$ is the group \eqref{def:G}.

Fix, once and for all, a maximal order $\Ocal_B$ of $B$ and a sufficiently small compact open subset $K \subset \Gbf(\A_{\Q,f})$.
In addition, fix an integer $N$ such that
\begin{enumerate}[(i)]
	\item $2 D_B \dF$ divides $N$, where $\dF$ is the discriminant of $F$; and
	\item $K_p$ is maximal for all $p \nmid N$, i.e.\
		\[ K_p \ = \ \left( \Ocal_B \otimes_{\Z} \Z_p \right)^{\times} \cap \Gbf(\Q_p)  \qquad \text{for all } p \nmid N .\]
\end{enumerate}

\subsection{$GSpin$ realization.} The assumption on the discriminant of $B$ implies that 
\[ B \ = \ B_0 \otimes_{\Q} F \]
for an indefinite quaternion algebra $B_0$ over $\Q$. More concretely, we may choose
$B_0$ by fixing a rational prime $q$ that is non-split in $F$, and taking
	\[ \mathrm{disc}(B_0) \ = \ 
		\begin{cases} 	p_1 \, \cdots \, p_r, & \text{if }r \text{ is even,} \\
						q \, p_1  \, \cdots \, p_r & \text{if } r \text{ is odd.} 
		\end{cases}
	\]
Abusing notation, we let $b \mapsto {}^{\iota}b$ denote the main involution of $B_0$, so that the involution on $B$ is given by
\[ {}^{\iota} (b_0 \otimes a)  \ = \ ( {}^{\iota} b_0 \otimes a) .\] With this presentation, the automorphism of $B$ induced by the non-trivial Galois automorphism $'\colon F \to F$ is
\[ \sigma \colon B \to B, \qquad {}^{\sigma}(b_0 \otimes a) \ = \ b_0 \otimes a'. \]
Define a 4 dimensional $\Q$-vector space 
\[  V_0 \ := \ \{ b \in B \ | \ {}^{\sigma}b \ = \ {}^{\iota} b \} \]
with quadratic form 
\[  Q_0(v) = \Nrd(v). \] 
Here $\Nrd$ stands for the reduced norm of $B$, which takes rational values on $V_{0}$, essentially by definition of $V_0$. For concreteness, if $B_0 = (\frac{a,b}{\Q})$ for some $a,b \in \Q^{\times}$, then there is a basis for $V_0$ such that $ Q_0$ takes the diagonal form
	\[ \begin{pmatrix} 1 & & & \\ & - \,  \dF \, a & & \\ & &  - \, \dF \, b & \\ &&& \dF \,  a \, b \end{pmatrix}.\]
Since $B_0$ is indefinite, at least one of $a$ or $b$ is positive, and so the signature of $ V_0$ is $(2,2)$, and  moreover 
\begin{equation}\label{eqn:Q0DiagForm}
	\det  Q_0 \ \equiv \ \dF \  \text{in} \ \Q^{\times} / \Q^{\times,2}.
\end{equation}

Define now a non-commutative algebra $C = B\langle v_0 \rangle $ extending $B$, by adjoining the formal symbol $v_0$ subject to the relations
\[ v_0^2 \ = \  1 \qquad \text{and} \qquad v_0 \cdot b = {}^{\sigma} b \cdot v_0. \]
Extend the involution $\iota$ to $C$ by setting ${}^{\iota}v_0 = v_0$. Finally, introduce the $\Q$-vector space 
\[ V \ :=  \  V_0 \cdot v_0 \ \subset C, \]
equipped with the quadratic form $Q(v) = Q_{0}(\tilde v) $ whenever $v = \tilde v \cdot v_0$ for some $\tilde v\in V_{0}$. One easily checks the following facts on the Clifford algebra $C(V)$ of the quadratic space $V$:

\begin{lemma} The inclusion $V \subset C$ induces isomorphisms $C(V) \simeq C$ and $C^+( V) \simeq B$. In particular, $GSpin(V) \simeq \Gbf$ as algebraic groups over $\Q$. \qed
\end{lemma}

\subsection{Lattices and symplectic structure.} 
Consider the lattice
\[ L \ := \ \Ocal_B \cdot v_0 \ \cap \ V; \]
by construction, the quadratic form on $V$ described in the previous section restricts to an integral form on $L$, and so we may view $L$ as a quadratic lattice of signature $(2,2)$. Note moreover that $C^+(L) = \Ocal_B$.
%

The inclusion $C(L) \hookrightarrow C(V)$ induces an action of $C(L)$ on $C(V)$ by left multiplication. In addition, the ring of integers $\OF$ embeds into $B = B_0 \otimes F$ and hence into $C(V)$, so there is an action of $\OF$ on $C(V)$ by right multiplication. The two aforementioned actions commute, thus endowing $C(V)$ with the structure of a $C(L) \otimes_{\Z} \OF$-module.

Fix an element 
\begin{equation} \label{eqn:thetadef} \theta \in B_0 \cap C^+(L)\end{equation}
 with ${}^{\iota} \theta = - \theta $ and $\theta^2 = - D_{B}$. Then the involution of $C(V)$ defined by
\[ c \mapsto c^* = \theta \cdot {}^{\iota} c \cdot \theta^{-1} \]
is positive, and the alternating form
\begin{equation}\label{eqn:UZPol}
	\langle c_1, c_2 \rangle \ := \ \Trd(    c_2^{\iota} \cdot \theta  \cdot c_2)
\end{equation}
is non-degenerate. Note also that
\[ \langle (c \otimes a) \cdot c_1 , \ c_2 \rangle \ = \ \langle c_1,  \ (  c^* \otimes a) \cdot c_2 \rangle  \qquad \text{ for all } c\otimes a \in C(L) \otimes \OF. \]
This motivates extending the involution $\ast$ to $C(L)\otimes\OF$ by acting trivially on the second factor. 


\paragraph{Notation:} It will be useful to write $U$ for the symplectic $C(L) \otimes \OF$-module given by the couple $(C(V),\langle\cdot,\cdot\rangle)$, and $\langle\cdot,\cdot\rangle_{U}$ for its symplectic form.

\subsection{A PEL moduli problem}  Recall that we had fixed  a compact open subgroup $K \subset \Gbf(\A_{\Q,f})$, which we assume is sufficiently small; in addition, we fixed an integer $N$ such that $K$ is maximal outside $N$, and that is divisible by $2 \, \dF \, \mathrm{disc}(B)$.

Following \cite[\S 12]{KR-HB}, consider the moduli functor  $\Mcal_K$ over $\Spec(\Z[1/N])$ that attaches to a scheme $S \to \Spec \Z[1/N]$ the set of $N$-primary isogeny classes of tuples
\[ \Mcal_K(S) \ = \ \{ \underline A =  (A, i, \lambda, \overline \eta) \} / \simeq \] where
	\begin{enumerate}[(i)]
		\item $A$ is an abelian scheme of relative dimension 8 over $S$;
		\item $\lambda \colon A \to A^{\vee}$ is a $\Z[1/N]^{\times}$-class of principal polarizations;
		\item $i \colon C(L) \otimes_{\Z} \OF \to \End_S(A) \otimes_{\Z} \Z[1/N]$ is a map such that:
			\begin{itemize}
				\item	$ \det( i(c \otimes a) |_{\Lie A} ) \ = \ \Nrd(c)^2 \cdot N_{F/\Q}(a)^4 \ \in \ \mathcal O_S$ ;
				\item 
				$i(c \otimes a)^* \ = \ i( c^* \otimes a)$ for the Rosati involution $*$ induced by $\lambda$; since we extended $\ast$ to $C(L)\otimes\OF$ by acting trivially on the second factor, the condition is more compactly stated as $i(b)^{\ast}=i(b^{\ast})$ for $b\in C(L)\otimes\OF$.
			\end{itemize}
		\item $\overline \eta$ is a $K_N$-equivalence class of $C(L) \otimes \OF$-linear isomorphisms
				\[ \eta \colon \Ta_N(A) \otimes_{\Z} \Q  \ \isomto \ U(\A_N) \]
			such that the pullback of the symplectic form $\langle \cdot, \cdot \rangle_{U}$ under $\eta$ is an $\A_N^{\times}$-multiple of the pairing induced by $\lambda$. 
			
			Here
			\[ K_N = \prod_{p\mid N}K_{p} \ \qquad \text{ and } \qquad \A_{N}=\prod_{p\mid N}\Q_{p}, \]
			and $\Ta_{N}(A)=\prod_{p\mid N}\Ta_{p}(A) $ is the product of the $p$-adic Tate modules $\Ta_p(A)$ of $A$. 
	\end{enumerate}
An \emph{$N$-primary isogeny} between tuples $\underline A$ and $\underline A'$ is a $C(L) \otimes \OF$-equivariant isogeny 
\[ \varphi \colon A \to A' \]
 such that $(i)$ the prime factors of  $\deg \varphi$  divide  $N$; and $(ii)$ the pullbacks of $\lambda_{A'}$ (resp.\ $\eta_{A'})$ lie in the same equivalence class as $\lambda_A$ (resp. $\eta_A$). 

\begin{proposition}[{\cite[\S 12]{KR-HB}}] \label{prop:MoKUnif} If $K$ is sufficiently small, then the moduli problem $\Mcal_K$ is representable by a smooth and projective scheme over $\Spec(\Z[1/N])$, and
\[ \Mcal_K(\C) \ \simeq \ M_K \ = \  \Gbf(\Q) \ \backslash X \times \Gbf(\A_{\Q,f}) / K, \]
where $X = \{ (z_1, z_2) \in \C^2 \ | \ \Imag(z_1) \cdot \Imag(z_2) > 0  \}$.
\end{proposition}
For the applications we have in mind, it will be more convenient to consider the base change to $\OF[1/N]$:
\begin{definition} \label{def:MKdef}
Let $\Scal = \Spec \OF[1/N]$, and consider the base change
	\[ \Mcal_{K ,\Scal}\ :=\ \Mcal_K \ \times_{\Z[1/N]}  \ \Scal. \]
Because $\dF $ divides $N$, the scheme $\Mcal_{K,\Scal}$ is again projective and smooth over $\Scal$; in particular, it is an arithmetic variety over $\Scal$.
\end{definition}


\subsection{Automorphic vector bundles}\label{subsec:automorphic-vector}
In the following sections, we fix a sufficiently small $K \subset \Gbf(\A_{\Q,f})$, and suppress it from the notation, writing e.g. $\Mcal = \Mcal_K$, $\Mcal_{\Scal} = \Mcal_{K,\Scal}$, etc.
\subsubsection{The cotangent bundle and the Petersson metric.}\label{subsubsec:Petersson-metric}

Since $\Mcal_{\Scal}\to \Scal=\Spec\OF[1/N]$ is smooth, the sheaf of relative K\"ahler differentials $\Omega_{\Mcal_\Scal/\Scal}$ is locally free of rank $2$. On $\Mcal_{\Scal}(\C)$, the attached analytic vector bundle is described as follows. First of all, we decompose $\Mcal_{\Scal}(\C)$ into connected components: for this, let $v_1, v_2 \colon F \to \C$ denote the two complex embeddings of $F$ into $\C$, and write
	\[ \Mcal_{\Scal}(\C) \ = \ \Mcal_{v_1}(\C) \ \coprod \ \Mcal_{v_2}(\C), \]
	where, by Proposition \ref{prop:MoKUnif} and \eqref{eqn:ShBUnif}, we have
	\[ \Mcal_{v_j}(\C) \ \simeq  \ \Gbf(\Q) \backslash X \times \Gbf(\A_{\Q,f}) / K \  \simeq \  \coprod_{i} \Gamma_i \backslash \lie H^2 \]
for $j=1,2$ and some arithmetic subgroups $\Gamma_i \subset \Gbf(\Q)^+$.

As an analytic coherent sheaf, the restriction of $\Omega_{\Mcal_{\Scal}(\C)}$ to a component $\Gamma_{i}\backslash\lie H^{2}$ is canonically isomorphic to the sheaf of $\Gamma_{i}$ invariant holomorphic differentials on $\lie H^{2}$. In the coordinates $(z_{1},z_{2})$ of $\lie H^{2}$, a local section $\theta$ of $\Omega_{\Mcal_{\Scal}(\C)}$ can be written
\begin{displaymath}
	\theta=f(z_{1},z_{2})dz_{1}+g(z_{1},z_{2})dz_{2},
\end{displaymath}
where $f$ and $g$ are holomorphic functions on some $\Gamma_{i}$ invariant analytic open subset of $\lie H^{2}$. Define a  hermitian metric $||\cdot ||_P$ on $\Omega_{\lie H^2}$ by setting\footnote{The normalization appearing here has been chosen to coincide with that of \cite{BBK}.}, at a point $\zbf= (z_1, z_2) \in \lie H^2$, 
\begin{displaymath}
	|| dz_{1} ||_{P,\zbf}^{2}\ = \ 16 \, \pi ^2  \, \Imag(z_1)^{2},\quad || dz_{2} ||_{P,\zbf}^{2}\ = \ 16 \, \pi^2 \, \Imag(z_2)^{2},\quad \langle dz_{1},dz_{2}\rangle_{P,\zbf}=0.
\end{displaymath}

Since this metric is easily seen to be $\Gamma_i$-invariant, it descends to a smooth hermitian metric $||\cdot||_P$ on the vector bundle $\Omega_{\Mcal_{\Scal}(\C)}$ that we call the \emph{Petersson metric}, and we obtain a hermitian vector bundle
\[ \overline \Omega_{\Mcal_{\Scal}/\Scal} \ := \ (\Omega_{\Mcal_{\Scal}/\Scal},\|\cdot\|_{P}). \]

\subsubsection{The universal Lie algebra and Kodaira-Spencer morphisms.} The moduli scheme $\Mcal_{\Scal}$ admits a universal abelian scheme, with additional structure, that we denote
\[ \underline{\Abf} \ = \  \left( \Abf, i_{\Abf} , \lambda_{\Abf} , \overline \eta_{\Abf} \right). \]
If $T_{\Abf/\Mcal_{\Scal}}$ is the relative tangent bundle and $e\colon\Mcal_{\Scal}\to\Abf$ the zero section, define the \emph{universal relative Lie algebra} to be
\begin{displaymath}
	\Lie(\Abf):=e^{\ast} T_{\Abf/\Mcal_{\Scal}},
\end{displaymath}
a locally free $\mathcal{O}_{\Mcal_{\Scal}}$-module of rank $8$. 

This module admits a decomposition into two rank four pieces, as follows. The action $i_{\Abf}$ induces a morphism
	\[ \Lie(i_{\Abf}) \colon C(L) \otimes_{\Z} \OF  \ \to \ \End_{\mathcal O_{\Mcal_{\Scal}}}( \Lie(\Abf) ); \]
we will often employ the abbreviation $ i = \Lie(i_{\Abf})$ and hope the meaning remains clear despite the abuse of notation. Restricting this action to elements of the form $1 \otimes a \in C(L) \otimes \OF$ yields a $\Z/2$ grading 
	\[ \Lie(\Abf)  = \Lie(\Abf)_1 \oplus \Lie(\Abf)_2, \]
 where
	\[ \Lie(\Abf)_1 \ := \ \{x \in \Lie(\Abf) \ |  \ i \left(1 \otimes a\right) (x) \ = \ \tau(a) \cdot x \ \text{ for all } a \in \OF \} \]
and
	\[ \Lie(\Abf)_2 \ := \ \{x \in \Lie(\Abf) \ | \ i \left(1 \otimes a\right) (x) \ = \ \tau(a') \cdot x \ \text{ for all } a \in \OF \}; \]
here $\tau \colon \OF[1/N] \to \mathcal O_{\Mcal_{\Scal}} $ is the structure map and  $a \mapsto a'$ is the non-trivial automorphism of $F$. In this splitting, we are implicitly using the assumption $2 \dF\mid N$, hence $2 \dF \in \mathcal{O}_{\Mcal_{\Scal}}^{\times}$.

The Kodaira-Spencer map, cf.\ \cite[Theorem 6.4.1.1]{Lan}, gives rise to a canonical exact sequence
	\begin{equation*}
		\begin{CD}
			0 @>>> \ker(KS) @>>> 	\Lie(\Abf)^{\vee} \, \otimes_{\mathcal O_{\Mcal_{\Scal}} }  \Lie(\Abf^t)^{\vee}  @>KS>> \Omega_{\Mcal_{\Scal}/\Scal} @>>> 0
		\end{CD}
	\end{equation*}
	where the superscript ${}^{\vee}$ denotes the $\mathcal O_{\Mcal_{\Scal}}$-linear dual module, $\Abf^t$ is the dual abelian scheme, and the kernel $\ker(KS)$ is the submodule generated by
	\begin{equation*}
		\left\{ \lambda^{\vee} (x) \otimes y - \lambda^{\vee}(y) \otimes x, \  i(b)^{\vee} z \otimes y - z \otimes i(b)^{t,\vee} y  \ | \ x,y \in \Lie(\Abf^{t})^{\vee}, \ z \in \Lie(\Abf)^{\vee}\text{ and } b \in C(L) \otimes \OF \right\} .
	\end{equation*}
Using the polarization $\lambda = \lambda_{\Abf}$ to identify $\Lie(\Abf)$ with $\Lie(\Abf^t)$, together with the compatibility of the Rosatti involution and the involution $\ast$ on $C(L)\otimes\OF$, we obtain a canonical isomorphism
{ \small	
\begin{equation} \label{eqn:KSbig} \Omega_{\Mcal_{\Scal}/\Scal} \ \simeq \ \left\{ \varphi \in \Hom \left(\Lie(\Abf), \Lie(\Abf)^{\vee} \right)  \ | \ \varphi \text{ is symmetric and } \varphi \circ i(b) \ = \ i(b^*)^{\vee} \circ \varphi \ \text{ for all } b \in C(L) \otimes \OF \ \right\}  \end{equation}}

Fix an element $\varpi \in \OF$ with $\varpi' = - \varpi$ and $\varpi^2 = \dF$, and note that $2 \tau(\varpi) \in \mathcal O_{\Mcal}^{\times}$. If $x \in \Lie(\Abf)_1 $, $y \in \Lie(\Abf)_2$ and $\varphi$ is as in \eqref{eqn:KSbig}, then
\begin{align*}  \tau(\varpi) \cdot \varphi(x)(y)  \ = \  \varphi \left(i(1\otimes \varpi) x\right)(y)  \ = \  \varphi \left( x\right) \left( i(1 \otimes \varpi ) y\right) \ = \ - \tau(\varpi) \cdot \varphi(x)(y),  \end{align*} 
	which in turn implies $\varphi(x)(y) = 0$.
From this, it follows that the morphisms in \eqref{eqn:KSbig} preserve the $\Z/2$ grading, and so $\Omega_{\Mcal_{\Scal}/\Scal}$ decomposes into a sum of line bundles
	\begin{equation} \label{eqn:KSmap}
	\Lcal_1 \oplus \Lcal_2 \ \simeq \ \Omega_{\Mcal_{\Scal}/\Scal},
	\end{equation}
where 
	{\small \[ \Lcal_j \ = \ \left\{ \varphi \in \Hom \left(\Lie(\Abf)_j, \Lie(\Abf)_j^{\vee} \right)  \ | \ \varphi \text{ is symmetric and }\varphi \circ i(b) \ = \ i(b^*)^{\vee} \circ \varphi \ \text{ for all } b \in C(L)   \right\}; \]}
here we observe that $\Lie(\Abf)_j$ is a $C(L) \otimes \mathcal O_{\Mcal}$-module of rank one, and in particular the symmetry condition in \eqref{eqn:KSbig} follows automatically from the equivariance.

The next statement describes the compatibility of the decomposition \eqref{eqn:KSmap} with the pointwise Petersson metric $||\cdot||_{P}$.

\begin{proposition} \label{prop:OmegaCpxDecom}
(i) Fix a complex embedding $v\colon \OF[1/N] \hookrightarrow \C$. The decomposition
\[ \mathcal L_{1,v} \ \oplus \ \mathcal L_{2,v} \ \simeq \Omega_{\Mcal_{v}(\C)}  \]
induced by \eqref{eqn:KSmap} is orthogonal with respect to the metric $|| \cdot ||_{P} $. 

(ii) For $i = 1,2$, the square of the first Chern form of the line bundle $\Lcal_{i,v}$ with the induced Petersson metric vanishes:
\[ c_1( \ov{\Lcal_{i,v}}  )^2 = 0.\]

(iii) There is an equality of Chern-Weil differential forms 
\[ c_1(\ov{\Omega_{\Mcal_{v}(\C)}})^2 \ =  \ 2 c_2(\ov{\Omega_{\Mcal_{v}(\C)}}). \]

\begin{proof} We work over a single component in the decomposition
\[ \Mcal_{v} (\C) \ = \ \coprod_i \Gamma_i \backslash \lie H^2. \]
Let $\mathbf A_{/ \lie H^2}$ denote the pullback of the universal abelian scheme $\mathbf A$ via the map $\lie H^2 \to \Gamma_i \backslash \lie H ^2 \subset \Mcal_{v}(\C) $. The de Rham homology $H^{dR}_{1}(\mathbf A_{ / \lie H^2})$, with its $C(L) \otimes \OF$ action and polarization, can be identified with the constant sheaf $\underline{U_{\C}}$, where
\[ U_{\C} \ = \ B_{\C }\ \oplus \ B _{\C} \cdot v_0 \ \simeq \ \left( M_2(\C) \oplus M_2(\C) \right) \ \oplus \ \left( M_2(\C) \oplus M_2(\C) \right) \cdot v_0; \]
here we used the isomorphism $B_{\R} \simeq M_2(\R) \oplus M_2(\R)$. 
Define
\[ \mathcal C_1 := \left\{ c \in U_{\C} \ |  \ i(1 \otimes a) c  = v(a) \cdot c \ \text{for all } a \in \OF \right\} \  = \ \{ (X,0) + (0,Y) \cdot v_0 \ | \ X,Y \in M_2(\C) \} \]
and
\[ \mathcal C_2 := \left\{ c \in U_{\C} \ |  \ i(1 \otimes a) c  = v(a') \cdot c \  \text{for all } a \in \OF \right\} \  = \ \{ (0,X) + (Y,0) \cdot v_0 \ | \ X,Y \in M_2(\C) \}.\]

Given a point $\zbf  = (z_1,z_2)\in \lie H^2$, let $A_{\zbf}$ denote the corresponding abelian variety, which is determined by complex structure \eqref{eqn:JzDef}. The (dual of the) Hodge exact sequence then reads
	\[ \begin{CD} 0 @>>> U_{\C}^{(0,-1)} \ = \ \Lie(A_{\zbf}^t)^{\vee} @>>> U_{\C} @>>> \Lie(A_{\zbf}) \ = \ U^{(-1,0)}_{\C} @>>> 0 \end{CD} \]
Concretely, we may identify the graded component
\[ \Lie( A_{\zbf})_1 \ = \ \{ c \in \mathcal C_1 \ | \ i \, c \ = c \cdot J_{\mathbf z} \}  \]
where 
\[ J_{\zbf} \ = \ \left( j_{z_1}, \ j_{z_2} \right) \ = \ \left(\frac{1}{y_1} \begin{pmatrix} x_1 & - x_1^2 - y_1^2 \\ 1 & - x_1 \end{pmatrix}, \  \frac{1}{y_2} \begin{pmatrix} x_2 & - x_2^2 - y_2^2 \\ 1 & - x_2 \end{pmatrix} \right) \in  B^{\times}(\R) \simeq  GL_2(\R) \times GL_2(\R).  \] 
On the other hand, for an element $c = (X,0) + (0,Y) v_0 \in \mathcal C_1$ 
	\begin{align*}
		c \cdot J_{\zbf} \ =& \ [(X, \ 0) + (0, \ Y) v_0] \cdot (j_{z_1}, \ j_{z_2}) \\
		=& \ \left( X \cdot j_{z_1} , \  0  \right) \ + \ (0, \ Y) \cdot (j_{z_2},    \ j_{z_1}) \cdot v_0 \\
		=&  \  \left( X \cdot j_{z_1} , \  0  \right) \ + \ (0, \ Y\cdot    j_{z_1}) \cdot v_0 ,
	\end{align*}
so
\[ (\Lie(\Abf_{ / \lie H^2})_1)_{\zbf} \ = \  \Lie(A_{\zbf})_1 \ = \ \left\{ c = (X,0)  + (0,Y) v_0 \in \mathcal C_1 \ | \  X\cdot (i - j_{z_1} ) \ = \ Y\cdot(i - j_{z_1}) \ = \ 0 \right\}. \]
From this description, it is easy to verify that the four nowhere vanishing sections
\[
\left( \begin{pmatrix} 1 & - \overline{z_1} \\ 0 & 0 \end{pmatrix}, \ 0 \right), \  \ \  \left( \begin{pmatrix} 0 & 0 \\ 1 & - \overline{z_1}  \end{pmatrix}, \ 0 \right),   \ \  \left( 0 , \ \begin{pmatrix} 1 & - \overline{z_1} \\ 0 & 0 \end{pmatrix}  \right)\cdot v_0, \ \ \  \left(0,  \begin{pmatrix} 0& 0 \\ 1 & - \overline{z_1}  \end{pmatrix} \right) \cdot v_0
\]
form a basis of $(\Lie(\Abf_{\ \lie H^2})_1)_{\zbf}$, where $\zbf = (z_1, z_2)$, and moreover the first element
\[ \sigma^{(1)}_{\mathbf z} \ := \ \left( \begin{pmatrix} 1 & - \overline{z_1} \\ 0 & 0 \end{pmatrix}, \ 0 \right)  \]
generates  $ \Lie(\Abf_{ / \lie H^2})_{1} $ as a $C(L)\otimes \Ocal_{\lie H^2}$-module. Similarly, the section
\[ \sigma^{(2)}_{\mathbf z} \ = \ \left( 0, \  \begin{pmatrix} 1 & - \overline{z_2} \\ 0 & 0 \end{pmatrix} \right)  \]
is a $C(L) \otimes \Ocal_{\lie H^2}$-module generator of $\Lie(\Abf_{ / \lie H^2})_{2}$. 
Moreover, note that at the point $\zbf$, a section of the line bundle
	{\small \[ \Lcal_j \ = \ \left\{ \varphi \in \Hom \left(\Lie(\Abf)_j, \Lie(\Abf)_j^{\vee} \right)  \ | \ \varphi \text{ is symmetric and }\varphi \circ i(b) \ = \ i(b^*)^{\vee} \circ \varphi \ \text{ for all } b \in C(L)   \right\} \]}
	is determined by its value at $\sigma^{(j)}_{\zbf}$; we may therefore define a nowhere vanishing section $\Phi^{(j)}_{\zbf}$ of $\Lcal_j$ over $\lie H^2$ by specifying the normalization
\[ \Phi^{j}_{\zbf} (\sigma^{(j)}_{\mathbf z})(\sigma^{(j)}_{\mathbf z}) = 1.\]
From this, one can verify that the Kodaira-Spencer map identifies
\[ KS(\Phi^{(j)}_{\zbf})  = \frac{1}{(D_B)^{1/2} \, (y_j)^{2}} \, dz_j \]
and so 
\[ \Lcal_j  \simeq \mathcal O_{\lie H^2} \cdot d z_j \]
under the Kodaira-Spencer map.

The first claim of the lemma now follows immediately from the definition of the metric $||\cdot ||_{P}$, as $dz_1$ and $dz_2$ were declared to be orthogonal. 

We note in passing that the above identifications allow us express the norm of a section $\varphi$ of $\Lcal_j$ as
	\begin{equation} \label{eqn:metricLj}
			|| \varphi ||^2_{\zbf} \ = \ \frac{16 \pi^2}{D_B} \cdot y_j^{-2 } \cdot | \varphi( \sigma^{(j)}_{\zbf})(\sigma^{(j)}_{\zbf}) |^2 
	\end{equation}
We also find for $j=1,2$
\[ c_1(\Lcal_j,||\cdot||_{P}) \ = \ - dd^c \log|| dz_j||^2_{P}  \ = \  -   d d^c  \log (y_j^2) \ = \ -  \frac{1}{4 \pi i } \frac{ dz_j \wedge d \overline{z_j}}{y_j^2} \]
and claim (ii) follows immediately. 

Finally, the fact that the Chern-Weil form attached to the total Chern class is multiplicative for orthogonal direct sums of line bundles implies that 
\[ c_2(\ov{\Omega_{\Mcal_{v}(\C)}}) \ = \ c_1(\ov{\Lcal_{1,v}} ) \, c_1(\ov{\Lcal_{2,v}}) \]
while part (ii) implies
\begin{displaymath}
	\begin{split}
		c_1(\ov{\Omega_{\Mcal_{v}(\C)}})^2 \ &= \ \left(c_1(\ov{\Lcal_{1,v}}) \ + \ c_1( \ov{\Lcal_{2,v}})\right)^2\\
		 &= \ 2 \cdot c_1(\ov{\Lcal_{1,v}}) \, c_1(\ov{\Lcal_{2,v}}).
	\end{split}
\end{displaymath}
 This proves (iii) and concludes the proof.
\end{proof}
\end{proposition}

\subsubsection{The tautological bundle} \label{sec:MTaut}
Recall that the lattice $L$ embeds naturally into $C(L)$, and it is straightforward to check that the image is
\[ L \ =\ \left\{ x \in C^-(L) \ | \ {}^{\iota} x = x  \right\}  \ = \ \left\{ x \in C(L) \ | \ {}^{\iota} x = x , \ x \delta = - \delta x \right\}, \]
where we recall $\delta \in Z(C^+(L)) = Z(\Ocal_B) = \Ocal_F $ is a fixed element satisfying $\delta' = - \delta$ and $\delta^2 = \dF$.

The group $\Gbf(\Z) = GSpin(L)\subset C(L)^{\times}$ acts on $L$ via conjugation; as this action preserves the quadratic form $Q(x) = x^2$ on $L$, this gives rise to a variation of polarized Hodge structures $ \Lbb_{/\Mcal_{\Scal}(\C)}$ of weight zero over $\Mcal(\C)$ . Concretely, working over a component $\Gamma \backslash \lie H^2 \subset \Mcal_{\Scal}(\C)$, the fibre $\Lbb_{\zbf}$ at point $[\mathbf z] = \Gamma \backslash \lie H^2$ decomposes as
\begin{equation} \label{eqn:LHdgDecomp} \Lbb_{\zbf} \ = \ \Lbb^{(1,-1)}_{\zbf} \ \oplus \ \Lbb^{(0,0)}_{\zbf} \ \oplus \ \Lbb^{(-1,1)}_{\zbf} 
\end{equation}
with 
\begin{equation} \label{eqn:omegaTautBasis} \Lbb_{\zbf}^{(p,q)} \ = \ \{ x \in L \otimes_{\Z} \C \ | \ (a + b J_{\zbf}) \cdot x \cdot (a + b J_{\zbf})^{-1} = (a+ib)^p (a-ib)^q \cdot x  \ \text{ for all } a + ib \in \C   \} 
\end{equation}
where $J_{\zbf} \in B(\R) \subset C(L)\otimes_{\Z} \R $ is as in \eqref{eqn:JzDef}. A straightforward computation reveals that $\Fil^{1}(\Lbb_{\zbf}) = \Lbb_{\zbf}^{(1,-1)}$ is the isotropic line 
\[
	 \Lbb_{\zbf}^{(1,-1)} \ = \ \mathrm{span}_{\C} \left\{  \left( (\begin{smallmatrix} z_1 & -z_1z_2 \\ -1 & - z_2 \end{smallmatrix}) , \, (\begin{smallmatrix} -z_2 & z_1z_2 \\ 1 & z_1 \end{smallmatrix})  \right)\cdot v_0 \right\}, \qquad \text{where } \zbf = (z_1, z_2).
\]
Following the normalization of \cite{Hor}, we may define a metric on the line bundle $\Fil^{1}( \Lbb_{/\Mcal(\C)})$ by setting
\begin{equation} \label{eqn:FritzMetric}
	|| \ell ||^2_{\zbf}  \ := \ - \frac14 \ e^{-\gamma - \log 2 \pi} \left( \ell \cdot \overline \ell + \overline \ell \cdot \ell \right),
\end{equation}
where $\gamma = - \Gamma'(1)$ is the Euler-Mascheroni constant.

Our next aim is to extend the bundle $ \Lbb_{/\Mcal_{\Scal}(\C)}$ to a bundle $ \Lbb$ over the integral model $\Mcal_{\Scal}$. Let $U_{\Z} = C(L)$ endowed with an action on $C(L) \otimes_{\Z} \OF$ and symplectic pairing $\langle\cdot,\cdot\rangle$ as in \eqref{eqn:UZPol}. An element $x \in L$ determines an endomorphism $\varphi_x \in \End(U_{\Z})$ by right-multiplication, defined by the formula
\[ \varphi_x(y) = y \cdot x \qquad \text{ for all } y \in U_{\Z}, \]
and is such that
\[\left(\varphi_x \circ i(c \otimes a) \right)(y) \ =\  c \cdot y \cdot a \cdot x \ = \ c \cdot y \cdot x \cdot a' \ = \ (i(c \otimes a' ) \circ \varphi_x)(y) \text{ for all } c \otimes a \in C(L) \otimes \OF 
\]
Moreover, 
\[
(\varphi_x)^* = \varphi_{{}^{\iota} x}  = \varphi_x, \] 
i.e.\ $\varphi$ is self-adjoint with respect to $\langle \cdot , \cdot \rangle$. Conversely, any endomorphism satisfying these two conditions is necessarily given by right-multiplication  by an element of $L$, and so we obtain an isomorphism
\[ 
L \isomto \left\{ \varphi \in \End(U_{\Z}) \ | \ \varphi^* = \varphi  \text{ and } \varphi \circ i(c \otimes a) = i(c \otimes a') \circ \varphi  \right\}.
\]
This isomorphism is an isometry, where the right hand side is endowed with the quadratic form $Q(\varphi) = \varphi \circ \varphi^* = \varphi^2$, and moreover is $GSpin(L)$-equivariant. This last fact implies that $L$ inherits a variation of Hodge structures via restricting the natural structure of weight 0 on $End(U_{\C})$, which can be easily checked to coincide with \eqref{eqn:LHdgDecomp}. 

These considerations, together with the observation that the variation of Hodge structures $U_{\Z}$ coincides with the homology of the universal abelian variety over $\Mcal_{\Scal}(\C)$, suggest the following integral extension.
Let $\underline \Abf = (\Abf, i, \lambda, [\eta]) / \Mcal_{\Scal}$ denote the universal object over the full integral model $\Mcal_{\Scal}$, and let
\[ \mathbb U \ := \ H_1^{dR}(\Abf ) =( R^1 \pi_* \Omega_{\Abf/\Mcal_{\Scal}})^{\vee} \]
denote the algebraic de Rham homology of  $\pi \colon \Abf \to \Mcal_{\Scal}$ (i.e. the $\Ocal_{\Mcal_{\Scal}}$-dual of the relative algebraic de Rham cohomology of $\Abf/\Mcal_{\Scal}$), which, along with the induced $C(L) \otimes \OF$ action and polarization,  comes equipped with a filtration
\[ 
  \Fil^{0}(\bbU) \ \subset \ \Fil^{-1}(\bbU) = \bbU;
\]
as before, we may identify
\[
 \Fil^{0}(\bbU) \ = \ \Lie(\Abf^t)^{\vee} \qquad \text{and} \qquad Gr^{-1}(\bbU) = \bbU /  \Fil^{0}(\bbU) \ = \ \Lie(\Abf). 
\]
Define
\begin{equation} \label{eqn:LDef}
	\Lbb \ := \ \left\{ \varphi \in \End(\bbU) \ | \ \varphi^* = \varphi \text{ and } i(c \otimes a) \circ \varphi = \varphi \circ i(c \otimes a') \text{ for all } c\otimes a \in C(L) \otimes_{\Z} \OF  \right\},
\end{equation}
where $\varphi^*$ is the image of $\varphi$ under the Rosati involution induced by $\lambda$. As a subsheaf of $\End(\bbU)$, the bundle $\Lbb$ inherits a filtration
\[ 
\Fil^1(\Lbb) \ \subset \  \Fil^0(\Lbb) \ \subset \Fil^{-1}(\Lbb) = \Lbb . 
\]

\begin{definition}
The \emph{tautological sheaf} is the coherent sheaf
\[ \omega^{\taut} \ := \ \Fil^{1}( \Lbb) \ = \ \left\{ \varphi \in \Lbb \ | \ \mathrm{im}(\varphi) \subset \Lie(\Abf^{t})^{\vee} \text{ and }  \varphi|_{\Lie(\Abf^t)^{\vee}} \equiv 0 \right\} \]
\end{definition}
More concretely, by identifying $\Lie(\Abf)^{\vee}$ with $\Lie(\Abf^t)^{\vee}$ via the polarization $\lambda$, we may write
\begin{equation} \label{eqn:omegaTautHomSpace}
	\omega^{\taut} \ \simeq \ \left\{ \varphi \in \Hom(\Lie(\Abf) , \Lie(\Abf)^{\vee}) \ | \ \varphi \text{ is symmetric, and } \varphi \circ i(c \otimes a) = i(c \otimes a')\circ \varphi  \right\}.
\end{equation}
Over $\Mcal_v(\C)$, $\omega^{\taut}_{/\Mcal_v(\C)}$ is canonically identified with $ (\Lbb_{/ \Mcal_v(\C)})^{(1,-1)}$, and is therefore a line bundle. Then, via this identification, \eqref{eqn:FritzMetric} defines a metric on $\omega^{\taut}_{/\Mcal_v(\C)}$.

The following proposition relates the tautological sheaf on $\Mcal$ to the determinant of the cotangent bundle, and implies in particular that $\omega^{\taut}$ is a line bundle. We will thus call it the \emph{tautological bundle}.

\begin{proposition} \label{prop:TautOmega}
There is an isomorphism $(\omega^{\taut})^{\otimes 2} \isomto \det \Omega \simeq \Lcal_1 \otimes \Lcal_2$ with the property that if $\ell_1 \otimes \ell_2$ corresponds to $\varphi_1 \otimes \varphi_2$, then
\[ ||\ell_1||^2 \cdot ||\ell_2||^2  \ = \ e^{-2  \gamma  } \frac{D_B^4 }{\pi^6 (64)^3} \cdot ||\varphi_1||^2 \cdot  ||\varphi_2||^2.  \]
\begin{proof}
	Recall the decomposition
	\[ \Lie(\Abf) \  =  \ \Lie(\Abf)_1 \ \oplus \ \Lie(\Abf)_2,\]
	where each factor is equipped with an action of $C(L)$.
	We may decompose each factor further via the action of $\OF$ viewed in the centre of $C^+(L)$; for example, 
	\[ \Lie(\Abf)_1 \ = \ \Fcal_1 \ \oplus \ \Fcal_1' \]
	where 
		\[ \Fcal_1 \ := \ \left\{ x \in \Lie(\Abf)_1 \ | \ i(a\otimes 1) \cdot x \ = \ \tau(a) \cdot x \text{ for all } a \in \OF \right\} \]
	and
		\[ \Fcal'_1 \ := \  \left\{ x \in \Lie(\Abf)_1 \ | \ i( a\otimes 1) \cdot x \ = \ \tau(a') \cdot x \text{ for all } a \in \OF \right\}. \]
	Recall that $\tau \colon \OF[1/N] \to \Ocal_{\Mcal_{\Scal}}$ is the structural morphism.
	
	Note that each factor is a rank two $\mathcal O_{\Mcal_{\Scal}}$-bundle, is stable under the action of $\mathcal O_B  = C^+(L)$, and moreover the element $v_0 \in C^-(L)$ induces an isomorphism
	\[ i(v_0) \colon \Fcal_1 \isomto \Fcal_1' . \]
	On the other hand, if $\varphi \in \Lcal_1$ then arguing as in \eqref{eqn:KSmap} implies that $\varphi$ preserves $\Fcal_1$ and $\Fcal_1'$, and so
	\[ \varphi \ =  \ \varphi_1 \oplus \varphi_1', \] 
	with $\varphi_1 \in \Hom(\Fcal_1, \Fcal_1^{\vee})$ and $ \varphi_1' \in \Hom(\Fcal_1', (\Fcal_1')^{\vee})$. 
	Since 
	\[ i(v_0)^{\vee} \circ \varphi_1' \ = \ i(v_0^*)^{\vee} \circ \varphi_1' \ = \ \varphi_1 \circ i (v_0) \]
	and $i(v_0)$ is an isomorphism, it follows that $\varphi_1'$ is determined by $\varphi_1$, and in particular
	\[ \Lcal_1 \ \simeq \ \left\{  \varphi \in \Hom(\Fcal_1, \Fcal_1^{\vee})  \ | \ \varphi\text{ is symmetric and }\varphi \circ i(b) \ = \ i(b^*)^{\vee} \circ \varphi \ \text{ for all } b \in C^+(L) \right\}. \]
	Define a map
		\begin{equation} \label{eqn:L1Hodge}
			 \Lcal_1 \ \to \ \det \Fcal_1^{\vee} \ = \ \left( \bigwedge^2 \Fcal_1 \right)^{\vee} 
		\end{equation}
	by sending a morphism $\varphi \colon \Fcal_1 \to \Fcal_1^{\vee}$ to the linear functional
	\[ x \wedge y \ \mapsto \ \varphi \left( i(\theta) x \right) (y)  \]
	where $\theta \in B_0 \cap C^+(L) $ was chosen as in \eqref{eqn:thetadef}. By the defining properties of morphisms in $\Lcal_{1}$ and because $\theta^2  = - \theta \cdot \theta^* = - D_B \in \Ocal_{\Mcal}^{\times}$, it is immediate that \eqref{eqn:L1Hodge} is well defined and an isomorphism. A similar argument gives an isomorphism
		\[ \Lcal_2 \ \isomto \ \det \Fcal_2^{\vee} \]
	where 
		\[ \Fcal_2 \ := \ \left\{ x \in \Lie(\Abf)_2 \ | \ i(a \otimes 1) \  x \ =\ \tau( a) \cdot x \ \text{ for all } a \in \OF \right\}. \]
	
	On the other hand, it is clear by the equivariance properties in \eqref{eqn:omegaTautHomSpace} that there is a map
	\begin{equation} \label{eqn:omegaTautEval}
		\omega^{\taut} \otimes \Fcal_1 \ \to \ \Fcal_2^{\vee}, \qquad \varphi \otimes f \mapsto \varphi(f) .
	\end{equation}
	We claim that this map is an isomorphism. Assuming the claim for the moment, we observe it follows that $\omega^{\taut}$ is locally free of rank 1. Indeed, the isomorphism implies that $\omega^{\taut}$ is a direct summand of the vector bundle $\Fcal_{1}^{\vee}\otimes\Fcal_{2}^{\vee}$ (we use that $4$ is invertible in $\Ocal_{\Mcal}$), and therefore it is flat, hence locally free. Moreover we know it has generic rank 1. We thus conclude that $\omega^{\taut}$ is indeed a line bundle. 
	
	To prove the claim, as $\Fcal_{2}^{\vee}$ is locally free and by Nakayama's lemma, it suffices to show that the map defines an isomorphism on the fibres at any geometric point $s \colon \Spec(\kappa(s)) \to \Mcal$ over $\Spec(\OF[1/N])$; in particular, the fibres $\Fcal_{1,s}$ and $\Fcal^{\vee}_{2,s}$ are vector spaces of dimension two over $\kappa(s)$ equipped with an action of 
	\[ C^+(L) \otimes_{\OF}\kappa(s)  \ = \ \Ocal_B \otimes_{\OF} \kappa(s)\] such that the determinant of the endomorphism corresponding to an element of $C^+(L) \otimes_{\OF} \kappa(s)$ is equal to its (reduced) norm.
	
	 Since  $\kappa(s)$ is algebraically closed, we may fix an isomorphism
	\[ C^+(L) \otimes_{\OF} \kappa(s) \simeq \ M_2(\kappa(s)), \]
	and hence, using the idempotents $(\begin{smallmatrix}1&0\\0&0 \end{smallmatrix})$ and $ (\begin{smallmatrix}0&0\\0&1 \end{smallmatrix})$, we obtain isomorphisms
	\[
		\Fcal_{1,s} \ \simeq \ \kappa(s) \oplus \kappa(s) \ \simeq \ \Fcal_{2,s}^{\vee}
	\]
	with the $M_2(\kappa(s))$ action being the natural one. Since $\omega^{\taut}_{s}$ consists of $\kappa(s)$-linear maps $\Fcal_1 \to \Fcal_2^{\vee}$ that commute with the $C^+(L)$-action, it follows that in these coordinates a non-trivial element of $\omega^{\taut}_s$ is simply multiplication by a non-zero scalar, which in turn implies that \eqref{eqn:omegaTautEval} is an isomorphism.
	
	In particular, taking determinants yields isomorphisms
	\[
		(\omega^{\taut})^{\otimes 2}  \ \simeq \  \det\Fcal_1^{\vee} \otimes \det\Fcal_2^{\vee} \  \simeq \  \Lcal_1 \otimes \Lcal_2 \  \simeq \ \det \Omega_{\Mcal/\Scal}
	\]
	as required.
	
	The claim regarding metrics can be checked, say at a point $\mathbf z = (z_1, z_2) \in \lie H^2$, by using the vector appearing in \eqref{eqn:omegaTautBasis}  as a basis for $\omega^{\taut}_{\mathbf z}$, and the bases
	\[
		\left(  (\begin{smallmatrix} 1 & -\overline{z_1} \\ 0 & 0 \end{smallmatrix}) , \, 0 \right) \wedge \left( (\begin{smallmatrix}0 & 0 \\ 1 & -\overline{z_1}  \end{smallmatrix}) , \, 0 \right)   \qquad \text{ and } \qquad 	  \left( (\begin{smallmatrix} 1 & -\overline{z_2} \\ 0 & 0 \end{smallmatrix}) , \, 0 \right) \cdot v_0 \wedge  \left( (\begin{smallmatrix}0 & 0 \\ 1 & -\overline{z_2}  \end{smallmatrix}) , \, 0 \right) \cdot v_0  
	\]
	for $ \det \Fcal_{1,\zbf}$ and $\det \Fcal_{2,\zbf}$ respectively; the details are left to the reader.
\end{proof}
\end{proposition}

\section{Arithmetic intersection numbers on twisted Hilbert modular surfaces}

\subsection{The arithmetic intersection $\ac_{1}(\overline{\Omega})\ac_{2}(\overline{\Omega})$}
This section is devoted to the evaluation of the following arithmetic intersection number, which will subsequently appear in the Grothendieck-Riemann-Roch formula for $\Mcal_K$. Recall that we are viewing $\Mcal_K$ as an arithmetic variety over $\Spec \Z[1/N]$, where $N$ is, in particular, divisible by $2 \dF D_B$, as in Section \ref{sec:HMSIntModel}.
\begin{theorem} \label{thm:c1c2thm}
	\[  \ \widehat \deg \ \ChernHatOne( \overline \Omega) \cdot \ChernHatTwo( \overline \Omega) \ \equiv \ \deg \left( c_1(\Omega_{\Mcal(\C)})^2 \right) \ \left( -4 \log \pi - 2\gamma + 1  +  \frac{\zeta'_F(2) }{\zeta_F(2)} \right) \ \in \ \R_N. \]
Here $\gamma = - \Gamma'(1)$ is the Euler-Mascheroni constant, and $\R_N = \R / \oplus_{q|N} \Q \log q$.
\end{theorem}
\begin{remark} \label{rmk:postmainthm}
$(i)$ By the definition of the arithmetic degree map, the theorem holds for the arithmetic variety $\Mcal_K$ over $ \Spec \Z[1/N]$ if and only if it holds for the base change $\Mcal_{K, \Scal}$ over $\Scal = \Spec \OF[1/N]$; this justifies our use of the base change in the sequel.

$(ii)$ Suppose $K' \subset K$ is another compact open subgroup, such that $K'_p = K_p$ for $p \nmid N$. Then there is a finite \'etale cover $ \Mcal' = \Mcal_{K'} \to \Mcal_K = \Mcal$ and 
\[  
 \frac{ \widehat \deg \ \ChernHatOne( \overline \Omega_{\Mcal'}) \cdot \ChernHatTwo( \overline \Omega_{\Mcal'}) }{ \deg \left( c_1(\Omega_{\Mcal'(\C)})^2 \right) } \ = \ 
  \frac{ \widehat \deg \ \ChernHatOne( \overline \Omega_{\Mcal}) \cdot \ChernHatTwo( \overline \Omega_{\Mcal}) }{  \deg \left( c_1(\Omega_{\Mcal(\C)})^2 \right) },
\]
since both numerator and denominator are multiplied by the degree of the cover. In particular, proving Theorem \ref{thm:c1c2thm} for a single sufficiently small $K$ implies that the theorem holds for all such $K$. 

$(iii)$ Let $\lie p $ be a prime of $F$ such that $(\lie p, N) = 1$, and 
\[ \Mcal_{(\lie p)} \ = \ \Mcal_K \times_{\Z[1/N]} \Spec(\mathcal O_{F,(\lie p)}). \]
One can consider the local analogue of Theorem \ref{thm:c1c2thm} for $\Mcal_{(\lie p)}$, where the equality takes place in $\R_{(\lie p)} := \R / \oplus_{(q, \lie p) = 1} \Q \log q$: Theorem \ref{thm:c1c2thm} holds for $\Mcal_K$ if and only if the corresponding local  statement holds for every $\lie p$ relatively prime to $N$. 

As explained in \cite[\S12]{KR-HB}, the scheme $\Mcal_{(\lie p)}$ is the ``localized" moduli space over $\mathcal O_{F, (\lie p)}$ whose $S$ points, for a scheme $S$ over $\Ocal_{F, (\lie p)}$, parametrize prime-to$p$ isogeny classes of tuples $\{\underline A = ( A, i, \lambda, [\eta^p]) \}$, where:
	\begin{enumerate}[(i)]
		\item $A \to S$ is an abelian scheme, up to prime-to-$p$ isogeny;
		\item $i \colon C(L) \otimes \OF\to \End(A) \otimes \Z_{(p)}$ satisfies the same determinant and involution conditions as above;
		\item $\lambda$ is a $\Z_{(p)}^{\times}$ class of principal polarizations;
		\item $[\eta^p]$ is a $K^p$-class of $C(L) \otimes \OF$-equivariant isomorphisms $\eta^p \colon \Ta^p(A) \otimes \Q \isomto U \otimes \A_f^p$ that identify the corresponding symplectic forms up to multiplication by a scalar in  $(\A_f^p)^{\times}$. 
	\end{enumerate}

%

\end{remark}

\paragraph{Notation:} For the remainder of this section, we fix a sufficiently small $K \subset \Gbf(\A_{\Q,f})$, a rational prime $p$ relatively prime to $N$, and a prime ideal $\lie p$ of $\Ocal_F$ above $p$. Note that the assumptions on $N$ imply that $p$ is odd and unramified. 

To lighten the notation, throughout this section we will drop the subscript $\lie{p}$, and write, for example,
\[ \Mcal =   \Mcal_K \times \Spec(\mathcal O_{F,(\lie p)}), \qquad \Omega = \Omega_{\Mcal / \mathcal O_{F, (\lie p)}}, \qquad \text{etc.} \]

In light of the above remarks, it will suffice to prove Theorem \ref{thm:c1c2thm} for the localized scheme $\Mcal$, with both sides of the identity viewed in $\R_{(p)}$. The strategy, carried out over the next several subsections, is to express $\ChernHatOne(\overline\Omega)$ in terms of the divisor of a Borcherds form. One problem, however, is that the absence of cusps on $\Mcal$ means that information about the arithmetic of Borcherds' forms is less accessible; we circumvent this issue by embedding everything in a larger Shimura variety, as described in the next subsection.

\subsection{Twisted Siegel threefolds and special cycles}\label{sec:Siegel-threefold}
In this section, we construct a twisted Siegel modular threefold containing $\Mcal$, and describe certain automorphic vector bundles on it. We continue to fix $p$ relatively prime to $N$, and a prime ideal $\lie p$ of $F$ above $p$. 
\subsubsection{Moduli problems} \label{subsec:Siegel1-moduli}

We start by enlarging the quadratic lattice $L$: let
\[ {\Lambda} \ := \ L \ \oplus \ \langle \mathbb 1 \rangle  \]
where $\langle \mathbb 1 \rangle$ is the rank one quadratic lattice with a basis vector $\mathbf e$ satisfying $Q(\mathbf e) = 1$. Note that the signature of $\Lambda $ is $(3,2)$ and that $ \Lambda_{(p)}$ is self-dual. There is an involutive embedding 
\[ \Ocal_B \simeq C^+(L) \ \hookrightarrow \ C^+(\Lambda). \]
We had previously fixed an element $\theta \in C^+(L)$ with ${}^{\iota} \theta = - \theta$ and $\theta^{2}=-D_{B}$, which we also view as an element of $C^+(\Lambda)$;   the involution
\[ c \mapsto c^* := \theta \cdot {}^{\iota} c \cdot \theta^{-1} \]
then defines a positive involution on $C(\Lambda)$. Let 
\[ \widetilde U = C^+(\Lambda_{\Q}), \]
 which we view as  a $C^+(\Lambda)$-module by left-multiplication, and equip it with the symplectic form
\begin{equation} \label{eqn:UTildePol} \langle c_1, c_2 \rangle_{\widetilde U} \ := \ Tr \left( {}^{\iota}c_2 \cdot \theta \cdot c_1 \right). 
\end{equation} 

Set $\widetilde \Gbf = GSpin( \Lambda_{ \Q})$ and let $\widetilde K = \widetilde K^p \widetilde K_p \subset \widetilde \Gbf(\A_{\Q,f})$ be a sufficiently small compact open subgroup such that
\[ \widetilde K_p \ = \ C^+(\Lambda_{\Z_p})^{\times} \ \cap \ \widetilde \Gbf(\Q_p) \qquad \text{and} \qquad \widetilde K^p \ \cap \  \Gbf(\A_{Q,f}^p) \ = \ K^p, \]
where we recall that $\Gbf= GSpin(L_{\Q})$ and $K$ are as in the previous sections. We also assume, for later convenience, that
\[ \widetilde K \ \subset \  \widetilde K^{max} \ =  \  C^+(\Lambda_{\widehat \Z})^{\times}  \cap \widetilde \Gbf(\A_{\Q,f}).   \]
Attached to all of this data is the Shimura variety\footnote{In fact, this Shimura variety has a canonical model over $\Q$, and the integral model can be defined over $\Z$.} $\widetilde \Mcal  \to \Spec F$, which is a \emph{twisted Siegel threefold} in the terminology of \cite{KRSiegel}. It is of  PEL type, with 
 a natural integral model 
\[ \widetilde{ \Mcal} \ \to \ \Spec(\mathcal O_{F, (\lie p)}) \]
defined by the following moduli problem.

\begin{definition} \label{def:MtildeModProb} Consider the moduli problem $\widetilde \Mcal$ over $\Spec \mathcal O_{F,({\lie p})}$ whose $S$-points, for a scheme $S \to \Spec \mathcal O_{F,({\lie p})}$ ,  are the prime-to-$p$ isogeny classes of tuples
\[ \widetilde{\Mcal} (S) \ = \ \{ \underline{\widetilde {A}} \ = \ (\tilde A, \tilde i, \tilde \lambda, [\tilde \eta])  \} \]
where
	\begin{itemize}
		\item $\tilde A \to S$ is an abelian scheme of relative dimension 8, up to prime-to-$p$ isogeny;
		\item $\tilde \lambda$ is a $\Z_{(p)}^{\times}$ class of principal polarizations;
		\item $\tilde i \colon C^+( \Lambda) \to \End_S({\tilde A}) \otimes \Z_{(p)}$  is an action of the even part of the Clifford algebra of $\Lambda$  such that:
			\begin{itemize}
						\item	$ \det( \tilde i(c ) |_{\Lie\tilde A} ) \ = \ Nrd(c)^2 \ \in \ \mathcal O_S$, where $Nrd$ is the reduced norm on $C^+(\Lambda)$;
						\item $\tilde i(c)^* \ = \ \tilde i( c^*)$ for the Rosati involution $*$ induced by $\tilde \lambda$;
					\end{itemize}
		\item $[\tilde \eta]$ is a $\widetilde K^p$-equivalence class of $C^+(\Lambda)$-equivariant isomorphisms
			\[ \tilde \eta \colon \Ta^p(\tilde A)_{\Q} \ \isomto \ \widetilde{ U} (\A_{\Q,f}^p) \]
			preserving the corresponding symplectic forms up to an $(\A_{\Q,f}^p)^{\times}$-multiple.
	\end{itemize}
Assuming $\widetilde K$ is sufficiently small, this moduli problem is representable, as a fine moduli space, by a smooth quasi-projective scheme over $\Spec \mathcal O_{F, (\lie p)}$, which we again denote by $\widetilde \Mcal$. 
\end{definition}

Given a point $\underline{\tilde A} = (\tilde A, \tilde i, \tilde \lambda, [\tilde \eta]) \in \widetilde \Mcal(S)$ for some connected base scheme $S$, consider the space of \emph{special endomorphisms} \cite[Definition 2.1]{KRSiegel}
\[{ \mathbf V}(\underline{\tilde A}) \ := \ \{ y \in \End(\tilde A, \tilde i) \otimes_{\Z} \Z_{(p)} \ | \ y= y^* \ \text{ and } \  tr^o(y) = 0 \}\]
where $y^*$ is the image of $y$ under the Rosati involution induced by $\tilde \lambda$, and  $tr^o(y)$ denotes the reduced trace on $ \End(\tilde A, \tilde i)\otimes \Q$.

If $y \in \mathbf V(\underline{\tilde A})$, then it can be checked that $y^2$ is a scalar multiple of the identity, cf.\ \cite[Lemma 2.2]{KRSiegel}. In particular, we obtain a quadratic form $\tilde Q$ on $ \mathbf V(\underline{\tilde A})$ defined by
\[ \tilde Q(y) \cdot \Id \ = \ \frac{1}{\dF} \ y^2 \ = \ \frac{1}{\dF} \ y \circ y^* \] 
which is positive-definite by the positivity of the Rosati involution; the reason for the normalization by $1/\dF$ will become apparent in the construction of the tautological bundle below.

For a $\widetilde K^p$-invariant compact open subset $\lie V \subset \Lambda \otimes \A_{\Q,f}^p$ and a rational number $m \in \Q_{>0}$, we define the \emph{special cycle} $\tilde \Zed (m, \lie V)$, as in \cite[\S 2]{KRSiegel}, to be the moduli space over $\Spec \mathcal O_{F,(\lie p)}$ whose  points are:
	\[ \tilde \Zed (m, \lie V)(S) \ = \ \{( \underline{  \tilde A}, y) \} \]
where
	\begin{itemize}
		\item $\underline {\tilde A } = (\tilde A, \tilde i, \tilde \lambda , [ \tilde \eta]) \in \widetilde \Mcal(S) $;
		\item $y \in \mathbf V(\underline{\tilde A})$ such that $\tilde Q(y) = m$;
		\item and the following ``prime-to-$p$ integrality condition" holds: for every $ \eta \in [\tilde \eta]$, the endomorphism
				\[ \eta^{-1} \circ \Ta^p(y) \circ \eta \ \in \ \End_{C^+(\Lambda)}( \widetilde U \otimes \A_{\Q,f}^p) \]
			  is given by right-multiplication by an element of 
			  	\[ \lie V \cdot \theta \, \mathbf e \ \subset \ C^+(\Lambda) \otimes_{\Z} \A_{\Q,f}^p; \]
			  recall that we had defined $\widetilde U = C^+(\Lambda_{\Q})$, viewed as a $C^+(\Lambda)$-module by left-multiplication.
	\end{itemize}
Here, as before, these objects are considered in the prime-to-$p$-isogeny category.

\begin{proposition}[{\cite[Proposition 2.6]{KRSiegel}}] The moduli problem described above is representable (as a fine moduli space) by a scheme, which we denote again by $\tilde \Zed(m, \lie V)$, and the forgetful morphism
\[ \widetilde \Zed(m, \lie V) \ \to \ \widetilde \Mcal \]
is finite and unramified.
\end{proposition} 
Abusing notation, we will use the same symbol $\widetilde \Zed(m,\lie V)$ to denote the cycle theoretic image of $\widetilde\Zed(m,\lie V)$ in $\widetilde\Mcal$.

\subsubsection{Embedding $\Mcal \hookrightarrow \widetilde\Mcal$}
We now bring the twisted Hilbert modular surface $\Mcal$ back into the picture: the embedding 
\[ \Gbf = GSpin(L_{\Q}) \to \widetilde \Gbf = GSpin(\Lambda_{\Q})   \] 
induces a morphism of Shimura varieties $\mathbf j\colon\Mcal \to \widetilde \Mcal$ that is best described in moduli-theoretic terms; roughly speaking, $\Mcal$ is the locus in $\widetilde \Mcal$ where the action of $C^+( \Lambda)$ can be extended to an action of the full Clifford algebra $C( \Lambda)$.

To make this more precise, fix a $\Z_{(p)}$-basis 
	\[ \OF \otimes \Z_{(p)}  \ = \ \Z_{(p)} \ \oplus \Z_{(p)} \cdot \varpi \]
with $\varpi^2 = \mathbf d_F$ and $\varpi' = - \varpi$.  Recall that $\Ocal_{B,(p)} = C^+(L)_{(p)}$ is a $\Z_{(p)}$ order in our original quaternion algbera $B$; in particular, as its centre contains a copy of $\OF$, we may also fix an element $\delta \in C^+(L)$ satisfying these same relations. 

Note that the inclusion $L \subset \Lambda$ induces an inclusion $C(L) \subset C(\Lambda)$ of algebras; on the other hand, we may define a (bijective linear) map
\begin{equation} \label{eqn:BetaDef}
 \beta \colon C(L) \isomto C^+(\Lambda) 
\end{equation}
determined by setting
\[ 
\beta(c) \ = \ \begin{cases} c, & \text{if } c \in C^+(L) \\ c \, \mathbf e , & \text{if } c \in C^-(L)  \end{cases} 
\]
and extending by linearity. It can be verified directly that $\beta$ is an algebra isomorphism.

Changing perspective slightly, we may also view $\beta$ as a map between the modules $U_{\Z}= C(L)$ and $\widetilde U_{\Z}=C^+(\Lambda)$; essentially by construction, $\beta$ intertwines the left-multiplication actions of $C(L)$ on $U_{\Z}$ with that of $C^+(\Lambda)$ on $\widetilde U_{\Z}$, i.e.\, 
\[ \beta \circ \iota (c \otimes 1) \ = \ \widetilde \iota (\beta(c)) \circ \beta \qquad \text{ for all } c \in C(L).
\]
Moreover, it can again be verified directly that $\beta$ preserves the symplectic forms \eqref{eqn:UZPol} and \eqref{eqn:UTildePol} on $U_{\Z}$ and $\widetilde U_{\Z}$ respectively, and preserves the previously defined involutions.

Finally,  note that the element $\delta \, \mathbf e \in C^-(\Lambda)$ satisfies $(\delta \, \mathbf e)^2 = \delta^2 \mathbf e ^2 = \dF$ and generates the centre of $C(\Lambda)$, and in particular there is an isomorphism
\[ 
	(C^+(\Lambda) \otimes_{\Z} \OF)_{(p)} \ \isomto \ C(\Lambda)_{(p)} \qquad \text{sending } c \otimes \varpi \mapsto c \, \delta \, \mathbf e.
\]
Thus we obtain an isomorphism
\begin{equation} \label{eqn:BetaTildeDef}
\begin{CD}
	\widetilde \beta \colon( C(L) \otimes \OF )_{(p)}  @>^{\beta \otimes 1}>> (C^+(\Lambda) \otimes \OF )_{(p)} @>>> C(\Lambda)_{(p)}.
\end{CD}
\end{equation}

We now give the modular definition of the morphism $\mathbf j \colon \Mcal \to \widetilde\Mcal$. Given a scheme $S \to \Spec \mathcal O_{F, ({\lie p})}$ and a point $\underline A = (A, \lambda, i, [\eta]) \in \Mcal(S)$, we may define a point 
\[ \mathbf j(\underline A) \  = \  (A,  \tilde i, \lambda, [\tilde \eta]) \  \in   \ \widetilde \Mcal(S) \]
with the same underlying abelian scheme and polarization, and where 
\[ \tilde i \ = \ i \circ \beta^{-1} \colon C^+(\Lambda)_{(p)} \to \End(A) \otimes \Z_{(p)} \]
is the restriction of $i$ to $C^+(\Lambda)_{(p)} \simeq C(L)_{(p)}$ via \eqref{eqn:BetaTildeDef}, and 
\[ \tilde \eta =  \beta  \circ \eta \colon \Ta^p(A) \otimes_{\Z} \Q \ \isomto  \widetilde U(\A_f^p). \]
It is easily verified that these data satisfies the hypothesis of Definition \ref{def:MtildeModProb}, and that the construction is functorial;
applying this construction to the universal abelian variety over $\Mcal$ gives the desired morphism
\[ \mathbf j \colon \Mcal \ \to \ \widetilde \Mcal. \]

Our next task is to investigate the special cycles along $\mathbf j$. Suppose $\underline A \in \Mcal(S)$ and set
\[ \underline{\tilde A} \ = \ \mathbf j (\underline A) \]
In particular, the action of $C^+(\Lambda)$ on  $\tilde A$ extends to an action 
\[ i \colon C(\Lambda)_{(p)} \ \simeq \ (C(L) \otimes \OF)_{(p)}  \to \End(\tilde A) \otimes \Z_{(p)}, \]
and so we may define two $\Z_{(p)}$ submodules 
\begin{equation}\label{eq:V-sharp}
\Vbf^{\sharp}(\mathbf j( \underline{ A})) \ := \ \{ y \in \Vbf(\mathbf j( \underline{ A})) \ | \ y \circ i(1\otimes a) \ = \ i(1 \otimes a') \circ y \ \text{for all } a \in \OF \}
\end{equation}
and 
\begin{equation}\label{eq:V-flat}
 \Vbf^{\flat}( \mathbf j (\underline{ A})) \ := \ \{ y \in \Vbf(\mathbf j( \underline{ A})) \ | \ y \circ i(1\otimes a) \ = \ i(1 \otimes a) \circ y \ \text{for all } a \in \OF \} 
 \end{equation}
of $\Vbf(\mathbf j(\underline A))$.
\begin{lemma} \label{lem:Vdecomp} Let $\underline A \in \Mcal(S)$ for some scheme $S$. Then:
	\begin{enumerate}[(i)] 
		\item $\Vbf^{\flat}(\mathbf j(\underline A))$ is a free $\Z_{(p)}$ module of rank one, and is generated by a basis vector with norm $1$. 
		\item There is an orthogonal direct sum decomposition
			\[  \Vbf(\mathbf j(\underline A)) \ = \  \Vbf^{\sharp}(\mathbf j(\underline A)) \ \stackrel{\perp}{\bigoplus} \  \Vbf^{\flat}(\mathbf j(\underline A)). \]

	\end{enumerate}
	\begin{proof}  In this proof, abbreviate $\Vbf^{\sharp} = \Vbf^{\sharp}(\mathbf j(\underline A))$ and $\Vbf^{\flat} = \Vbf^{\flat}(\mathbf j(\underline A))$.
	
		(i) Note that $i(1 \otimes \varpi) \in \Vbf^{\flat}$ and
		\[ \widetilde Q \left( i(1 \otimes \varpi) \right) \ = \ \frac1\dF \cdot \varpi^2 \ = \   \frac1\dF \cdot \dF \  = \ 1. \]
		
		We wish to show that $i(1 \otimes \varpi)$ generates $\Vbf^{\flat}$; since $\End(A) \otimes \Z_{(p)}$ is torsion free as a $\Z_{(p)}$ module, it will suffice to show that 
		\[ \dim_{\Q} \left( \Vbf^{\flat} \otimes_{\Z_{(p)}} \Q  \right) \leq 1. \]
		
		Since this dimension can only increase upon specialization, it suffices to prove this inequality in the case $\underline A = (A,i,\lambda,[\eta])\in \Mcal(\kappa)$ for an algebraically closed  field $\kappa$. Fix a prime $\ell \neq p$ relatively prime to $\mathtt{char}( \kappa)$, so that
			\[ \End_{\kappa}(A, i) \otimes \Q_{\ell} \ \hookrightarrow \ \End_{\Q_{\ell}}(\Ta_{\ell}(A, i)_{ \Q_{\ell}}) \ 
				\stackrel{\eta_{\ell} }{\simeq}  \ \End_{C(L) \otimes \OF}( U_{\Q_{\ell}}),\]
		where $\eta_{\ell}$ is the $\ell$-component of a $C(L)\otimes \OF$-equivariant isomorphism 
			\[ \eta \colon \Ta^p(A)_{\Q} \ \isomto \ U(\A_f^p) , \qquad \qquad \eta \in [\eta]. \]
		Therefore
		\[ \dim_{\Q}( \Vbf^{\flat}\otimes {\Q}) \ \leq \ \dim_{\Q_{\ell}} \{ y \in  \End_{C(L) \otimes \OF}( U_{\Q_{\ell}}) \ | \ y^* = y , \ tr(y)=0  \}. \]
		For an element $y$ of this latter space: 
			\begin{itemize}
				\item the $C(L) \otimes \OF$-equivariance implies that $y$ is given by right-multiplication by some element $x = y(1) \in C^+(V) \otimes \Q_{\ell}  = B \otimes \Q_{\ell}$;
				\item the condition $y = y^*$ then implies that ${}^{\iota}x = x$, which in turn gives $x \in Z(B \otimes \Q_{\ell}) = F \otimes \Q_{\ell}$, a two-dimensional $\Q_{\ell}$ vector space;
				\item the condition $tr(y) = 0$ then implies that $tr_{F_{\ell}}(x) = 0$ .
			\end{itemize}		
		Thus, the dimension of this latter space is one. 
		
		(ii) The symmetric bilinear form attached to  the quadratic form $\widetilde Q(y) = \frac1\dF y^2$ is given by
		\[ [x,y] \ = \ \frac1{2 \dF}\left( x\circ y + y \circ x \right). \] 
		Therefore using the basis vector $i(1 \otimes \varpi)$, it follows that the orthogonal complement $ (\Vbf^{\flat}) ^{\perp}$ of $\Vbf^{\flat}$ in $\Vbf$ consists exactly of those vectors $y$ such that 
		\[  y \circ i(1 \otimes \varpi) \ + \ i(1 \otimes \varpi) \circ y \ =\ 0 ;\]
		From this it immediately follows
		\[ \Vbf^{\sharp} \ = \ (\Vbf^{\flat})^{\perp},\]
		which implies $(ii)$.
	\end{proof}
\end{lemma}

\begin{proposition} \label{prop:SCBundle}
Suppose $m \in \Z$ such that $m \not\equiv \square \mod p$ and $\lie V \subset \Lambda(\A_f^p)$ is a $\widetilde K$-invariant compact open subset. Consider the intersection
\[ \pi \colon \widetilde \Zed(m,\lie V) \ \times_{\widetilde \Mcal} \ \Mcal \ \to \ \Mcal.  \]
Then, assuming the intersection is non-empty, there is an isomorphism of line bundles
\[  \xi \colon \pi^* \mathcal L_1 \simeq \pi^* \mathcal L_2 \]
where $\Omega_{\Mcal/\Scal} = \Lcal_1 \oplus \Lcal_2$ as in \eqref{eqn:KSmap}, and such that for any section $\varphi$ of $\pi^*\Lcal_1$, 
\[ || \xi(\varphi)||^2 \ = \ c \cdot || \varphi ||^2  \]
for some locally constant function $c$ valued in $\Z_{(p)} ^{\times} $. 

\begin{proof}
Let $Z$ be a connected component of $  \Zed(m,\lie V) \times_{\widetilde \Mcal} \Mcal$. In terms of the moduli problems, the universal abelian variety $\underline A$ over $Z$ is equipped with an extra endomorphism
\[ y \in \Vbf( \underline A), \qquad \widetilde Q(y) = \frac1\dF y^2 = m. \]
By Lemma \ref{lem:Vdecomp}, the endomorphism $y$ decomposes as $y = y^{\sharp} + y^{\flat}$, with
\[ m \ =  \ \widetilde Q(y^{\sharp}) \ + \  \widetilde Q(y^{\flat}) \]
and moreover $\widetilde Q(y^{\flat})$ is of the form $\alpha^2$ for some $\alpha \in \Z_{(p)}$. The assumption that $m \not\equiv \square \pmod{p}$ then implies that 
\[ \widetilde Q(y^{\sharp}) \ = \ \frac1\dF \, (y^{\sharp})^2 \  \in \ \Z_{(p)}^{\times}; \]
recall here that $(p,\dF)=1$ by assumption.

In particular, $y^{\sharp} \neq 0$, and by construction, $y^{\sharp}$ anti-commutes with the $\OF$-action on $A$. Recalling the definition,
	{\small \[ \Lcal_j \ = \ \left\{ \varphi \in \Hom \left(\Lie(\Abf)_j, \Lie(\Abf)_j^{\vee} \right)  \ | \ \varphi\text{ symmetric and }\varphi \circ i(b) \ = \ i(b^*)^{\vee} \circ \varphi \ \text{ for all } b \in C(L)   \right\}, \]}
note that there is a map
	\[  \xi \colon (\pi^*\Lcal_1)|_Z \to ( \pi^*\Lcal_2)|_Z, \qquad 	\varphi  \mapsto \ (y^{\sharp})^{\vee} \circ \varphi \circ  y^{\sharp}, \]
where, as usual, we use the same symbol $y^{\sharp}$ to denote the induced endomorphism of $\Lie(A)$. Since $(y^{\sharp})^2 \in \mathcal O_Z^{\times}$, it follows immediately that $\xi$ is an isomorphism.

It remains to track the effect of $\xi$ on metrics. As in the proof of Proposition \ref{prop:OmegaCpxDecom}, fix a complex embedding $v \colon F \to \C$ and work over the component
 \[  \lie H^2 \to \Gamma_i \backslash \lie H ^2  \ \subset \  \Mcal_{v}(\C) . \]
Suppose $\mathbf z =(z_1, z_2) \in \lie H^2 $ is a point on $Z(\C)$ mapping to this component, and $A_{\zbf}$ the corresponding abelian variety; then the endomorphism $y^{\sharp}$ induces an endomorphism of 
\[ H_1^{dR}(A_{\zbf}(\C)) \ \simeq \ U_{\C} \ = \ C(L) \otimes_{\Z} \C \]
 with the following properties:
	\begin{itemize}
		\item Since $y^{\sharp}$ commutes with the left $C(L)$-action on $U_{\C}$, the action of $y^{\sharp}$ is given by right-multiplication by an element $X \in C(L)\otimes \C$.
		\item As $y^{\sharp}$ preserves the homology lattice $H_1(A_{\zbf}(\C), \Z) \otimes \Z_{(p)} \simeq C(L)_{(p)}$, it follows that $X \in C(L) \otimes \Z_{(p)}$ -- the compatibility condition with $\lie V$ imposes further integrality conditions, but this plays no role in the present discussion.
		\item Recall the decomposition
			\[ U_{\R} \ = \ (M_2(\R) \oplus M_2(\R) ) \ \oplus \ (M_2(\R) \oplus M_2(\R) ) v_0 , \]
		on which the element $1 \otimes \varpi \in C(L) \otimes \OF$ acts by right multiplication by 
		\[ \left( \tau(\varpi) \cdot \Id, \ \tau(\varpi') \cdot \Id\right) \ = \ ( \tau(\varpi) \cdot \Id, \  - \tau(\varpi) \cdot \Id). \]
		Since $y^{\sharp}$ anti-commutes with this action, it follows that the image of $X \in U_{\R}$ is of the form
		\[ X \ = \  (B_1, B_2) \cdot v_0  \]
		for some $B_1, B_2 \in M_2(\R).$
		\item We have the implications 
			\[ y^{\sharp} = (y^{\sharp})^*  \ \implies \  X^{\iota} = X  \ \implies \  B_1 = B_2^{\iota} ; \] 
			thus $X = (B, B^{\iota}) \cdot v_0$ for some $B \in M_2(\R)$.
		\item Since $ \frac1 \dF \cdot (y^{\sharp})^2 = \widetilde Q(y^{\sharp}) \cdot \Id$, 
		\[ \dF \cdot \widetilde Q(y^{\sharp}) \cdot \Id \ = \ X^2 \ = \  (B, B^{\iota}) \cdot v_0 \cdot (B, B^{\iota}) \cdot v_0 \ = \ (B, B^{\iota}) \cdot (B^{\iota} , B) \cdot (v_0)^2 \ = \ \det(B) \cdot 1  \ = \ \det B \cdot \Id,\]
		and so
		\[	\det B = \dF \, \widetilde Q(y^{\sharp}) \ \in \ \Z_{(p)}^{\times}. \]
		\item Finally, $j^{\sharp}$ must preserve the complex structure on $A_{\zbf}$, which implies that $X$ and $J_{\zbf}$ commute. Since
		\[ X \cdot J_{\zbf} \ = \ (B, B^{\iota}) \cdot v_0 \cdot (j_{z_1}, j_{z_2}) \ =\ (B, B^{\iota}) \cdot (j_{z_2}, j_{z_1}) \cdot v_0  \ =  \  (B \cdot j_{z_2},  \ B^{\iota} \cdot j_{z_1}) \cdot v_0  \]
		and 
		\[ J_{\zbf} \cdot X \ = \ (j_{z_1}, j_{z_2}) \cdot (B, B^{\iota}) \cdot v_0   \ =\ (j_{z_1} \cdot B,  \ j_{z_2} \cdot B^{\iota}) \cdot v_0\]
		it follows that the condition that $X$ and $J_{\zbf}$ commute is equivalent to the identity
		\[ j_{z_1} \cdot B \ = \ B \cdot j_{z_2}. \]
		One easily checks that this last condition implies
		\begin{equation} \label{eqn:cycleCpxH2} B \cdot z_2 \ = \ \frac{a z_2 + b}{c z_2 + d } \ = \ z_1 \qquad \text{ where } B = \begin{pmatrix}	a & b \\ c & d \end{pmatrix} \in GL_2(\R). \end{equation}
	\end{itemize} 
	Now suppose $\varphi_{\zbf} \in \Lcal_{1, \zbf}$; recall from \eqref{eqn:metricLj} that
		\[ || \varphi_{\zbf} ||_{\zbf}^2 \ = \ \frac{16 \pi^2}{D_B} \cdot (y_1)^{-2} \cdot | \varphi_{\zbf}(\sigma^{(1)}_{\zbf}) (\sigma^{(1)}_{\zbf} ) |^2 \]
	and
		\[ || \xi(\varphi_{\zbf}) ||_{\zbf}^2 \ =  \ \frac{16 \pi^2}{D_B} \cdot (y_2)^{-2} \cdot \left|  \xi(\varphi_{\zbf})(\sigma^{(2)}_{\zbf}) (\sigma^{(2)}_{\zbf}) \right|^2 \ = \  \frac{16 \pi^2}{D_B} \cdot y_2^{-2} \cdot | \varphi_{\zbf}(\sigma^{(2)}_{\zbf} \cdot X) (\sigma^{(2)}_{\zbf}\cdot X) |^2. \]
On the other hand, 
	\begin{align*} 
		\sigma^{(2)}_{\zbf} \cdot X \ = \  \left( 0, \ \begin{pmatrix} 1 & - \overline{ z_2} \\ 0 & 0 \end{pmatrix} \right) \cdot ( B, B^{\iota}) \cdot v_0  \ =& \  \left( 0, \ \begin{pmatrix} 1 & - \overline{ z_2 }\\ 0 & 0 \end{pmatrix} \cdot B^{\iota} \right) \cdot v_0 \\
		=& \ \left( 0, \ \begin{pmatrix} 1 & - \overline {z_2} \\ 0 & 0 \end{pmatrix} \cdot \begin{pmatrix} d & -b \\ -c & a \end{pmatrix} \right) \cdot v_0 \\
		=& \ (c \overline {z_2} + d) \ \left( 0, \ \begin{pmatrix} 1 & - B \cdot \overline{z_2} \\ 0 & 0 \end{pmatrix}  \right) \cdot v_0 \\ 
		=& \ (c \overline{ z_2} + d)  \ v_0 \cdot  \left(  \begin{pmatrix} 1 & - B \cdot \overline{z_2} \\ 0 & 0 \end{pmatrix} , \ 0 \right) 
		 \ =  \ (c \overline{z_2} + d) \ i(v_0) \cdot \sigma_{\zbf}^{(1)}.
	\end{align*}
Therefore
	\begin{align*} 
	|| \xi (\varphi_{\zbf}) ||_{\zbf}^2 \ =& \  \frac{16 \pi^2}{D_B} \cdot  | c z_2 + d |^4 \cdot (y_2)^{-2} \cdot \left| \varphi_{\zbf}( i(v_0) \sigma^{(1)}_{\zbf} )(i(v_0) \sigma^{(1)}_{\zbf} ) \right|^2 \\
	=& \ \frac{16 \pi^2}{D_B} \cdot  | c z_2 + d |^4 \cdot (y_2)^{-2} \cdot \left| \varphi_{\zbf}(  \sigma^{(1)}_{\zbf} )(\sigma^{(1)}_{\zbf} ) \right|^2 \\
	 =&  \ \frac{16 \pi^2}{D_B} \cdot \det(B)^{-2} \cdot (y_1)^{-2} \cdot \left| \varphi_{\zbf}( \sigma^{(1)}_{\zbf} )( \sigma^{(1)}_{\zbf} ) \right|^2 \\
	 =& \  \dF^{-2} \cdot \widetilde Q(y^{\sharp})^{-2} \cdot || \varphi_{\zbf}||_{\zbf}^2,
	\end{align*}
where the second line follows from the $C(L)$-equivariance of $\varphi_{\zbf}$, and the third line follows from comparing imaginary parts in \eqref{eqn:cycleCpxH2}. Since $\dF \cdot \widetilde Q(y^{\sharp}) \in \Z_{(p)}^{\times}$, the proposition follows. 
\end{proof}
\end{proposition}
\subsubsection{Tautological bundles}\label{subsubsec:taut-siegel}
The space $\widetilde \Mcal$ also comes equipped with a tautological bundle, defined in an analogous way to the bundle $\omega^{\taut}$ on $\Mcal$. Consider the first de Rham homology bundle
\[ \widetilde{\bbU} \ := \ H_1^{dR}(\widetilde\Abf) \]
of the universal abelian variety $\widetilde\Abf$ over $\widetilde\Mcal$, equipped with the induced $C^+(\Lambda)$-action
\[ \tilde i \colon C^+(\Lambda) \otimes \Z_{(p)} \ \to \ \End(\widetilde\bbU) \]
and polarization
\[ \lambda_{\widetilde\bbU} \colon \widetilde \bbU \ \to \ \widetilde \bbU^{\vee} ,\]
  as well as the Hodge filtration
\[ \Fil^{0}(\widetilde\bbU) \  \subset \ \Fil^{-1}(\widetilde\bbU) \ = \ \widetilde\bbU. \] 
Set
\[ \widetilde{ \Lbb }  \ : = \ \left\{ \widetilde \varphi \in \End_{C^+(\Lambda)}(\widetilde\bbU) \ | \ \widetilde \varphi^* = \widetilde \varphi \text{ and }  tr^o(\widetilde \varphi) = 0   \right\}, \]
equipped with the quadratic form given (on sections) by $\widetilde{Q}(\widetilde\varphi) = \frac1\dF ( \widetilde \varphi \circ\widetilde \varphi^*) =  \frac1\dF \cdot\widetilde \varphi^2$. 
As a subbundle of $\End(\widetilde\bbU)$, the bundle $\widetilde{\mathbb L}$ inherits a filtration
\[
	\Fil^1(\widetilde{\mathbb L}) \ \subset \  \Fil^0(\widetilde{\mathbb L}) \ \subset \ \Fil^{-1}(\widetilde{\mathbb L}) = \widetilde{\mathbb L} ,
\]
and
\[ \widetilde{\omega}^{\taut} \ := \ \Fil^1(\widetilde{\mathbb L}) \]
is an isotropic line bundle called the \emph{tautological bundle}.

We now consider the pull-back $\mathbf j^* \widetilde\bbU$ to $\Mcal$. Recall that, for the universal object $\underline \Abf$ over $\Mcal$, we had defined the bundle
\[ \bbU \ := \ H^{dR}_1(\Abf), \]
with its induced $C(L) \otimes_{\Z} \OF$ action and polarization.
Since $\mathbf j^* (\widetilde \Abf) = \Abf$, by construction, the pullback $\mathbf j^*(\widetilde \bbU)$ is simply equal to $\bbU$, and the $C^+(\Lambda)$ action extends to an action of 
\[  \iota \colon (C(L)\otimes \OF)_{(p)}    \ \to \ \End(\mathbf j^* (\widetilde \bbU)) \]
via the identification $  (C(L) \otimes \OF)_{(p)}   \simeq C(\Lambda)_{(p)}  $ as in \eqref{eqn:BetaTildeDef}.

Consider the sheaf 
\[ \Lbb \ := \ \{  \varphi \in \End(\bbU) \ | \   \varphi^* = \varphi  \text{ and } \varphi \circ \iota(c \otimes a) = \iota(c \otimes a') \circ \varphi \ \text{ for all } c\otimes a' \in C(L) \otimes \OF  \}
\]
on $\Mcal$, as in \eqref{eqn:LDef}; it follows easily from definitions that the map
\begin{equation}
\Lbb \ \hookrightarrow \ \mathbf j^* \widetilde \Lbb, \qquad \varphi \mapsto \varphi \circ i(1 \otimes \varpi) 
\end{equation}
is an isometric embedding of sheaves on $\Mcal$, where we equip $\Lbb$ with the quadratic form $Q(\varphi) = \varphi^2$, and $\widetilde\Lbb$ with the form $\widetilde Q$ above. This induces an isomorphism of line bundles
\begin{equation}\label{eqn:relation-taut}
\omega^{\taut}  \, = \, \Fil^1(\Lbb) \ \isomto \  \mathbf j^* \widetilde\omega^{\taut} \, = \, \Fil^1(\mathbf j^* \widetilde \Lbb) .
\end{equation}
Finally, as for $\omega^{\taut}$, there is a hermitian metric on $\widetilde{\omega}^{\taut}$ defined by the same rule as \eqref{eqn:FritzMetric}. Clearly the isomorphism \eqref{eqn:relation-taut} becomes an isometry for these choices of hermitian metrics. 

\subsubsection{Complex uniformizations}\label{subsec:complex-unif-Siegel}
For later purposes, we would like to have a more concrete descriptions of these constructions over the complex fibre. As usual, fix an embedding $\nu \colon F \to \C$.

Attached to the lattice $\Lambda$ is the symmetric space
\[ 
\Dbb(\Lambda) \ := \ \{ \lambda \in \Lambda \otimes_{\Z} \C \ | \ \la\lambda , \lambda  \ra = 0 , \ \la \lambda  , \overline\lambda\ra < 0 \} \big/ \C^{\times} 
\]
which we may view as an open subset of the quadric of isotropic lines in $\mathbb P(\Lambda \otimes \C)$, and so is in particular a complex manifold; here $\la\cdot, \cdot\ra$ is the $\C$-linear extension of the symmetric bilinear form on $\Lambda$.

The space $\Dbb(\Lambda)$ can be interpreted as a $\widetilde \Gbf(\R) = GSpin(\Lambda)(\R)$-conjugacy class of Hodge structures in the following way. Given an isotropic line $\lie z \in \Dbb(\Lambda)$, choose a basis vector $x + iy$ where $x,y \in \Lambda\otimes \R$ with $\la x,y\ra=0$ and $Q(x) = Q(y)  = -1$, and define a map
\[ \tilde h_{\lie z} \colon \mathrm{Res}_{\C/\R} \mathbb G_m \ \to \ \widetilde \Gbf(\R) \]
determined (on $\R$ points) by sending $a + ib \in \C^{\times}$ to 
\[  \tilde h_{\lie z}(a+ib) \ = \ a \ + \ b xy \ \in  \ \widetilde \Gbf(\R) \subset C^+(\Lambda\otimes_{\Z}\R). \]
In particular, the pair $(\Dbb(\Lambda), \widetilde \Gbf)$ is a Shimura datum.  The double quotient space
\[
	\widetilde \Gbf(\Q) \big\backslash \Dbb(\Lambda) \times \widetilde \Gbf(\A_f) / \widetilde K 
\]
can be identified with the set of complex points $\widetilde \Mcal_{\nu}(\C)$ as follows.
Given a pair $[\lie z, \widetilde g] \in \Dbb(\Lambda) \times \widetilde \Gbf(\A_{\Q,f})$ , consider the tuple
\[ \underline{\widetilde A}_{[\lie z,\tilde g]} = \left(\widetilde A_{[\lie z,\tilde g]}, \tilde \iota, \lambda, [\widetilde \eta^p] \right)   \]
where 
\begin{itemize}
	\item $\widetilde A_{[\lie z, \tilde g]}$ is the complex abelian variety that, as a real torus, is the quotient
\[ \widetilde A_{[\lie z, \tilde g]} = \widetilde U \otimes_{\Z}{\R} \big/ \left( \widetilde g \cdot( \widetilde U_{\Z} \otimes \widehat\Z) \cap \widetilde U_{\Q} \right) \]
with complex structure given by right multiplication by $h_{\lie z}(i)$;
	\item the $C^+(\Lambda)$ action and polarization are naturally induced by those on $\widetilde U_{\Z}$;
	\item and $\widetilde \eta^p$ is the prime-to-$p$ part of $\widetilde g^{-1}$, viewed as a map
\[ \widetilde \eta^{p} = (\widetilde g^{-1})^p \colon \Ta^p(A_{[\lie z, \widetilde g]})\otimes \Q \simeq \widetilde g \cdot ( \widetilde U_{\Z} \otimes_{\Z} \A_{\Q,f}^p)  \isomto  \widetilde U_{\Z} \otimes_{\Z} \A_{\Q,f}^p. \]
\end{itemize}
Viewing this tuple up to prime-to-$p$ isogeny gives a map $\mathbb D(\Lambda) \times \widetilde \Gbf(\A_{\Q,f}) \to \widetilde \Mcal_{\nu}(\C)$ that descends to an isomorphism
\[ 
		\widetilde \Gbf(\Q) \big\backslash \Dbb(\Lambda) \times \widetilde \Gbf(\A_{\Q,f}) / \widetilde K  \ \isomto \ \widetilde \Mcal_{\nu}(\C).
\]
We note in passing, though we will not require this fact, that the fibre of the tautological bundle $\widetilde \omega^{\taut}$ at the point $[\lie z, \widetilde g]$ can be identified with the line $\lie z \subset \Lambda \otimes \C$. 

Similar considerations hold for $\Mcal_{\nu}(\C)$: let
\[ \Dbb(L) \ := \ \left\{   \ell \in  L \otimes_{\Z} \C \ | \   \la\ell,\ell\ra =0, \ \la \ell , \overline \ell\ra <0 \right\}\Big/ \C^{\times}, \]
which may be identified, in precisely the same way as above, as a $\Gbf(\R) = GSpin(\Lambda)(\R)$-conjugacy class of Hodge structures
\[
	h_{\lie z_0} \colon \mathrm{Res}_{ \C / \R} \mathbb G_m \ \to \ \Gbf_{/\R} , \qquad   h_{\lie z_0}(i) \ = \  xy
\]
where $\lie z_0 = \mathrm{span}(x+iy) \in \Dbb(L)$ with $x,y \in L\otimes \R$ such that $Q(x) = Q(y) = -1$ and $\la x,y\ra=0$.

This gives rise to an alternative
\footnote{To compare this uniformization to the previous one, first recall that the image of $L$ in $C(L)$ can be identified as 
\[ L \ \simeq \ \left\{ x \in C^-(L) \ | \ x^{\iota} = x   \right\}  \ = \ \left\{ x \in C(L) \ | \ x^{\iota}  = x, \, \delta x = - x \delta \right\} \]
and then observe that the map 
\[ X := \{ (z_1, z_2) \in \C^2 \ | \ Im(z_1)Im(z_2)>0 \} \ \isomto \ \Dbb(L), \qquad \zbf = (z_1, z_2) \mapsto  \mathrm{span}_{\C} \left\{  \left( (\begin{smallmatrix} z_1 & -z_1z_2 \\ -1 & - z_2 \end{smallmatrix}) , \, (\begin{smallmatrix} -z_2 & z_1z_2 \\ 1 & z_1 \end{smallmatrix})  \right)\cdot v_0 \right\} . \]
is an isomorphism compatible with the respective moduli interpretations. }
uniformization
\begin{equation}\label{eqn:alternative-complex-unif-Hilbert}
\Gbf(\Q) \backslash \Dbb(L) \times \Gbf(\A_{\Q,f}) / K  \ \isomto \ \Mcal_{\nu}(\C).
\end{equation}
Here, a pair $[\lie z_0, g]$ is mapped to (the prime-to-$p$ isogeny class of) the point 
\[ \underline A_{[\lie z_0, g]}  \ =  \ \left( A_{[\lie z_0, g]}, \, \iota, \, \lambda, \, [\eta^p] \right) \ \in \ \Mcal_{\nu}(\C) \]
 where 
\[
	A_{[\lie z_0,g]} \  =  \ U_{\R} \ \big/ \  g \cdot U_{\widehat \Z} \, \cap \, U_{\Q}
\]
with the complex structure given by right multiplication by $h_{\lie z_0}(i)$, the $C(L)\otimes\OF$-action $\iota$ and polarization $\lambda$ are induced by the respective structures on $U_{\Z}$, and 
\[ 
	\eta^p = (g^{-1})^p  \colon \Ta^p(A_{[\lie z_0, g]}) \ \simeq \ g\cdot (U_{\A_{\Q,f}^p}) \ \isomto \ U_{\A_{\Q,f}^p} .
\]
In these terms, the morphism $\mathbf j \colon \Mcal \to \widetilde \Mcal$ takes on a particularly simple form: the inclusion $L \hookrightarrow \Lambda$ induces a natural embedding
\[ j \colon \Dbb(L) \ \hookrightarrow \  \Dbb(\Lambda) \]
giving by simply viewing an isotropic line in $L\otimes \C$ as an isotropic line in $\Lambda \otimes \C$. There is also an induced embedding
\[ C(L) \subset C(\Lambda) \]
which, by definition of the $GSpin$ groups, gives an embedding that, abusing notation, we also denote
\[ j \colon  \mathbf G = GSpin(L) \ \hookrightarrow \  \widetilde{\Gbf} = GSpin(\Lambda).\]
Since $K \subset \widetilde K$, there is a map
\[ j \colon 	\Gbf(\Q) \backslash \Dbb(L) \times \Gbf(\A_{\Q,f}) / K \ \to \ \widetilde \Gbf(\Q) \backslash \Dbb(\Lambda) \times \widetilde\Gbf(\A_{\Q,f}) / \widetilde K. \]
On the other hand, suppose $\underline A_{(\lie z_0, [g])} \in \Mcal_{\nu}(\C)$ is a complex point corresponding to a pair $(\lie z_0, [g]) \in \Dbb(L) \times \Gbf(\A_{\Q,f})/K$. 
The isomorphism $\beta \colon U_{\Z} \isomto \widetilde U_{\Z}$ of \eqref{eqn:BetaDef} is equivariant under the right-multiplication action of $C^+(L)$, and so is in particular $GSpin(L)$-equivariant. Thus it induces an isomorphism
\[
	\beta \colon U_{\R} \big/ \left(g \cdot U_{\widehat \Z} \cap U_{\Q} \right)\ \isomto \ \widetilde U_{\R} \big/ \left( g \cdot \widetilde U_{\widehat\Z} \cap \widetilde U_{\Q} \right).
\]
such that
\[ \beta (x \cdot h_{\lie z_0}(i))\ = \ \beta(x)\cdot h_{\lie z_0}(i) \ = \ \beta(x) \cdot \widetilde h_{j(\lie z_0)}(i) , \]
i.e.\ $\beta$ induces an isomorphism between the complex abelian varieties $A_{(\lie z_0, [g])}$ and $\widetilde A_{j(\lie z_0, [g])}$. By tracking through the additional data in the modular definition of the map $\mathbf j \colon \Mcal \to \widetilde\Mcal$, one verifies that $\beta$ induces an isomorphism between
\[ \mathbf j \left( \underline A_{(\lie z_0, [g])} \right) \qquad \text{ and } \qquad \widetilde{ \underline A}_{j(\lie z_0, [g])}  \qquad \text{as objects of }\widetilde\Mcal_{\nu}(\C). 
\]
To summarize, the natural inclusion $\Dbb(L) \times \Gbf(\A_{\Q,f}) \hookrightarrow \Dbb(\Lambda) \times \widetilde{\Gbf}(\A_{\Q,f})$ induces a commutative diagram
\begin{equation} \label{eqn:GrassmannCommDiag} \begin{CD}
			\Mcal_{\nu}(\C) @>\mathbf j>>  \widetilde{\Mcal}_{\nu}(\C) \\
			@A\simeq AA    @A\simeq AA\\
			\Gbf(\Q) \backslash \Dbb(L) \times \Gbf(\A_{\Q,f}) / K @>>> \widetilde \Gbf(\Q) \backslash \Dbb(\Lambda) \times \widetilde\Gbf(\A_{\Q,f}) / \widetilde K.
	\end{CD}
\end{equation}

Finally, we record the following lemma, which can be proved by tracing through the definitions in this subsection; we leave this exercise to the reader.
\begin{lemma}  \label{lemma:VSharpCpx}
Suppose $\nu \colon F \to \C$ is a complex embedding, and  $\underline A  = \underline A_{(\lie z_0, [g])}\in \Mcal_{\nu}(\C)$ corresponds to the class of a pair $(\lie z_0, [g]) \in \mathbb D(L) \times \Gbf(\A_{\Q,f})/K$ under the complex uniformization \eqref{eqn:alternative-complex-unif-Hilbert}. Then 
\[ \Vbf^{\sharp}(\mathbf j(\underline A)) \ = \ \{ v \in V \cap g \cdot \widehat L , \ \la v, \lie z_0\ra = 0 \}  \otimes_{\Z} \Z_{(p)}.   \] \qed
\end{lemma}
\subsection{Computing $\ac_{1}(\overline{\Omega}_{\Mcal/\Scal})\ac_{2}(\overline{\Omega}_{\Mcal/\Scal})$}

The purpose of this section is to prove Theorem \ref{thm:c1c2thm}, by computing the arithmetic intersection number $\ac_{1}(\overline{\Omega}_{\Mcal/\Scal})\ac_{2}(\overline{\Omega}_{\Mcal/\Scal})$, 
via the Kodaira-Spencer comparisons of \S \ref{subsec:automorphic-vector}  and the theory of Borcherds forms, which we presently review, on the twisted Siegel threefolds of Section \ref{sec:Siegel-threefold}. 
\subsubsection{Borcherds forms}
In this section, we record a few results on Borcherds forms that will play a key role in the sequel.

 Working  more generally than in the previous sections, suppose that $\mathcal L$ is an integral quadratic lattice over $\Z$ of signature $(b^+,2)$, and set 
  \[ S_{\mathcal L} = \C[ \mathcal L'/\mathcal L], \]
 where $\Lcal' = \{ x \in \Lcal_{\Q}\ | \ \langle x, L \rangle \subset \Z \}$ is the dual lattice. Let $\{ \lie e_{\mu} \}$ denote the basis indexed by the cosets $\mu \in \mathcal L'/\mathcal L$.  We may embed $S_{\mathcal L}$ into the space of Schwartz functions $S(\mathcal L \otimes_{\Z} \A_{\Q,f})$ by sending 
 \[
 	\lie e_{\mu} \ \mapsto \ \mathbf 1_{\mu + \mathcal L \otimes \widehat\Z} \ = \ \text{char.\ function of } \mu \ + \ \mathcal L \otimes \widehat \Z .
 \]
 Throughout this paper, we will freely identify elements of $S_{\Lcal}$ with their images in $S(\Lcal \otimes \A_{\Q,f})$ without further comment. We also define 
 \[ S_{\Lcal}(\Q) \ := \ \Q[ \Lcal'/\Lcal] \]
 which we view as a subset of $S_{\Lcal} = S_{\Lcal}(\Q) \otimes_{\Q} \C$.  
 
Let $\omega_L \colon Mp_2(\Z) \to \Aut(S_{\Lcal})$ denote the action of the metaplectic group $Mp_2(\Z)$ via the Weil representation (see e.g.\ \cite[\S 1]{Brunier-Habil} for more precise definitions of this representation and the discussion below).
 For a half-integer $k \in \frac12 \Z$, let $M_k^!(S_{\Lcal})$ denote the space of \emph{weak holomorphic forms}: roughly speaking, elements of this space are (vector-valued) holomorphic functions
 \[ f \colon \lie H \to S_{\mathcal L} \]
that satisfy the ``weight-$k$" slash invariance property, i.e.\
\[
	(f|_k[\gamma])(\tau) \ = \ f(\tau) \qquad \text{ for all } \gamma \in Mp_2(\Z), 
\]	
 and are meromorphic at the cusp $\infty$.  In particular, each $f \in M_k^!(\rho_{\Lcal})$ has a Fourier expansion
\[ f(\tau) \ = \ \sum_{m \gg - \infty} c_f(m) \, q^m, \qquad \, c_f(m)\in S_{\mathcal L} \]
with only finitely many negative terms, where $q = e^{2 \pi i \tau}$. The sum is taken over $m \in \Q$, but in fact there is a bound, determined by $L$, on the size of the denominators of $m$ for which $c_f(m)$ could possibly be nonzero.

Another important family of vector-valued forms are the Eisenstein series, defined as follows: let $\mathtt{ev}_0 \in S_{\Lcal}^*$ denote the linear functional $\varphi \to \varphi(0)$, and consider, for a complex parameter $s$ with $Re(s) \gg 0$, the weight $k$ Eisenstein series 
\begin{equation} \label{eqn:GenEisDef}
 E_{\Lcal, k}(\tau,s) \ := \ \frac12 \ \sum_{\gamma \in \Gamma_{\infty} \backslash Mp_2(\Z) } \left. \left( Im(\tau)^{\frac{ s +1-k}{2}} \mathtt{ev}_0 \right)\right|_{k}[\gamma]   \ \in \ M_{2-k}(S_{\Lcal}^{\vee}).
\end{equation}

From this point on in this section, we assume that $b^+ \geq 3$. In this case, the special value $E_{\mathcal L,\kappa}(\tau, b^+/2 )$ of weight $\kappa := b^+ / 2 + 1$ is a holomorphic modular form, with Fourier expansion of the form
\begin{equation} \label{eqn:EisSeries} 
 E(\tau) \ := \  E_{\mathcal L,\kappa}\left(\tau, \frac{b^+}{2} \right) \ =  \ \mathtt{ev}_0 \ + \ \sum_{m \geq 0} c_E(m) \, q^m, \qquad c_E(m) \in S_{\Lcal}^{\vee}. 
\end{equation}
Fix a finite set $\Sigma$ of finite primes such that $\Lcal_{v} $ is self-dual for all $ v \notin \Sigma$, and for later convenience, assume $p \notin \Sigma$. For each $m$, there is a factorization	
	\[ c_E(m)(\mathtt{ev}_0)  \ = \ w_{\Sigma}(m) \ \times \ w^{\Sigma}(m) \ \in \ \Q \]
whose terms are products of special values of local Whittaker functions, see \cite[Proposition 2.6]{KY}. Though we will have more to say about these quantities later, it will suffice for our present purposes to make the following, somewhat crude, observations that can be extracted from \cite[\S 2 and \S 4]{KY}: if $m$ is squarefree and relatively prime to $\Sigma$, then
	\begin{itemize}
		\item the value $w_{\Sigma}(m)$ depends only on the residue class of $m$ modulo $T$, where $T$ is the product of the primes in $\Sigma$, and
		\item  $w^{\Sigma}(m)$ is, up to an overall non-zero constant factor, equal to a twisted divisor sum 
		\[  w^{\Sigma}(m) \ \sim \ \sum_{d | m} \, \chi^{(m)}(d) \, d^{b^+/2} \]
		for a certain  quadratic character $\chi^{(m)} \colon \Z \to \{ -1, 1 \}$. In particular, for such $m$ there is an overall constant $\kappa >0$, indepedent of $m$, such that 
		\[ |w^{\Sigma}(m)| \ \geq \ \kappa \cdot m^{b^+/2}. \]
	\end{itemize}
Let\footnote{We also remark in passing that the non-vanishing of $w_{\Sigma}(m)$ implies the non-vanishing of $c_E(m)(\mathtt{ev}_0)$. This is equivalent to the existence of a vector $x \in \Lcal \otimes \A_{\Q,f}$ with $Q(x) = m$; since the integral Hasse principle holds for quadratic lattices of dimension at least 5, this condition is in turn equivalent to the existence of a vector $x \in \Lcal$ with $Q(x) = m$. }
\[ I_{0} \ := \ \left\{ m \text{ squarefree} \ | \   (m, \Sigma) = 1 \text{ and } w_{\Sigma}(m) \neq 0  \right\}. \]
Combining the two observations above, we see that $I_0$ is an infinite set, and that there exists a constant $C>0$ such that 
\[ | c_E(m) (\lie e_{0}) | \ \geq  \ C \cdot m^{b^+/2} \qquad \qquad \text{for all } m \in I_0. \]

The following lemma can be proved in exactly the same way as \cite[Lemma 4.11]{BBK}, by using the Fourier coefficients of $E(\tau)$ in place of coefficients denoted $B_D(m)$ in \emph{loc.\ cit.} 
 
\begin{lemma} \label{lem:modForm}
Suppose  $b^+ \geq 3$, and fix a finite set of primes $\Sigma$ that includes the prime 2, and  such that $\mathcal L_v$ is self-dual for any $v \notin \Sigma$. Suppose that 
\[ I \ \subset \ \mathbb N \]
is any subset of integers such that
	\[ \#  (I \cap I_0) \ = \ \infty . \]
Then there exists a form $f \in M^!_{k}(S_{\mathcal L})$ of weight $k=1-b^+/2$ with Fourier expansion of the form
\[ f(\tau) \ = \ \sum_{\substack{m<0 \\ m \in I} } c_f(m) \, q^m \ + \ \sum_{m\geq 0} c_f(m) \, q^m \]
such that $c_f(-m) \in S_{\Lcal}(\Q)$ for all $m \leq 0$ and $c_f(0)(0) \neq 0$.  \qed
\end{lemma}

In his seminal paper \cite{Bor}, Borcherds uses weakly holomorphic forms to construct certain automorphic forms whose divisors lie along special cycles, and these results were extended to a $p$-integral context by H\"ormann \cite{Hor}. We state their results here for the only case in which we need it, but only to avoid introducing extraneous notation; the analogous statements hold for $GSpin(n,2)$ Shimura varieties at primes of good reduction.

Recall our lattice $\Lambda$ of signature $(3,2)$ from \S \ref{subsec:Siegel1-moduli} .
Since we had chosen the level structure $\widetilde{ K}$ such that 
\[ \widetilde { K} \ \subset \  C^+(\Lambda \otimes_{\Z} {\widehat\Z}{}^p)^{\times}, \]
it follows that $\widetilde{ K}$ stabilizes $\Lambda \otimes \A_{\Q,f}^p$ and permutes the cosets $\Lambda'/\Lambda$. View 
\[ S_{\Lambda} \ = \ \C[\Lambda' / \Lambda] \ \subset \ S(\Lambda \otimes {\A_{\Q,f}^p}) \]
via the embedding 
\[ \mu \in \Lambda' / \Lambda  \  \mapsto \ \text{characteristic function of } \mu \, + \, \Lambda \otimes_{\Z} \widehat\Z{}^p. \]
Then there is a rational basis $\{\varphi_1, \dots , \varphi_t \}$ for the $\widetilde{ K}$-invariant elements $S_{\Lambda}^{\widetilde{ K}}$, where each $\varphi_i$ is the characteristic function of a $\widetilde{K}$-invariant subset $\lie V_i \subset \Lambda \otimes \A_{\Q,f}^p$; concretely, we may take $\lie V_i$ to be a union of cosets $\mu + \Lambda \otimes \widehat\Z{}^p$ as $\mu$ ranges over an orbit for the action of $\widetilde{ K}$ on $\Lambda'/\Lambda$. If $\lie s \in S_{\Lambda}^{\widetilde{ K}}$ we write its components as 
\[ \lie s  \ = \ \sum_i \ \lie s(\varphi_i) \, \varphi_i . \]

\begin{theorem} \label{thm:BorcherdsForm}
Suppose $f \in M_{-1/2}^!(S_{\Lambda})$ satisfies
	\begin{equation} \label{eqn:intForm}
	 	c_f(m)(\varphi_i) \in 2 \Z \qquad \qquad \text{for all } m \leq 0 \text{ and } 1 \leq  i \leq t. 
	\end{equation}
Then there exists an integer $k >0$ and a  meromorphic section  
\[ \Psi(f)  \qquad \text{of} \qquad (\widetilde\omega^{\taut})^{\otimes k \cdot c_f(0)(0)/2} \]
 such that the following holds.
	\begin{enumerate}[(i)]
		\item  \cite[Theorem 13.3]{Bor} On the generic fibre, 
			\[ \mathsf{div} \Psi(f)_{F}  \ = \ \frac{k}{2}\sum_{m < 0} \ \sum_{i=1}^t \ c_f(m,\varphi_i) \cdot \left( \Zed(|m|, \lie V_i)_{F}\right).\]
			as cycles on $\widetilde \Mcal_{F}$. 
		\item  \cite[Theorem 3.2.14]{Hor} On the integral model $\widetilde \Mcal$, the divisor $\mathsf{div} \Psi(f)$ is equal to the closure of its generic fibre. In particular, 
			\[ | \mathsf{div} \Psi(f) | \ \subset \ \bigcup | \tilde\Zed(|m|,\lie V_i)| \]
			where the union is taken over pairs $(m, i)$ with $m < 0$ and $c_f(m)(\varphi_i) \neq 0$. 
	\end{enumerate}
\end{theorem}

Borcherds' proof of this theorem hinges on studying the function $- \log || \Psi_F||^2 $ on the complex points of $\widetilde \Mcal$; we briefly recall its description here. As usual, fix a complex embedding $\nu \colon F \to \C$. Recall that we had defined 
\[ \Dbb(\Lambda) \ := \ \left\{ \zeta \in \Lambda \otimes_{\Z} \C \ | \ \la\zeta, \zeta\ra = 0 , \ \la \zeta, \overline \zeta\ra < 0     \right\} / \C^{\times}.\]
If $\lambda \in \Lambda_{\Q}$ and $\lie z \in \Dbb(\Lambda)$ is an isotropic line, define
\[
	R(\lambda, \lie z) \ : = \ \frac{|\la\lambda, \zeta\ra|^2}{|\la \zeta, \overline \zeta\ra|} \qquad \text{ for any } \zeta \in \lie z ;
\]
note that $R(g \cdot \lambda, g \cdot \lie z) =  R(\lambda, \lie z)$ for any $g \in \widetilde\Gbf(\Q)$. 

Consider the \emph{Siegel theta function}
\[ \Theta_{\Lambda} \colon  \lie H \times \Dbb(\Lambda) \times \widetilde \Gbf(\A_f) \ \to \ S(\Lambda\otimes \A_{\Q,f})^{\vee} \]
defined by the formula
\[ \Theta_{\Lambda}(\tau, \lie z, \widetilde g) (\varphi) \ = \ v \cdot \sum_{\lambda \in \Lambda_{\Q} }  \varphi ( \widetilde g^{-1} \cdot \lambda) \, e^{-2 \pi v R(\lambda, \lie z)} \, e^{2 \pi i \tau Q(\lambda)}, \]
where $v = Im(\tau)$. If $\varphi \in S_{\Lambda}^{\widetilde K} \subset S(\Lambda \otimes \A_f)$, then it is easily seen that 
\[ \Theta_{\Lambda}(\tau, \,  \gamma \cdot \lie z, \, \gamma \cdot \widetilde g \cdot k)(\varphi) \ = \   \Theta_{\Lambda}(\tau, \, \lie z, \, \widetilde g)(\varphi) \qquad \text{ for all } \gamma \in \widetilde \Gbf(\Q) \text{ and } k \in \widetilde{ K}. \]
Thus, using the complex uniformization of the Siegel twisted threefold described in \S \ref{subsec:complex-unif-Siegel}, the Siegel theta function descends to a function
\[ \Theta_{\Lambda} \colon \lie H \times \widetilde\Mcal_{\nu}(\C) \ \to ( S_{\Lambda}^{\widetilde K})^{\vee}\]
that transforms like a weight $1/2$ (vector-valued) modular form in the $\tau$ variable (see, for example, \cite[\S 4]{Bor}). In particular, if $f \in M_{-1/2}(S_{\Lambda}^{\widetilde K})$ and $(\lie z, [\tilde g]) \in \widetilde\Mcal_{\nu}(\C)$ is fixed, the evaluation pairing
\[ \Theta_{\Lambda}\left(\tau,(\lie z, [\tilde g]) \right) \cdot f(\tau) \]
is an $SL_2(\Z)$-invariant function in the $\tau$-variable. Borcherds' key insight was to consider the ``regularized" integral of this function over $SL_2(\Z) \backslash \lie H$, defined as follows: for a complex parameter $s$ with $Re(s) \gg 0$ sufficiently large, the integral
\[
	I^{reg}_f(s, \lie z, [\widetilde g]) \ := \ 	\lim_{T \to \infty} \ \int_{\mathcal F_T} \, \Theta_{\Lambda}\left(\tau,(\lie z, [\tilde g]) \right) \cdot f(\tau) \cdot v^{-s} \ \frac{ du \, dv}{v^2},
\]
defines a holomorphic function in $s$ with a meromorphic continuation to all $s \in \C$; here $\tau = u + iv \in \lie H$ and 
\[ 
	\mathcal F_T \ := \ \left\{ \tau = u + iv \in \lie H \ | \ |\tau| \geq 1 , \  u \in [-\frac12, \frac12) , \  v \leq T \right\}
\]
is a truncation of the usual fundamental domain for the action of $SL_2(\Z)$ on $\lie H$.
 The regularized integral
\begin{equation} \label{eqn:IregDef}
	I^{reg}_f(\lie z, [\widetilde g])  := \  \mathrm{CT}_{s=0} \ I^{reg}_f(s,\lie z, [\widetilde g])
\end{equation}
is defined to be the constant term in the Laurent series expansion at $s=0$ of this continuation. After taking our normalization of the metric on $\widetilde\omega^{\taut}$ into account, we obtain:
\begin{proposition}[{\cite[Theorem 13.3 (iii)]{Bor}}] \label{prop:BorcherdsThetaLift} Let $\Psi_f$ be as in Theorem \ref{thm:BorcherdsForm}, and suppose that $(\lie z, [\widetilde g]) \in \widetilde \Mcal(\C)$ is a point that does not lie on the divisor $\mathtt{div}(\Psi_f)(\C)$. Then
\[ - \log || \Psi_f||_{(\lie z, [\widetilde g])}^2 \ = \ \frac{k}{2} \,  I^{reg}_{f}(\lie z, [\widetilde g]) \]
where $||\cdot ||$ is the metric on $(\widetilde\omega^{\taut})^{ \otimes k \cdot c_f(0)(0)/2}$ (see \S \ref{subsubsec:taut-siegel} and \eqref{eqn:FritzMetric}).
\qed
\end{proposition}

\subsubsection{Reduction via Borcherds forms and Kodaira-Spencer theory}

We are now in a position to put the pieces together and prove  \refThm{thm:c1c2thm}; it suffices to show that both sides of the equality are true in $\R_{(p)} = \R / \oplus_{q \neq p} \Q \cdot \log q$ for each fixed $p \nmid N$, and that we may work with the ``localized" moduli problems as in Remark \ref{rmk:postmainthm}(iii).  We continue with the notation in Section \ref{sec:Siegel-threefold}, so that $K \subset \Gbf(\A_{\Q,f})$ is a fixed sufficiently small compact open subgroup, $\lie p$ is a prime above a fixed prime $p$, and
\[ \Mcal \ = \ \Mcal_K \times_{\Z[1/N]} \Spec(\mathcal O_{F, (\lie p)}) , \qquad \Omega = \Omega_{\Mcal / \mathcal O_{F, (\lie p)}}, \qquad \text{etc.} \] 
Recall that we had constructed a morphism
\[ \mathbf j \colon \Mcal \ \to \ \widetilde \Mcal \]
where $\widetilde \Mcal $ is the twisted Siegel modular threefold. 

Let
\[ I \ = \ \{ m \in \Z \ | \ m \not\equiv \square \ \mod p \}; \] 
by Lemma \ref{lem:modForm}, there exists $f \in M_{-1/2}^!(S_{\Lambda}^{\widetilde{ K}})$ with Fourier expansion of the form
\[ f(\tau) \ = \ \sum_{\substack{m < 0 \\ m \in I}} \, c_f(m) q^m \ + \ c_f(0) \ + \ \dots  \]
with $c_f(0)(0) \neq 0$ and $c_f(-m) \in S_{\Lambda}(\Q)$ for all $m \leq 0$. 
Replacing $f$ by a sufficiently large integer multiple, if necessary, we may assume \eqref{eqn:intForm} holds, and so by Theorem \ref{thm:BorcherdsForm} we obtain a meromorphic section $\Psi = \Psi(f)$ of $(\widetilde{\omega}^{\taut})^{\otimes k \cdot c_f(0)(0)/2}$ for some integer $k\neq 0$, which we may further assume is even without loss of generality. For convenience, set 
\[ r  \ := \ k \cdot \frac{c_f(0)(0)}{2} ,\]
and note that $r$ is even by assumption.

Recall \eqref{eqn:relation-taut} that the pullback $ \mathbf j^* \widetilde{\omega}^{\taut}$ of the tautological bundle on $\widetilde \Mcal$ is identified with the tautological bundle $\omega^{\taut}$ on $\Mcal$. The following lemma implies that that pullback $\mathbf j^{\ast} \Psi$ defines a non-trivial rational section of $(\omega^{\taut})^{\otimes r}$.
\begin{lemma} The section $\Psi$ does not vanish identically along any connected component of $\Mcal$. 
\begin{proof}
Suppose otherwise. Then by considering complex points in particular, Theorem \ref{thm:BorcherdsForm} and the choice of $f$ defining $\Psi=\Psi(f)$ imply that there is some $\lie V \subset \Lambda \otimes \A_{\Q,f}^p$ and $m \in \Z_{(p)}$ with $m \not\equiv \square \pmod{p}$ such that the intersection
\[
	U \ := \ 	\mathbf j \left(\Mcal(\C)\right) \cap \widetilde \Zed(m, \lie V)(\C)
\]
contains a submanifold of complex dimension two.

First, we claim that there exists a point $z_0  \in \Mcal(\C)$, corresponding to an abelian variety $\underline A_0 = \underline A_{z_0}$, such that $j(z_0) \in U$ and $\Vbf^{\sharp}(\mathbf j(\underline A_0)) = \{ 0 \}$, cf.\ \eqref{eq:V-sharp}.

Indeed, if $z \in \Mcal(\C)$ corresponds to a pair $(\lie z, [g]) \in \mathbb D(L) \times \Gbf(\A_{\Q,f}) / K$ under the complex uniformization \eqref{eqn:alternative-complex-unif-Hilbert}, then Lemma \ref{lemma:VSharpCpx} implies that
\[ \Vbf^{\sharp}(\mathbf j (\underline A_z)) = \{ 0\} \ \iff \ \langle\lie z, v\rangle \neq 0 \text{ for all }  v \in V \cap g \cdot \widehat L.
\]
The union 
\[ 
\bigcup_{v \in V} \left\{ \lie z \in \mathbb D(L) \ | \  \langle\lie z, v\rangle = 0 \right\} 
\]
is a countable union of submanifolds of $\mathbb D(L)$ of complex dimension at most one, and so for dimension reasons, there exists a point $z_0 = (\lie z_0, [g_0] )$ mapping to $U$ such that $\lie z_0 $ is not in this union, and hence $\Vbf^{\sharp}(\mathbf j(\underline A_0)) = \{0\}$. 

We now arrive at a contradiction:  since $\mathbf j(\underline A_0)$ lands in the special cycle $\widetilde \Zed(m, \lie V)(\C)$ there exists an element of 
\[ \Vbf(\mathbf j(\underline A_0)) \ = \ \Vbf^{\flat}(\mathbf j(\underline A_0)) \]
of norm $m$. By Lemma \ref{lem:Vdecomp}(i), this implies that $m \in (\Z_{(p)})^2$, contradicting the assumption that $m \not\equiv \square \pmod{p}$. 
%
\end{proof}
\end{lemma}

We may therefore use the section $\mathbf j^{\ast} \Psi$ to represent $\ac_{1}(\overline{\Omega})$ by an arithmetic cycle:

\[ \frac{r}{2} \cdot  \ac_{1}(\overline \Omega) \ = \ \widehat{\mathsf{div}}(\mathbf j^* \Psi) \ = \left(\mathsf{div} (\mathbf j^* \Psi), \ - \log || \mathbf j^* \Psi ||_{P}^2 \right) \ \in \  \widehat{\mathsf{CH}}^{1}(\Mcal).  \]
Decompose $\mathtt{div}(\mathbf j^*\Psi) = \sum n_i Z_i$ for some prime divisors $Z_i$;
  if $\pi_i \colon Z_i \to \Mcal$ denotes the natural map, then
\begin{align}
\notag  \frac{r}{2} \cdot \widehat\deg \left( \ac_{1}(\overline \Omega)  \cdot \ac_{2}(\overline \Omega) \right) \ =& \ \sum_i n_i \, \widehat\deg\left( \ac_{2}( \pi_i^* \overline \Omega)  \right) \\ & \qquad \qquad + \frac{1}{2} \ \sum_{v \colon F \to \C} \ \int_{\Mcal_{v}(\C)} (-\log|| \mathbf j^* \Psi ||_{P}^2) \, c_2(\overline \Omega)  \label{eqn:Ch2deg},
\end{align}
where in the first line, we use the functoriality of Chern classes to identify $\pi_i^* \ac_{2}(\overline\Omega)= \ac_{2}(\pi_i^*\overline\Omega)$.
Similarly, 
\begin{align}
\notag \frac{r}{2}\cdot  \widehat\deg \left( \ac_{1}(\overline \Omega)^3 \right)   \equiv& \  \sum_i n_i \, \widehat\deg \left( \ac_{1}( \pi_i^* \overline \Omega)^2  \right) \\ & \qquad \qquad + \frac{1}{2} \ \sum_{v\colon F \to \C} \ \int_{\Mcal_{v}(\C)} (-\log|| \mathbf j^* \Psi ||_{P}^2) \, c_1(\overline \Omega)^{\wedge 2}  \label{eqn:Ch1deg}.
\end{align}

Consider the orthogonal decomposition of metrized vector bundles $\overline{\Omega}= \overline{\Lcal}_1 \oplus \overline{\Lcal}_2$, according to \eqref{eqn:KSmap} and Proposition \ref{prop:OmegaCpxDecom}. Restricting via $\pi_i \colon  Z_i \to\Mcal$ gives
\[ \pi_i^* \overline \Omega \  = \ \pi_i^* \overline \Lcal_1 \ \oplus \ \pi_i^* \overline \Lcal_2. \]
The multiplicativity of the total arithmetic Chern class on orthogonal sums gives
	\[ \ac_2( \pi_i^* \overline \Omega) \ = \ac_{1}(\pi_i^* \overline \Lcal_1) \cdot \ac_{1}(\pi_i^* \overline \Lcal_2) .\]
Now Proposition \ref{prop:SCBundle} implies that 
	\[ \ac_{1}(\pi_i^* \overline \Lcal_2) \ = \ \ac_{1}(\pi_i^* \overline \Lcal_1) \ + \ (0, \mathbf c) \] 
where $\mathbf c$ is a locally constant function on $Z_i \subset \mathtt{div}(\mathbf j^* \Psi)$ valued in $\log|\Z_{(p)}^{\times}|$. Thus, working in $\R_{(p)} = \R / \oplus_{q \neq p} \Q \log q$, we find
	\begin{align*}
		\widehat\deg\left( \ac_{2}( \pi_i^* \overline \Omega) \right) \ =& \ \widehat\deg \left( \ac_{1}( \pi_i^* \overline \Lcal_1)^2 \right) \ + \ \widehat\deg \left( \ac_{1}( \pi_i^* \overline \Lcal_1) \cdot (0, \mathbf c) \right)  \\
		\equiv &  \ \widehat\deg\left( \ac_{1}( \pi_i^* \overline \Lcal_1)^2 \right) \ \in \ \R_{(p)}.
	\end{align*}
Similarly, 
	\[ \ac_{1}(\pi_i^*\overline \Omega)^2 \ =\  \left( \ac_{1}(\pi_i^*\overline \Lcal_1)  \ + \ \ac_{1}(\pi_i^*\overline \Lcal_2) \right)^2 \ = \ \left(2  \ac_{1}(\pi_i^*\overline \Lcal_1) + (0, \mathbf c) \right)^2 \]
and so
	\begin{align*}
		\widehat\deg(\ac_{1}(\pi_i^*\overline \Omega)^2) \ =& \ 4 \cdot \widehat\deg \left( \ac_{1}( \pi_i^* \overline \Lcal_1)^2 \right) \\ 
	=& \ 4 \cdot \widehat\deg \left( \ac_{2}( \pi_i^* \overline \Omega) \right)  \ \in \ \R_{(p)}.
	\end{align*}
Finally, recall from Proposition \ref{prop:OmegaCpxDecom} that $c_1(\overline \Omega)^{\wedge 2} = 2 \cdot c_2(\overline{\Omega})$. Thus, comparing \eqref{eqn:Ch2deg} and \eqref{eqn:Ch1deg} implies the equality, in $\R_{(p)}$, of
	\begin{equation}\label{eqn:first-reduction}
	4 \cdot \widehat\deg \left(	\ac_{1}(\overline \Omega)  \cdot \ac_{2}(\overline \Omega) \right) \ \equiv \  \, \widehat\deg \left( \ac_{1}(\overline \Omega)^3 \right)\ + \ \frac{1}{r} \cdot \sum_{v \colon F \to \C} \ \int_{\Mcal_{v}(\C)} (-\log|| \mathbf j^* \Psi ||_{P}^2) \, c_1(\overline \Omega)^{\wedge 2}.
	\end{equation}
To conclude the proof of Theorem \ref{thm:c1c2thm} we compute both terms in the right hand side of this equality. The first quantity was essentially computed by Fritz H\"ormann:

\begin{theorem}[H\"ormann \cite{Hor}] The following equality holds in $\R_{(p)} = \R / \oplus_{q \neq p} \Q \log q$:
\begin{align*}
\widehat\deg & \left(\ac_{1}(\overline\Omega)^3 \right) \ \\
=& \ \deg \left( c_1(\Omega_{\Mcal/F} )^2 \right) \cdot \left( -12 \log \pi -6 \gamma + 3  + 2 \frac{\zeta_F'(2)}{\zeta_F(2)} + 2 \frac{\zeta_{\Q}'(2)}{\zeta_{\Q}(2)} \right) 
\end{align*} 

\begin{proof}
By Proposition \ref{prop:TautOmega}, 
\[ 
	2 \ac_{1}(\overline{\omega^{\taut}})  \ = \ \ac_{1}(\overline\Omega) \ - \ (0, \log \mathsf c),
\]
where $\mathsf c = e^{-2 \gamma} \cdot D_B^4 \pi^{-6} (64)^{-3}  $. Note that 
\[
	(0, \log \mathsf c)^2  \ = \ 0 \in   \Ch{2}(\Mcal)
\]
and
\[
	\widehat\deg \left( \ac_{1}(\overline\Omega)^2 \cdot (0, \log\mathsf c) \right) \ = \ \frac{ \log \mathsf c}{2} \cdot \deg \left( c_1(\Omega_{\Mcal(\C)} )^2 \right)
\]
so
\begin{equation} \label{eqn:ch3relation}
	\widehat\deg\left(	\ac_{1}(\overline\Omega)^3 \right)  \ = \ 8 \ \widehat\deg\left(\ac_{1}(\overline{\omega^{\taut}})^3\right) \ + \ \frac32 \log \mathsf c \cdot \deg \left( c_1(\Omega_{\Mcal(\C)} )^2 \right).
\end{equation}
The main result of \cite{Hor} is a computation of the degree of $\ac_{1}(\overline{\omega^{\taut}})^{n+1}$ for orthogonal Shimura varieties of signature $(n,2)$ for any $n$. His result \cite[Theorem 3.5.5]{Hor}, taken in $\R_{(p)}$ and translated to our language in the case at hand, reads
\[
	\widehat\deg\left(\ac_{1}(\overline{\omega^{\taut}})^3\right) \ \equiv \  \ \deg \left( c_1(\omega_{/F}^{\taut} )^2 \right) \cdot \left. \frac{ \frac{d}{ds} \lambda^{-1}(L,s)}{\lambda^{-1}(L,s)} \right|_{s=0} \in \R_{(p)},
\]
where $\lambda^{-1}(L,s)$ is a function of the form
\[
	\lambda^{-1}(L,s) \ = \ \frac18 \, \frac{\Gamma(\frac{s}{2} + 1) \Gamma(\frac{s}{2} +\frac32)\Gamma(\frac{s}{2}  + 2)}{\pi^{\frac32(s+3)}}  \cdot \zeta_{\Q}(2+2s) \cdot L\left(2+s, \left(\frac{\det L}{\cdot} \right) \right) \cdot \left( \begin{matrix} \text{Euler factors} \\ \text{at primes } \ell \mid N \end{matrix} \right);
\]
see \cite[Appendix B]{Hor} as well as \cite[\S 8]{Hor-padic} for more details.
 Here $\det(L)$ is the determinant of the matrix of inner products
\[
	(\la x_i, x_j \ra)_{i,j=1,\dots 4}
\]
of any $\Z$-basis $\{ x_1, \dots, x_4\}$ of $L$.
 It follows from \eqref{eqn:Q0DiagForm} that $\det L \equiv \dF \pmod{\Q^{\times,2}}$, and since both $\det L$ and $\dF$ divide $N$,  we have
\[ \left( \frac{\det L}{x} \right) \ = \ \left(   \frac{\dF}{x} \right) \  =:\ \chi_F(x), \]
for all $x\in \Z$ with $(x,N) = 1$. 
 Therefore, we may write
 \[
 	\lambda^{-1}(L,s) \ = \ \frac18 \,  \frac{\Gamma(\frac{s}{2} + 1) \Gamma(\frac{s}{2} +\frac32)\Gamma(\frac{s}{2}  + 2)}{\pi^{\frac32(s+3)}}  \cdot \zeta_{\Q}(2+2s) \cdot L\left(2+s, \chi_F \right) \cdot \left( \begin{matrix} \text{Euler factors} \\ \text{at primes } \ell \mid N \end{matrix} \right).
 \]
Consider the logarithmic derivative of this expression at $s=0$. The contributions from the Euler factors at $\ell | N$ vanish in the quotient $\R \to \R_{(p)}$; computing the remaining terms, and using the relation $\zeta_F(s) = \zeta_{\Q}(s) \cdot L(s, \chi_F)$, gives 
\[
\left. \frac{ \frac{d}{ds} \lambda^{-1}(L,s)}{\lambda^{-1}(L,s)} \right|_{s=0} \equiv \ \frac{3}{2} \left( - \log \pi - \gamma + 1 \right) + \frac{\zeta_F'(2)}{\zeta_F(2)} \ + \ \frac{\zeta_{\Q}'(2)}{\zeta_{\Q}(2)} \ \in \ \R_{(p)}.
\]
Inserting this expression back into \eqref{eqn:ch3relation}, using the fact that $4 c_1(\overline{\omega}_{/F}^{\taut})^{ 2} = c_1(\overline{\Omega_F})^{ 2}$, and discarding terms that vanish in the quotient $\R = \R_{(p)}$ gives the theorem after a little straightforward algebra.
\end{proof}
\end{theorem}

\subsubsection{Integrals of Borcherds forms along $\Mcal$}
We proceed to complete the proof of Theorem \ref{thm:c1c2thm}, by evaluating the integral \eqref{eqn:first-reduction}:

\begin{theorem} The following equality holds in $\R_{(p)}$:
		\begin{align*}
			 \sum_{v \colon F \to \C} \int\limits_{\Mcal_{v}(\C)} \ - \log|| \mathbf j^* \Psi & ||_{P}^2 \, c_1(\overline{\Omega})^{\wedge 2} \\
			 & \equiv \ \deg\left( c_1(\Omega_{\Mcal(\C)})^2 \right) \cdot r \cdot \left(   -2 \gamma   - 4 \log \pi + 1 + 2 \frac{\zeta'_F(2)}{\zeta_F(2)} - 2 \frac{ \zeta'(2)}{\zeta(2)}  \right) 
		\end{align*}
\begin{proof}
This proof is a variation on the method of \cite{Kudla-IBF}. To start, fix an embedding $v \colon F \to \C$ and use Proposition \ref{prop:TautOmega} to relate the metrics on $\det\Omega$ and $(\omega^{\taut})^{\otimes 2}$ to obtain
\begin{equation} \label{eqn:logMetricComp}
 - \log	|| \mathbf{j}^* \Psi||_{P}^2 \ =\ - \log || \mathbf{j}^* \Psi||_{(\omega^\taut)^{\otimes r}}^2 \ + \ \frac{r}{2} \left( - 2 \gamma   + \log \left(\frac{D_B^4}{(64)^3 \pi^6}\right) \right).
\end{equation}
On the other hand, using the isometry $\mathbf j^* \widetilde \omega^{\taut} \simeq \omega^{\taut}$ as in \eqref{eqn:relation-taut} and then applying Proposition \ref{prop:BorcherdsThetaLift}, we have that
\begin{equation} \label{eqn:logRegIntComp}
	- \log ||\mathbf j^* \Psi||^2_{(\omega^{\taut})^{\otimes r}} \ = \ \frac{k}{2} \ \mathbf j^*( I^{reg}_f )
\end{equation}
where both sides are viewed as functions on the complement of $\mathtt{div}(\mathbf j^* \Psi)(\C)$ in $\Mcal(\C)$; on the right appears the pullback to $\Mcal(\C)$ of the regularized theta lift
\[ 
	I^{reg}_f(z) \ = \ \mathrm{CT}_{s=0} \lim_{T \to \infty} \int_{\mathcal F_T} \Theta_{\Lambda} (\tau, z) \cdot f(\tau) \ \frac{du \, dv}{v^{2+s}}, \ \qquad z \in \widetilde \Mcal(\C)
\]
as in \eqref{eqn:IregDef}. 

We first consider the pullback of the Siegel theta function. By definition,
\[ \Lambda \ = \ L  \  \oplus \ \Z \mathbf e; \]
since $Q(\mathbf e) = 1$, the dual lattice to $\Z \mathbf e$ is $\frac12 \Z \mathbf e$, and so
\[ S_{\Lambda}  \ = \ \C[\Lambda'/\Lambda] \ \simeq \ S_L \otimes S_{\Z \mathbf e}, \]
where $S_L = \C[L'/L]$ and $S_{\Z \mathbf e} = \C[ \frac12 \Z  \mathbf e / \Z\mathbf e] \simeq \C[ \Z / 2]$. Note that inside of $C(\Lambda)$, the element $\mathbf e$ commutes with $C^+(L)$, and so  $\Gbf(\A_{\Q,f})$ fixes $\mathbf e$. In particular, since we assumed $ K \subset \widetilde{K}$, 
\begin{equation} \label{eqn:SLambdaTensor}
	S_{\Lambda}^{\widetilde{ K}} \ \subset \ S_{L}^{ K} \ \otimes \ S_{\Z \mathbf e}.
\end{equation}

Just as was the case with $\Lambda$,  we may consider the (weight zero) Siegel theta function 
\[
\Theta_{L}  \colon \lie H \times \mathbb D(L) \times \Gbf(\A_{\Q,f}) \ \to \ (S_L^{K})^{\vee}, \qquad \Theta_{L}(\tau, \lie z_0, g ) (\varphi)  \ := \  v \, \sum_{\ell \in L_{\Q}}  \varphi(g^{-1}\ell) \, e^{-2 \pi v R_0(\ell, \lie z_0)} \, e^{2 \pi i \tau Q(\ell)} 
\]
attached to $L$;
here  
\[ \Dbb(L) \ := \ \{ \lie z_0 \in  L \otimes_{\Z} \C \ | \ \la \lie z_0, \lie z_0 \ra = 0, \ \la \lie z_0, \overline{\lie z_0} \ra < 0     \} / \C^{\times} \]
and
\[ R_0(\ell, \lie z_0) \ := \ \frac{ | \la \ell, \, \zeta_0 \ra |^2}{| \la \zeta_0, \, \overline{\zeta_0} \ra|} \qquad \text{ for any } \zeta_0 \in \lie z_0. \]
This theta function satisfies $\Theta_L(\tau,\, \gamma \cdot \lie z_0, \, \gamma \cdot g \cdot k) = \Theta_L(\tau, \lie z_0, g)$ for all $\gamma \in \Gbf(\Q)$ and $k \in  K$, and so descends to a function
\[ 
\Theta_{L}  \colon \lie H \times \Mcal(\C) \ \to \ (S_L^{K})^{\vee} 
\]
via the complex uniformization \eqref{eqn:alternative-complex-unif-Hilbert}. 

On the other hand, the theta function attached to $S_{\Z \mathbf e}$ is simply the (vector-valued) Jacobi theta function of weight 1/2:
\[ 
\Theta^{Jac} \colon \lie H \to S_{\Z \mathbf e}^{\vee}, \qquad \Theta^{Jac}(\tau)(\varphi) \ = \ \sum_{x \in \Q \mathbf e} \varphi(x) \, e^{2 \pi i Q(x) \tau};  
\]
taking $\varphi = \mathbf 1_{\widehat\Z \mathbf e} \in S_{\Z \mathbf e} $ to be the characteristic function of $\widehat \Z \mathbf e$, we obtain the usual (scalar-valued) Jacobi theta function
\[ 
	\Theta^{Jac}(\tau)( \mathbf 1_{\widehat\Z \mathbf e}) \ = \ \sum_{x \in \Q \mathbf e}  \mathbf 1_{\widehat\Z \mathbf e}(x) \, e^{2 \pi i Q(x) \tau} \ = \ \sum_{x \in \Z \mathbf e} e^{2 \pi i Q(x) \tau} \ = \ \sum_{n \in \Z} e^{2 \pi i n^2 \tau}.
\]

 Now recall from  \eqref{eqn:GrassmannCommDiag} that the morphism $\mathbf j_{v} \colon \Mcal_{v}(\C) \to \widetilde\Mcal_{v}(\C)$ is given, on the level of complex uniformizations, by the natural map
\[
 j \colon 	\Mcal_{v}(\C) \simeq \Gbf(\Q) \backslash \Dbb(L) \times \Gbf(\A_{\Q,f}) /  K  \ \to \ \widetilde \Gbf(\Q) \backslash \Dbb(\Lambda) \times \widetilde\Gbf(\A_{\Q,f}) / \widetilde{ K} \simeq \widetilde\Mcal_{v}(\C)
\]
induced by the inclusion $\mathbb D(L) \times \Gbf(\A_{\Q,f}) \hookrightarrow \mathbb D(\Lambda) \times \widetilde \Gbf(\A_{\Q,f})$.

Note in particular that if 
\[ \lambda \ = \  \ell + a \mathbf e  \ \in \ \Lambda_{\Q}  = L_{\Q} \oplus \Q \mathbf e \]
 and $\lie z_0 \in \mathbb D(L) \subset \mathbb D(\Lambda)$, then
\[ R(\lambda, \lie z_0) \ = \ R(\ell, \lie z_0) \ = \ R_0(\ell, \lie z_0). \]
Thus, for $\lie z_{0}\in \mathbb D(L)$, $g\in\Gbf(\A_{\Q,f})$ and $\varphi  \ =  \ \varphi_0 \otimes \varphi_1  \ \in \ S_{L}^{K} \otimes S_{\Z \mathbf e} $
we have
\begin{align*}
j^* \Theta_{\Lambda}(\tau, (\lie z_0, [g]))(\varphi) \ 
=& \ v \, \sum_{\lambda \in \Lambda_{\Q}} \varphi(g^{-1}\lambda) \, e^{-2 \pi v R(\lambda, \lie z_0)} \, e^{2 \pi i \tau Q(\lambda)}  \\
=& \ v \, \sum_{\substack{\ell \in L_{\Q} \\ x \in \Q \mathbf e }} \varphi(g^{-1}(\ell) + x) \, e^{-2  v R(\ell + x, \lie z_0)} \, e^{2 \pi i \tau Q(\ell + x)} \\
=& \ v \, \sum_{\substack{\ell \in L_{\Q} \\ x \in \Q \mathbf e }} \varphi_0(g^{-1}\ell) \varphi_1( x) \, e^{-2  v R_0(\ell, \lie z_0)} \, e^{2 \pi i \tau (Q(\ell)+ Q(x))} \\
=& \ \Theta_L(\tau, (\lie z_0, [g]) ) (\varphi_0)  \cdot  \Theta^{Jac}(\tau)(\varphi_1).
\end{align*}
More succinctly, we may write
\begin{equation}\label{eqn:factorization}
	j^*\Theta_{\Lambda}(\tau, \cdot) \ = \ \Theta_L(\tau, \cdot) \ \otimes \ \Theta^{Jac}(\tau)
\end{equation}
where we view the right hand side as a linear functional on $S_{\Lambda}^{\widetilde{ K}}$ via restriction in \eqref{eqn:SLambdaTensor}.

The next step is to work out the integral of $j^* I^{reg}_f $ over $\Mcal_{v}(\C)$, with the help of the factorization formula \eqref{eqn:factorization} and the Siegel-Weil formula. Note that the set $\mathtt{div}(\mathbf j^* \Psi)(\C)$ has measure zero and we ignore it in the sequel. For convenience let $d \Omega(z) $ denote the measure on $\Mcal_{v}(\C)$ induced by $c_1(\overline \Omega)^{\wedge 2}$. Arguing exactly as in \cite[\S 3]{Kudla-IBF}, we may interchange the integral over $\Mcal_{v}(\C)$ with the regularization integral:
\begin{align*}
	\int_{\Mcal_{v}(\C)} j^* I_f^{reg}(z)  \, d \Omega(z) \ =& \ 	\int_{\Mcal_{v}(\C)} CT_{s=0} \lim_{T \to \infty} \int_{\mathcal F_T} j^* \Theta_{\Lambda}(\tau, z) \cdot f(\tau)  \frac{du \, dv}{v^{2+s}} \, d \Omega(z) \\
	=&  \  CT_{s=0} \, \lim_{T \to \infty} \, \int_{\mathcal F_T}   \,\left( \int_{\Mcal_{v}(\C)} j^* \Theta_{\Lambda}(\tau, z) \, d \Omega(z)     \right) \cdot f(\tau)  \frac{du \, dv}{v^{2+s}}  \\
	=& \ CT_{s=0} \, \lim_{T \to \infty} \, \int_{\mathcal F_T} \,\left( \int_{\Mcal_{v}(\C)} \Theta_{L}(\tau, z) \, d \Omega(z)  \ \otimes \ \Theta^{Jac}(\tau)   \right)  \cdot f(\tau)  \frac{du \, dv}{v^{2+s}};
\end{align*}
these manipulations are justified in part by the fact that the Shimura variety $\Mcal_{v}(\C)$ is not an ``exceptional case" in the terminology of \emph{loc.\ cit.}, since the quadratic space $L_{\Q}$ is anisotropic. The \emph{Siegel-Weil formula} computes the inner integral
\[
	 \int_{\Mcal_{v}(\C)} \Theta_{L}(\tau, z) \, d \Omega(z) \ = \ \vol(\Mcal_{v}(\C), d \Omega(z)) \cdot E_{L,0}(\tau,1)
\]
where $E_{L,0}(\tau, 1)$ is the value at $s=1$ of the weight zero Eisenstein series \eqref{eqn:GenEisDef};
implicit in this statement is the fact that $E_{L,0}(\tau,s)$ is holomorphic in $s$ at the point $s=1$.

By \cite[(2.17)]{Kudla-IBF}, 
\begin{equation} \label{eqn:EisLowering}
	-2  \, i\, v^2 \, \frac{\partial}{\partial \overline{\tau}} E_{L, 2}(\tau,s) = \frac12( s - 1) E_{L,0}(\tau,s).
\end{equation}
Comparing Taylor expansions in $s$ at $s=1$ on both sides, we find that $(i)$  $E_{L,2}(\tau,1)$ is a holomorphic modular form in the $\tau$ variable, and $(ii)$ if $E'_{L,2}(\tau,s)$ is the derivative with respect to $s$, then
\[
	- 2 \, i \, v^2 \   \frac{\partial}{\partial \overline{\tau}} \,  E'_{L,2}(\tau, 1)   \ =  \ \frac12 \ E_{L,0}(\tau, 1)  ,
\]
and so, since $\Theta^{Jac}(\tau)$ is holomorphic,
\begin{equation} \label{eqn:StokesEis}
	- 2 \, i \, v^2 \   \frac{\partial}{\partial \overline{\tau}} \, \left( E'_{L,2}(\tau, 1) \otimes \Theta^{Jac}(\tau)  \right) \ = \ \frac12 \  E_{L,0}(\tau, 1) \otimes \Theta^{Jac}(\tau).
\end{equation}
Let $\lie G(\tau) :=  E'_{L,2}(\tau, 1) \otimes \Theta^{Jac}(\tau)$ and write its Fourier expansion as
\[ 
  \lie G(\tau) \ = \ 	\sum_{n \in \Q} c_{\lie G}(n,v) \, e^{2 \pi i n \tau}, \qquad c_{\lie G}(n,v) \in (S_{\Lambda}^{\widetilde{\mathbf K}})^{\vee}.
\]
Now applying the same argument as \cite[\S 2]{Kudla-IBF}, with \eqref{eqn:StokesEis} playing the role of Equation (2.9) of \emph{loc.\ cit.}, we obtain
\begin{align}
\frac1{2 \, \vol(\Mcal_{v}(\C), \, d \Omega) } 	\int_{\Mcal_{v}(\C)} \ & j^* I^{reg}_f(z) d \Omega(z) \notag \\
 =& \ \sum_{m \neq 0} \left( \lim_{T \to \infty} c_{\lie G}(m, T) \cdot c_f(-m) \right) \ + \ \lim_{T \to \infty} \left( c_{\lie G}(0,T) \cdot c_f(0) \ - \ \frac{c_f(0)(0)}{2} \log T \right). \notag
\end{align}
Our next task is to understand the coefficients 
\begin{equation} \label{eqn:FCThetaTensor}
 c_{\lie G}(m, T) \ = \ \sum_{n \in \Q} c_{E'}(n,T) \otimes c_{\Theta^{Jac}}(m-n)
\end{equation}
where, for convenience, we let $c_{E'}(n,v)$ denote the $n$-th Fourier coefficient of $E'_{L,2}(\tau,1)$. By  Proposition \ref{prop:EisFC} (i) below, 
\[ 
	\lim_{T \to \infty} c_{E'_{2}}(n,T) = 0 \qquad \text{ whenever } n < 0
\]
which, as $\Theta^{Jac}(\tau)$ has only non-negative Fourier coefficients, implies that
\[ \lim_{T \to \infty} c_{\lie G}(m, T) \  = \ 0 \ \qquad \text{ whenever } m < 0 \]
as well. Similarly, combining Proposition \ref{prop:EisFC} (ii) with the fact that
\[ 
c_{\Theta^{Jac}}(0) (\varphi_1) \ = \ \varphi_1(0) , \qquad \text{ for any } \varphi_1 \in  S_{\Z \mathbf e} ,
\]
implies that
\[ 
\lim_{T \to \infty} \left[ c_{\lie G}(0,T) \cdot c_f(0) \ - \ \frac{c_f(0)(0)}{2} \log T \right] \ = \ 0 .
\]
Thus, we have now shown that
\begin{align*}
	\frac1{2 \, \vol(\Mcal_{v}(\C), \, d \Omega) } 	\int_{\Mcal_{v}(\C)} \  j^* I^{reg}_f(z) d \Omega(z) \ =& \ \sum_{m > 0}  \lim_{T \to \infty} c_{\lie G}(m, T) \cdot c_f(-m) \\
	=& \  \sum_{m > 0}  \sum_{n \leq m} \left(  \lim_{T \to \infty} c_{ E'_2}(n, T) \otimes c_{\Theta^{Jac}}(m-n) \right) \cdot c_f(-m)
\end{align*}
Recall that $f$ was chosen so that $c_f(-m) \in S_{\Lambda}(\Q) \subset S_L(\Q) \otimes S_{\Z \mathbf e}(\Q)$, and the only non-zero terms appear when $m \not\equiv \square \pmod p$. On the other hand, if $m - n \notin (\frac12 \Z)^2$, then $c_{\Theta^{Jac}}(m-n) = 0$. Thus, all the $n$'s appearing in the last display are elements of $\Z_{(p)}$ with $n \not\equiv 0 \pmod{p}$. Setting 
\[
g(\tau) := E_{L,2}(\tau,1) \otimes \Theta^{Jac}(\tau) \qquad \text{and} \qquad A :=   \frac12 \log \pi \ + \  \frac12 \gamma - \frac12 \ - \ \frac{\zeta_F'(2)}{\zeta_F(2)} \ + \ \frac{\zeta'(2)}{\zeta(2)} ,
\]
Proposition \ref{prop:EisFC} (iv) below implies the following equality in $\R_{(p)}$:
\[
	\frac1{2 \, \vol(\Mcal_{v}(\C), \, d \Omega) } 	\int_{\Mcal_{v}(\C)} \  j^* I^{reg}_f(z) d \Omega(z) \ \equiv \ A \cdot \left( \sum_{m > 0} c_g(m) \cdot c_f(-m) \right) \ \in \ \R_{(p)}.
\]
But $ g(\tau) \cdot f(\tau)$ is an $SL_2(\Z)$-invariant function that is holomorphic on the upper half-plane $\lie H$. Integrating around the closed contour $\partial \mathcal F_1$, where
\[
	\mathcal F_1 \ = \ \left\{\tau \in \lie H \ |  \ Re(\tau) \in \left[ -\frac12, \frac12\right] , \ Im(\tau) \leq 1 , \  |\tau| \geq 1  \right\},
\]
gives
\[ 0 = \int_{\partial \mathcal F_1} \, g(\tau) \cdot f(\tau) \, d \tau \ = \ \sum_{m \geq 0} c_g(m) \cdot c_f(-m); \]
thus
\[ \sum_{m>0} c_g(m) \cdot c_f(-m) \ = \ - c_g(0) \cdot c_f(0) \ = \ - c_f(0)(0), \]
where the last equality follows from the fact that $c_g(0) = \mathrm{ev}_0$. 
This implies
\[ 
	\frac1{2 \, \vol(\Mcal_{v}(\C), \, d \Omega) } 	\int_{\Mcal_{v}(\C)} \  j^* I^{reg}_f(z) d \Omega(z) \ \equiv \ - A \ c_f(0)(0) \ \in \ \R_{(p)}.
\]
The theorem follows, after a little algebra, upon combining this last expression with \eqref{eqn:logMetricComp} and \eqref{eqn:logRegIntComp}, applying the identity
\[ \sum_v \, \vol(\Mcal_v(\C), c_1(\Omega_{\Mcal_v(\C)})^{\wedge 2})  \ = \ \deg(c_1(\Omega_{\Mcal(\C)})^2), \]
 and finally disregarding contributions that vanish in the quotient $\R \to \R_{(p)}$.
\end{proof}
\end{theorem}

It remains to prove the following proposition.

\begin{proposition} \label{prop:EisFC} Write the Fourier expansions of the special value and derivative of the Eisenstein series $E_{L,2}(\tau, s)$ as 
\[
	E_{L,2}(\tau,1)  \ = \ \sum_{n \geq 0} \, c_{E_2}(n) \, e^{2 \pi i n \tau} \qquad \text{and}  \qquad E'_{L,2}(\tau,1) \ = \ \sum_{n \in \Q} c_{E'_2}(n,v) \, e^{2 \pi i n \tau} 
\]
with $c_{E_2}(n), c_{E_2'}(n,v) \in S_L^{\vee}$. 
\begin{enumerate}[(i)]
\item If $n<0$ then $\lim_{v \to \infty} c_{E'_2}(n,v) = 0$.
\item $ \lim_{v \to \infty} \left(  c_{E'_2}(0,v) - \frac12 \log v \cdot \mathtt{ev}_0 \right) = 0 $.
\item If $\varphi \in S_L(\Q)$, then $c_{E_2}(n)(\varphi) \in \Q$ for all $n$.
\item If $\varphi \in S_L(\Q)$ and furthermore if $n > 0$ satisfies $n \not\equiv 0 \pmod p$, then the following equality holds in $\R_{(p)} = \R / \oplus_{q \neq p} \Q \cdot \log q$: 
\[
	\lim_{v \to \infty} c_{E'_2}(n,v)(\varphi) \ \equiv \ c_{E_2}(n)(\varphi) \cdot \left( \frac12 \log \pi \ + \  \frac12 \gamma - \frac12 \ - \ \frac{\zeta_F'(2)}{\zeta_F(2)} \ + \ \frac{\zeta'(2)}{\zeta(2)} \right) \in \R_{(p)} . 
\]
\end{enumerate}

\begin{proof}
We make heavy use of the results of \cite{KY} regarding the Fourier coefficients of the general Eisenstein series $E_{L,2}(\tau,s)$, written as
\[
	E_{L,2}(\tau,s) \ = \ \sum_{n} \, A_n(\tau,s)  , \qquad A_n(\tau,s) \in S_L^{\vee},
\]
so that 
\[ c_{E_2}(n) = A_n(\tau, 1) \cdot e^{-2 \pi i n \tau}  \qquad \text{and} \qquad  c_{E'_2}(n) = \left. \left( \frac{d}{ds} A_n(\tau,s) \right)\right|_{s=1} \cdot e^{-2 \pi i n \tau} . \]
Fix an element $\varphi \in S_L(\Q)$ and $n \in \Q$. 

Let $\Sigma$ denote a finite set of finite primes that contains all prime factors of $N$ and all prime factors of $n$ when $n \neq 0$.  In particular, note that if $\ell \notin \Sigma$ then $L_{\ell}$ is a self-dual $\Z_{\ell}$ lattice. Recall also that $\dF$ divides $N$, by assumption, so primes $\ell \notin \Sigma$ are unramified in $F$. 

 When $n \neq 0$, there is a product expansion, \cite[Theorem 2.4]{KY},
\begin{equation} \label{eqn:AnDef}
	A_n(\tau, s)(\varphi) \ = \ W_{n,\infty}(\tau, s) \cdot \left( \prod_{\ell\in \Sigma}  W_{n, \ell}(s, \varphi) \right) \cdot \frac{1}{L^{\Sigma}(s+1, \chi_L)}
\end{equation}
where 
\begin{enumerate}[(i)]
\item $W_{n,\infty}(\tau,s)$ is the weight two archimedean Whittaker function (and is independent of $\varphi$), and $W_{n,\ell}(s, \varphi)$ is the local non-archimedean Whittaker function attached to $\varphi$ (and is independent of $\tau$); see \cite[\S 2]{KY} for more precise definitions. Note that $W_{n,\infty}(\tau,s) \ = \ W_{n,\infty}(\tau, s, \Phi^2_{\infty})$ in the notation of \cite{KY}. 
\item $L^{\Sigma}(s, \chi_L)  = \prod_{\ell \notin \Sigma} (1 - \chi_{L,\ell}(\ell)\ell^{-s})^{-1}$ is the $L$-function, with Euler factors at $\Sigma$ removed, attached to the quadratic character 
\[ \chi_{L,\ell}(x) \ =  \ (x, \det L)_{\ell} \ = \ (x, \dF)_{\ell}  \]
given by the Hilbert symbol $(\cdot, \cdot)_{\ell}$; the second equality follows from the fact that 
\[ \det L \equiv \dF \mod (\Q^{\times})^{2}, \] cf.\ \eqref{eqn:Q0DiagForm}. Note that if $(\ell, \dF) = 1$, then $\chi_{L, \ell}(\ell) =  (\frac{\dF}{\ell})$, and that this is in particular true for $\ell \notin \Sigma$. Thus 
\[ 
	L^{\Sigma}(s, \chi_L) \ = \ L^{\Sigma}(s, \chi_F) 
\]
where $\chi_F = (\frac{\dF}{\cdot})$.
\end{enumerate}
Similarly, for $n=0$, we have the product expansion
\[
	A_{0}(\tau, s)( \varphi) \ = \ v^{\frac{s-1}{2}} \, \varphi(0) \ + \   W_{0,\infty}(\tau, s) \cdot \left( \prod_{\ell\in \Sigma}  W_{0, \ell}(s, \varphi) \right) \cdot \frac{L^{\Sigma}(s, \chi_F)}{L^{\Sigma}(s+1, \chi_F)}.
\]

For any $\ell$ and $n$, the non-archimedean Whittaker functional $W_{n, \ell}(s, \varphi)$ can be written, via \cite[\S 4]{KY}, as
\begin{equation}
				W_{n,\ell}(s, \varphi) \ = \ \gamma_{\ell} \cdot |\det L|_{\ell}^{1/2} \cdot F_{n,\ell}(\ell^{-(s-1)}, \varphi);
\end{equation}
where 
\begin{itemize}
\item $F_{n,\ell}(X, \varphi) \in \Q(X)$ is a rational function in $X$ with rational coefficients, and is in fact a polynomial when $n \neq 0$. Moreover,  $F_{n,\ell}(1, \varphi) \in \Q$ for all $n$ and $\varphi \in S_L(\Q)$, and so  each $W_{n,\ell}(s, \varphi)$ is holomorphic at $s=1$.
\item $\gamma_{\ell}$ is the ``local splitting index", cf.\ \cite[(4.3)]{KY}. It satisfies 
\[ \prod_{\ell < \infty} \gamma_{\ell}  \ = \ 1 \qquad \text{ and } \qquad \gamma_{\ell} \ =  \ 1  \text{ if } \ell \notin \Sigma, \]
since $L_{\ell}$ is unimodular for primes $\ell$ outside $\Sigma$.
\end{itemize}
These considerations together imply that
\begin{equation} \label{eqn:locWhitRat}
	\prod_{\ell \in \Sigma} W_{n,\ell}(1, \varphi) \ = \ \prod_{\ell \in \Sigma} \gamma_{\ell} \,  |\det L|_{\ell}^{1/2} \, F_{n,\ell}(1,\varphi) \ \in \ |\det L|^{-1/2} \cdot \Q \ = \ \dF^{-1/2} \cdot \Q.
\end{equation}
On the other hand, suppose $n \leq 0$. Then the archimedean Whittaker function satisfies 
\begin{equation} \label{eqn:WhitValueNegN}
	W_{n, \infty}(\tau, 1) \ = \ 0 
\end{equation}
cf.\ \cite[Proposition 2.3]{KY}.
This, together with the fact that $L^{\Sigma}(s, \chi_F)\neq 0$ , for $Re(s) \geq 1$, implies that
\[ A_n'(\tau, 1)( \varphi) \ = \ \frac{\log v}{2} \varphi (0) \cdot \delta_{n,0} \ + \ W'_{n, \infty}(\tau, 1) \cdot  c(n, \varphi) \]
for some constant $c(n,\varphi) \in \R$. Appealing again to \cite[Proposition 2.3]{KY}, we have
\[
 \lim_{v \to \infty} e^{-2 \pi i n \tau} W'_{n, \infty}(\tau, 1) \ = \ 0 .
\]
From this, it follows that 
\[
	\lim_{v \to \infty} \left[ e^{- 2 \pi i n v} \, A'_{n}(\tau)(\varphi) \ - \  \frac{\log v}{2} \, \varphi(0) \cdot \delta_{n,0} \right] \ = \ 0
\]
for every $\varphi \in S_L(\Q)$, and hence also for every $\varphi \in S_L$, proving parts (i) and (ii) of the proposition. 

To prove part (iii), we evaluate \eqref{eqn:AnDef} at $s=1$.  Note that $\eqref{eqn:WhitValueNegN}$ and the non-vanishing of $L^{\Sigma}(s, \chi_F)$ at $s=2$ imply that 
\[
c_{E_2}(n)(\varphi) \ = \   e^{2 \pi i n \tau} \, A_n(\tau,1) (\varphi) \ = \ \varphi(0) \cdot \delta_{n,0} \in \Q \qquad \text{for any } n \leq 0 \text{ and }  \varphi \in S_L(\Q).
\]
Turning now to the case $n>0$, \cite[Proposition 2.3]{KY} gives
\begin{equation} \label{eqn:WArchPosn}
	W_{n,\infty}(\tau, 1) \ = \ - 4 \, \pi^2  \, n \, e^{2 \pi i n \tau} \ \neq \ 0 .
\end{equation}
On the other hand, write 
\[ 
	L^{\Sigma}(s, \chi_F) \ = \ \prod_{\ell \in \Sigma} L_{\ell}(s, \chi_F)^{-1} \cdot L(s, \chi_F),
\]
which implies that 
\[ 
	L^{\Sigma}(2, \chi_F) \ \in \ L(2, \chi_F) \cdot \Q^{\times}.
\]
Note that $L(s, \chi_F) \ = \ \frac{\zeta_F(s)}{\zeta_{\Q}(s)}$, and moreover
\[ 
	\zeta_F(2) \ \in \  (\dF)^{-3/2} \, \pi^4 \, \Q^{\times}, \qquad \text{and} \qquad \zeta_{\Q}(2)  \ \in \ \pi^2 \, \Q^{\times};
\]
to see the first statement, observe that $\zeta_F(-1)\in \Q^{\times}$ by Siegel's theorem and use the functional equation. Therefore,
\[ 
	L^{\Sigma}(2, \chi_F) \in \dF^{-1/2} \cdot \pi^2 \cdot \Q^{\times}
\]
and, setting $q = e^{2 \pi i \tau}$, and using \eqref{eqn:locWhitRat} and \eqref{eqn:WArchPosn}, we find
\[ 
c_{E_2}(n)(\varphi) \ = \ q^{-n} \, A_{n}(\tau, 1)(\varphi)   \ = \  q^{-n} \cdot W_{\infty,n}(\tau,1) \cdot \left( \prod_{\ell \in \Sigma} W_{n,\ell}(1, \varphi) \right) \cdot \frac{1}{L^{\Sigma}(2, \chi_F)}  \in \Q
\]
as required.

For part (iv), under the assumption $n>0$ satisfying $n \not\equiv 0 \pmod p$, we may further assume that $p \notin \Sigma$.

First consider the case $c_{E_2}(n)(\varphi) = 0$ or equivalently, $A_{n}(\tau, 1)(\varphi) = 0$. Then   $W_{n,\ell_0}(1, \varphi) = 0 $ for some $\ell_0 \in \Sigma$, and note that 
\[
W'_{n,\ell_0}(1, \varphi)  \ = \ \gamma_{\ell_0} \cdot |\det(L)|^{1/2}_{\ell_0} \cdot  \left. \left( \frac{\partial}{\partial X} F_{n,\ell_0}(X, \varphi)   \right) \right|_{X = 1} \cdot \log \ell_0  \ \in \ \left( \gamma_{\ell_0} \,  |\det(L)|^{1/2}_{\ell_0}  \, \log \ell_0 \right) \cdot \Q
\]
 so, as before, we have
\begin{align*}
c_{E'_2}(n,v) (\varphi) \ = \ 	q^{-n} \cdot A'_n(\tau, 1) (\varphi) \ =& \ q^{-n} \cdot  W'_{n, \ell_0} (1, \varphi) \cdot W_{n, \infty} ( \tau , 1)  \cdot \left( \prod_{\ell \in \Sigma \setminus \{ \ell_0 \}} W_{n,\ell}(1, \varphi) \right) \cdot \frac{1}{L^{\Sigma}(2, \chi_F)} \\ \in& \ \Q \cdot \log \ell_0.
\end{align*}
Since $\Sigma$ does not contain $p$ and $\ell_0 \in \Sigma$, this quantity vanishes in $\R_{(p)}$, proving $(iv)$  in this case.

If $A_n(\tau,1)(\varphi) \neq 0$, then taking the logarithmic derivative  in \eqref{eqn:AnDef} gives
\[
\frac{A'_{n}(\tau,1)(\varphi)}{A_n(\tau,1)(\varphi)} \ = \ \frac{ W_{n,\infty}'(\tau,1)}{W_{n,\infty}(\tau,1)} \ + { \sum_{\ell \in \Sigma}} \frac{ W_{n,\ell}'(1, \varphi )}{W_{n,\ell}(1,\varphi)} \ - \ \frac{L^{\Sigma,'}(2, \chi_F)}{L^{\Sigma}(2,\chi_F)}
\]
By \cite[Proposition 2.3]{KY} again, and using the fact that $\log n \equiv 0 $ in $\R_{(p)}$,
\begin{align*}
	\lim_{v \to \infty}  \frac{ W_{n,\infty}'(\tau,1)}{W_{n,\infty}(\tau,1)}  \ =& \ \frac12 \left( \log \pi + \log n - \frac{\Gamma'(2)}{\Gamma(2)} \right) \\
	&  =  \ \frac12 \left( \log \pi  + \log n - 1+\gamma \right) \ = \  \frac12 \left( \log \pi   - 1+\gamma \right) \ \in \ \R_{(p)}.
\end{align*}
The sum over $\ell \in \Sigma$ contributes a rational linear combination of $\log \ell$'s, and so it too vanishes in $\R_{(p)}$. Finally, note that
\[  \frac{\zeta_F'(2)}{\zeta_F(2)}  \ - \ \frac{\zeta'(2)}{\zeta(2)} \ = \  \frac{L^{'}(2, \chi_F)}{L (2,\chi_F)}  \ \equiv \frac{L^{\Sigma,'}(2, \chi_F)}{L^{\Sigma}(2,\chi_F)}  \ \in \ \R_{(p)} ;\]
 this proves part $(iv)$ of the Proposition.
%
\end{proof}

\end{proposition}

\begin{remark} An alternate proof of $(iii)$ can be obtained as an application of \cite{Kudla-Millson}, who prove that the coefficient $c_{E_2}(n)$ is the degree of a certain rational algebraic divisor on $\Mcal_{F}$.
\end{remark}

\section{Holomorphic analytic torsion and automorphic representations}
The aim of this section is to compare the holomorphic analytic torsions of $\Mcal(\C)$ and $\Scal_{1}(\C)$, via the Jacquet-Langlands correspondence. It turns out that it will be more convenient, from the automorphic point of view, to replace $\Mcal$ with a closely related Shimura variety $\Mcal^{\star}$, attached to the group $\Gbf^{\star} = \mathrm{Res}_{F/\Q} B^{\times}$. 

\subsection{Dolbeault complex of $\lie H^{2}$ and 
Maass differential operators}\label{subsubsec:Dolbeault-complex}
We begin with some generalities on the hyperbolic metric on $\lie H^2$.  
The Dolbeault complex becomes
\begin{displaymath}
	0\overset{\cpd}{\longrightarrow}A^{0,0}(\lie H^{2})\overset{\cpd}{\longrightarrow} A^{0,1}(\lie H^{2})\overset{\cpd}{\longrightarrow} A^{0,2}(\lie H^{2})\overset{\cpd}{\longrightarrow}0,
\end{displaymath}
and the differential $\cpd = \cpd_1 + \cpd_2$ decomposes into the partial derivatives on each factor of $\lie H^2$.

Recall that we had fixed a $PSL_2(\R) \times PSL_2(\R)$-invariant metric on the cotangent bundle $\Omega_{\lie H^2}$ by setting, at a point $\zbf= (z_1, z_2) \in \lie H^2$, 
\begin{displaymath}
	|| dz_{1} ||_{\zbf}^{2}\ = \ 16 \, \pi ^2  \, \Imag(z_1)^{2},\quad || dz_{2} ||_{\zbf}^{2}\ = \ 16 \, \pi^2 \, \Imag(z_2)^{2},\quad \langle dz_{1},dz_{2}\rangle_{\zbf}=0.
\end{displaymath}
Let $\omega$ denote the corresponding K\"ahler form.
As in Section \ref{subsec:hol-an-torsion}, we obtain a pointwise metric $\la \cdot, \cdot \ra$ on each space $A^{0,k}(\lie H^2)$, as well as a Hodge star operator
	\begin{equation}\label{eqg:1}
		\star\colon A^{p,q}(\lie H^{2})\rightarrow A^{2-p,2-q}(\lie H^{2}).
	\end{equation}
determined by the identity
\begin{equation}\label{eqg:4}
	\alpha_{x}\wedge\star\beta_{x}=\langle\alpha_{x},\beta_{x}\rangle_{x}\frac{\omega^{2}}{2}.
\end{equation}
for differential forms $\alpha$ and $\beta$. We obtain \emph{formal adjoints}
\begin{displaymath}
	\cpd^{\ast}_j:=(-1)^{k}\star^{-1}\cpd_j\star: A^{0,k}(\lie H^{2})\rightarrow A^{0,k-1}(\lie H^{2}), \qquad \text{for } j=1,2,
\end{displaymath}
and also $\cpd^* = \cpd^*_1 + \cpd_2^*$. Define the \emph{laplacian operators}
\[
	\Delta_{\cpd,j}^{0,k} \ := \ (\cpd_j+\cpd_j^{\ast})^{2} \ = \ \cpd_j\cpd_j^{\ast}+\cpd_j^{\ast}\cpd_j, \qquad \text{ for } j = 1,2,
\]
and
\[
	\Delta_{\cpd}^{0,k}\ = \ \cpd\cpd^{\ast}+\cpd^{\ast}\cpd \ = \ \Delta_{\cpd,1}^{0,k}+\Delta_{\cpd,2}^{0,k}.
\]
It can easily be verified that the maps $\cpd_j, \cpd_j^*, \cpd$ and $ \cpd^*$ commute with $\Delta_{\cpd,1}^{0,k}, \Delta_{\cpd,2}^{0,k}$ and $\Delta_{\cpd}^{0,k}$; in particular, the various laplacians pairwise commute.

If $\Gamma \subset PSL_2(\R)\times PSL_2(\R)$ is a discrete subgroup acting freely on $\lie H^2$, then 
we may identify $A^{0,k}( \Gamma \backslash \lie H^2)$ with the space of $A^{0,k}(\lie H^2)^{\Gamma}$ of $\Gamma$-invariant forms. Note that all the operators $\cpd, \cpd^*, $ etc., are automatically $\Gamma$-invariant, and so descend to maps between spaces of $\Gamma$-invariant forms. Moreover, if $\Gamma$ is cocompact, then the formal adjoints coincide with the adjoints constructed in Section \ref{subsec:hol-an-torsion}. When $\Gamma$ does not act freely, these identifications can still be made in the context of orbifolds.

The following proposition, whose proof is a direct computation and is left to the reader, relates  the spectral resolution of the Dolbeault laplacian to the spectral resolution of the hyperbolic laplacians on classical Maass forms, providing the link to the automorphic theory needed in subsequent sections.

\begin{proposition}\label{prop:Dolbeault-Maass}
\begin{enumerate}[(i)]
\item The unitary transformations
\begin{displaymath}
	f(z_1,z_2)d\cz_{j}\longmapsto 4\pi y_{j}f(z_1,z_2),\quad g(z_1,z_2)d\cz_1\wedge d\cz_2\longmapsto 16\pi^{2} y_{1}y_{2}g(z_1,z_2),
\end{displaymath}
send $\Gamma$-invariant differential forms of type $(0,1)$ (resp. $(0,2)$) to classical $\Gamma$-automorphic forms of weight $(-2,0)$ or $(0,-2)$ (resp. parallel weight $(-2,-2)$).\footnote{Recall the weight reflects the action of $SO(2)^{2}\subset GL_{2}(\R)^{2}$.}

\item The Dolbeault operators $\cpd_{j}$ (resp. $\cpd^{\ast}_{j}$) are transformed into Maass lowering operators (resp. raising) in the $j$-th variable. 

\item The vector spaces of $\Gamma$-invariant $\lambda$-eigenforms under the $\cpd_{j}$ (resp. $\cpd$) laplacians are isomorphic to $\Gamma$-automorphic Maass forms of eigenvalue $\lambda/16\pi^{2}$ under the $j$-th variable invariant laplacian (resp. the hyperbolic laplacian). 
\end{enumerate}
\end{proposition}

\subsubsection{Restriction of scalars: coarse twisted Hilbert modular surfaces}\label{subsec:restriction-scalars}
The twisted Hilbert modular surfaces $\Mcal_{K}$ are well suited to the geometric study of (twisted) Hilbert modular forms, due to the existence of a universal abelian scheme. However, the study via automorphic representations requires the introduction of an auxiliary Shimura variety; we briefly elaborate on this detour.

Let $\mathbf{G}^{\star} =  \mathrm{Res}_{F/\Q}B^{\times}$ be the algebraic group over $\Q$ obtained from $B^{\times}$ by restriction of scalars; thus for every $\Q$ algebra $A$, we have
\begin{displaymath}
	\mathbf{G}^{\star}(A)=(B\otimes_{\Q}A)^{\times}.
\end{displaymath}
The reduced norm of $B$ induces a morphism of algebraic groups
\begin{displaymath}
	\nu\colon\mathbf{G}^{\star}\longrightarrow\mathrm{Res}_{F/\Q}\Gm_{m\,/F}.
\end{displaymath}
Observe there is a canonical ``diagonal" inclusion of algebraic groups $\Gm_{m\,/\Q}\hookrightarrow\mathrm{Res}_{F/\Q}\Gm_{m\,/F}$. The group $\mathbf{G}$, from which our former twisted Hilbert modular varieties $\Mcal_{K}$ were constructed, cf.\ \eqref{def:G}, fits into a cartesian diagram
\begin{displaymath}
	\xymatrix{
		\mathbf{G}\ar@{^{(}->}[r]\ar[d]		&\mathbf{G}^{\star}\ar[d]^{\nu}\\
		\Gm_{m\,/\Q}\ar@{^{(}->}[r]		&\mathrm{Res}\Gm_{m\,/F}.
	}
\end{displaymath}
Let $\Kstar$ be a compact open subgroup of $\mathbf{G}^{\star}(\A_{\Q,f})=(B\otimes_F\A_{F,f})^{\times}$ and define a compact complex orbifold
\begin{displaymath}
	M^{\star}_{\Kstar}  \ := \ 	\mathbf{G}^{\star}(\Q)\cdot Z^{\star}(\R)\backslash\mathbf{G}^{\star}(\A_{\Q})/ \Kstar \Kstar_{\infty},
\end{displaymath}
where $Z^{\star}(\R)$ is the center of $\mathbf{G}^{\star}(\R)=GL_{2}(\R)^{2}$ and $\Kstar_{\infty}=SO(2)\times SO(2)$. For $\Kstar$ small enough, the theory of canonical models implies that this is the set of complex points of a smooth projective variety $(\Mcal_{\Kstar}^{\star})_{/\Q}$ over $\Q$.

To relate this space to the twisted surfaces $\Mcal_K$ of previous sections, we begin by fixing a set of elements  $g_{1},\ldots, g_{h}\in B(\A_{F})^{\times}$ such that $\nu(g_{1}),\ldots,\nu(g_{h})\in\A_{F}^{\times}$ form a complete set of representatives for the quotient
\begin{displaymath}
 F^{\times}\backslash\A_{F,f}^{\times}/\nu(\Kstar).
\end{displaymath}
For each $i = 1, \dots, h$, let 
\[ K_i = g_i \Kstar g_i^{-1}  \ \cap \ \Gbf(\A_{\Q,f}); \]
then the theory of canonical models implies that the map
\[
	\eta_i \colon M_{K_i} \to M_{\Kstar}^{\star}, \qquad  \eta_i([g])  = [g g_i] \in M_{\Kstar}^{\star},
\]
defined initially on the level of complex points, is induced from a map $\eta_i \colon (\Mcal_{K_i})_{/\Q} \to (\Mcal_{\Kstar}^{\star})_{/\Q}$ of algebraic varieties over $\Q$ which is finite and \'etale over its image (but not in general surjective). One can easily check that the disjoint union
\[
\eta \colon \coprod_{i}( \Mcal_{K_i})_{/\Q} \ \to \ (\Mcal_{\Kstar}^{\star})_{/\Q}
\]
gives a finite \'etale cover.
Finally, let 
\[ K' := \cap_i K_i \subset \Gbf(\A_{\Q,f}).\]
 Then for each $i$, the natural projection map $(\Mcal_{K'})_{/ \Q} \to (\Mcal_{K_i})_{/ \Q}$ is itself a finite \'etale cover, and so taking the disjoint union over $i$ and composing with $\eta$ gives a finite \'etale cover
 \[
 	\pi \colon \coprod_i (\Mcal_{K'})_{/\Q} \to (\Mcal_{\Kstar}^{\star})_{/\Q}
 \]
from $h$ disjoint copies of $(\Mcal_{K'})_{/\Q}$ to $(\Mcal_{\Kstar}^{\star})_{/\Q}$, such that
\begin{equation}\label{eq:rel-Omega-Hilb}
	\pi^{\ast}(\Omega_{\Mcal_{\Kstar}^{\star} /\Q})=\Omega_{\sqcup \Mcal_{K'}/\Q}.
\end{equation}
By construction, the relation \eqref{eq:rel-Omega-Hilb} is an isometry for the invariant metrics induced on both sides from a choice of metric on $\lie H^2$.

In fact, the Shimura variety $(\Mcal_{\Kstar}^{\star})_{/\Q}$ is the {coarse} moduli scheme of a suitable moduli functor, hence explaining the title of this paragraph, and admits a smooth projective model over $\Z[1/N]$ for some integer $N$, cf.\ \cite{Dimitrov-Tilouine}; however, we  will not require these facts in the sequel. 

\subsubsection{Holomorphic analytic torsion for coarse twisted Hilbert modular surfaces}\label{subsection:holom-torsion}
We wish to understand the holomorphic analytic torsion of our coarse twisted Hilbert modular surfaces and Shimura curves (for the trivial hermitian line bundle and with respect to a Poincar\'e metric), in terms of automorphic representations. We will deal with the case of coarse twisted Hilbert modular surfaces, and leave the simpler case of Shimura curves to the reader (see also the proof of Theorem \ref{thm:JL-an-torsion} below).

Suppose, from now on, that $\Kstar\subseteq\mathbf{G}^{\star}(\A_{\Q,f})$ is sufficiently small. Let 
\[
	\Gbf^{\star}(\R)^+ := GL_2^+(\R) \times GL_2^+(\R) \ \subset \ \Gbf^{\star}(\R)
\]  
denote the subgroup consisting of pairs of matrices with positive determinant, and set $\Gbf^{\star}(\Q)^+ = \Gbf^{\star}(\R)^+ \cap \Gbf^{\star}(\Q)$. Fixing a finite set $\{h_j \} \subset \Gbf^{\star}(\A_{\Q,f})$ such that
\[
	\Gbf^{\star}(\A_{\Q,f})  \ =\ \coprod_j \Gbf^{\star}(\Q)^+ \, h_j \, \Kstar,
\]
it is easy to verify that 
\[
 	M_{\Kstar}^{\star} = 	\Mcal_{\Kstar}^{\star}(\C)=\coprod_{j}\Gamma_{j}\backslash\lie H^{2},
\]
for some discrete co-compact subgroups 
$\Gamma_{j} = h_j \, \Kstar \, h_j^{-1} \cap \Gbf(\Q)^+  $
acting on $\lie H^{2}$ without fixed points. We endow $\lie H^{2}$ with a Poincar\'e metric and $\Mcal_{\Kstar}^{\star}(\C)$ with the induced metric.

There is a  decomposition of complex vector spaces,
\begin{equation} \label{eqn:A0kCompDecomp}
	A^{0,k}(\Mcal_{\Kstar}^{\star}(\C))=\overset{\perp}{\bigoplus}_{i}A^{0,k}(\Gamma_{i}\backslash\lie H^{2}) \simeq \overset{\perp}{\bigoplus}_{i}A^{0,k}( \lie H^{2})^{\Gamma_i}
\end{equation}
orthogonal with respect to the $L^{2}$ pairings; in particular, we can define the Dolbeault operators, Laplacians, etc.\ on  $	A^{0,k}(\Mcal_{\Kstar}^{\star}(\C))$ by applying the constructions of the previous section on each component.

Next, we define (finite dimensional) complex vector spaces of eigenforms: given a non-negative real number $\lambda\geq 0$, put
\begin{displaymath}
	V_{\lambda}^{0,k}:=\lbrace\alpha\in	A^{0,k}(\Mcal_{\Kstar}^{\star}(\C))\mid \Delta_{\cpd}^{0,k}\alpha=\lambda\alpha\rbrace.
\end{displaymath}
In light of the definition of analytic torsion, define the weighted dimension
\begin{displaymath}
	D_{\lambda}:= \sum_{k=0}^2  k (-1)^k \dim {V_{\lambda }^{0,k}}  =
	2\dim V_{\lambda }^{0,2}-\dim V_{\lambda }^{0,1}.
\end{displaymath}

\begin{lemma}
Let $\lambda>0$. Then the map 
\begin{displaymath}
	\begin{split}
		V_{\lambda}^{0,0}\overset{\perp}{\oplus} V_{\lambda}^{0,2}&\longrightarrow V_{\lambda}^{0,1}\\
			(\alpha,\beta)&\longmapsto\cpd\alpha+\cpd^{\ast}\beta.
	\end{split}
\end{displaymath}
is a quasi-isometry for the $L^2$ pairings.
In particular, $D_{\lambda} = \dim V_{\lambda }^{0,2} - \dim V_{\lambda}^{0,0}$.
\end{lemma}
\begin{proof}
We prove the  morphism preserves norms up to non-zero constant. First, $\cpd\alpha$ and $\cpd^{\ast}\beta$ are orthogonal, because
\begin{displaymath}
	\langle\cpd\alpha,\cpd^{\ast}\beta\rangle=\langle\cpd\cpd\alpha,\beta\rangle=0.
\end{displaymath}
Second, we have
\begin{displaymath}
	\langle\cpd\alpha,\cpd\alpha\rangle=\langle\alpha,\cpd^{\ast}\cpd\alpha\rangle=\lambda\langle\alpha,\alpha\rangle
\end{displaymath}
and
\begin{displaymath}
	\langle\cpd^{\ast}\beta,\cpd^{\ast}\beta\rangle=\langle\beta,\cpd\cpd^{\ast} \beta\rangle=\lambda\langle\beta,\beta\rangle.
\end{displaymath}
We also exhibit an inverse up to non-zero constant:
\begin{displaymath}
	\begin{split}
		V_{\lambda}^{0,1}&\longrightarrow V_{\lambda}^{0,0}\oplus V_{\lambda}^{0,2}\\
		\alpha&\longmapsto(\cpd^{\ast}\alpha,\cpd\alpha).
	\end{split}
\end{displaymath}
This finishes the proof.
\end{proof}
 For a pair of real numbers $\lambda_1,\lambda_2\geq 0$, we put
\begin{displaymath}
	V_{\lambda_1,\lambda_2}^{0,k}=\lbrace\alpha\in A^{0,k}(\Mcal_{\Kstar}^{\star}(\C))\mid \Delta_{\cpd_j}^{0,k}\alpha=\lambda_j\alpha,\, j=1,2\rbrace.
\end{displaymath}
Observe that $V_{\lambda_1,\lambda_2}^{0,k}\subseteq V_{\lambda_1+\lambda_2}^{0,k}$. Because all our laplacians mutually commute and are self-adjoint, they can be simultaneously diagonalized. Hence, there is a finite orthogonal sum decomposition. 
\begin{equation}\label{eqg:5}
	V_{\lambda}^{0,k}=\overset{\perp}{\bigoplus}_{\lambda_1+\lambda_2=\lambda}V_{\lambda_1,\lambda_2}^{0,k}.
\end{equation}
In particular, setting 
\begin{displaymath}
	D_{\lambda_1,\lambda_2}=\dim V_{\lambda_1,\lambda_2}^{0,2}-\dim V_{\lambda_1,\lambda_2}^{0,0},
\end{displaymath}
we have 
\[
	D_{\lambda} = \sum_{\lambda_1 + \lambda_2 = \lambda} D_{\lambda_1, \lambda_2}.
\]
\begin{lemma}
Suppose that $\lambda_1,\lambda_2\neq 0$. Then there is a quasi-isometry
\begin{displaymath}
	V_{\lambda_1,\lambda_2}^{0,0}\overset{\cpd_1\cpd_2}{\longrightarrow} V_{\lambda_1,\lambda_2}^{0,2}.
\end{displaymath}
Consequently, $D_{\lambda_1,\lambda_2}=0$.
\end{lemma}
\begin{proof}
Let us compute the norm of this morphism:
\begin{align*}
	\langle\cpd_1\cpd_2\alpha,\cpd_1\cpd_2\alpha\rangle&=\langle\alpha,\cpd_{2}^{\ast}\cpd_{1}^{\ast}\cpd_{1}\cpd_{2}\alpha\rangle\\
	&=\langle\alpha,\cpd_{2}^{\ast}\cpd_{2}\Delta_{\cpd_1}^{0,0}\alpha\rangle\\
	&=\lambda_1\langle\alpha,\Delta_{\cpd_2}^{0,0}\alpha\rangle\\
	&=\lambda_1\lambda_2\langle\alpha,\alpha\rangle.
\end{align*}
Similarly, a quasi-inverse is $\cpd_{2}^{\ast}\cpd_1^{\ast}$.
\end{proof}
\begin{corollary}\label{corollary:multiplicities}
Let $\lambda>0$. We have the following expression for the weighted dimension $D_{\lambda}$:
\begin{displaymath}
	D_{\lambda}=D_{\lambda,0}+D_{0,\lambda}.
\end{displaymath}
\end{corollary}


While the preceding discussion is based on general considerations arising from the uniformization by $\lie H^{2}$, we now apply the automorphic theory, which is particular to the Shimura variety context. In order to do so, we define an embedding 
\begin{equation}
	\bigoplus_{k=0}^2 A^{0,k}(\Mcal_{\Kstar}^{\star}(\C)) \ \hookrightarrow \ L^{2}(\mathbf G^{\star}(\Q)\backslash\mathbf G^{\star}(\A_{\Q}))^{\Kstar}
\end{equation}
as follows:
\begin{itemize}
	\item if $k=0$, then a function on $\Mcal_{\Kstar}^{\star}(\C)$ naturally pulls back to a $\Kstar$-invariant square-integrable function on $\Gbf^{\star}(\Q) \backslash \Gbf^{\star}(\A_{\Q})$ via the adelic presentation
		\[
			\Mcal^{\star}_{\Kstar}(\C)  =	\mathbf{G}^{\star}(\Q)\cdot Z^{\star}(\R)\backslash\mathbf{G}^{\star}(\A_{\Q})/ \Kstar \Kstar_{\infty}.
		\]
	\item if $k=2$, then a differential form $\eta \in A^{0,2}(\Mcal_{\Kstar}^{\star}(\C))$ corresponds, via \eqref{eqn:A0kCompDecomp} and Proposition \ref{prop:Dolbeault-Maass}(i), to a tuple $(g_j)$ where each $g_j \colon \lie H^2 \to \C$ is a modular function of parallel weight $-2$ for $\Gamma_j$. For each $j$, define a function
	\[
		\varphi_j \colon \Gbf^{\star}(\R)^+ = GL_2^+(\R) \times GL_2^+(\R)  \to \C 
	\]
	by setting
	\[
		\varphi_j \left( A_1, A_2 \right) 	= g_j \left(A_1 \cdot i, A_2 \cdot i \right) \frac{ (c_1 i + d_1)^2 (c_2 i + d_2)^2}{\det A_1 \, \det A_2}, \qquad \text{ where } (A_1, A_2) = \left( (\begin{smallmatrix}	a_1 & b_1 \\ c_1 & d_1	\end{smallmatrix}), \, (\begin{smallmatrix}	a_2 & b_2 \\ c_2 & d_2	\end{smallmatrix}) \right).
	\]
	This extends to a function $ \varphi \colon \Gbf^{\star}(\R)^+ \times \Gbf^{\star}(\A_{\Q,f}) \to \C $ by setting
	\[  \varphi( g_{\infty}, \gamma h_j k) = \varphi_j( \gamma^{-1} g_{\infty} )\]
	for $\gamma \in \Gbf^{\star}(\Q)^+$ and $k \in \Kstar$; this construction is clearly invariant under $\Gbf^{\star}(\Q)^+$ acting by left multiplication on both factors, and $\Kstar$ acting on the second factor on the right. Finally, we may view $\varphi$ as a function on $\Gbf^{\star}(\A)$ by identifying
	\[
		\Gbf^{\star}(\Q)^+ \backslash \Gbf^{\star}(\R)^+ \times \Gbf^{\star}(\A_{\Q,f}) \ \simeq \ 	\Gbf^{\star}(\Q) \backslash \Gbf^{\star}(\R) \times \Gbf^{\star}(\A_{\Q,f}) = 		\Gbf^{\star}(\Q) \backslash \Gbf^{\star}(\A).
	\]
	It is easy to check that the assignment $g \mapsto \varphi$ is independent of all choices, and determines the desired embedding $A^{0,2}(\Mcal_{\Kstar}^{\star}(\C)) \hookrightarrow   L^{2}(\mathbf G^{\star}(\Q)\backslash\mathbf G^{\star}(\A_{\Q}))^{\Kstar}$. 
	
	\item if $k=1$, a form in $A^{0,1}(\Mcal_{\Kstar}^{\star}(\C)) $ determines, via Proposition \ref{prop:Dolbeault-Maass}(i), a collection of pairs $(f_j, g_j)$ of modular functions for $\Gamma_j$ of weight $(-2,0)$ and $(0, -2)$ respectively; we may piece them together in an analagous way to the previous case.
\end{itemize}


The upshot is that $A^{0,k}(\Mcal_{\Kstar}^{\star}(\C))$ can be neatly described in terms of automorphic representations as follows. Let $w_{k}$ be the weight $k$ representation of $SO(2)$, given by the character
\begin{displaymath}
	w_{k}\left(\begin{array}{cc} \cos(\theta)	&-\sin(\theta)\\ \sin(\theta)	&\cos(\theta)\end{array}\right)=e^{ik\theta}.
\end{displaymath}
Define representations of $\Kstar_{\infty}=SO(2)(\R)\times SO(2)(\R)$ by
\begin{displaymath}
	\rho_{0}=\text{trivial} ,\quad \rho_{1}=(w_{0}\boxtimes w_{2})\,\oplus\,( w_{2}\boxtimes w_{0}),\quad \rho_{2}=w_{2}\boxtimes w_{2}.
\end{displaymath}
These can be extended to representations of $\Kstar \Kstar_{\infty}$, by letting $\Kstar$ act trivially. Denoting these extensions by the same symbol, we have
\begin{equation}\label{eq:L2-decomposition}
	A^{0,k}(\Mcal_{\Kstar}^{\star}(\C))=\overset{\perp}{\bigoplus}_{\pi}\Hom_{\Kstar \Kstar_{\infty}}(\rho_{k},\pi),
\end{equation}
where the sum is to be taken in the sense of Hilbert spaces, and $\pi$ runs over the irreducible automorphic representations of $\mathbf{G}^{\star}(\A_{\Q})$. The theory of automorphic representations ensures that the summands are finite dimensional. We say that \emph{a representation $\pi$ contributes to} $A^{0,k}(\Mcal_{\Kstar}^{\star}(\C))$ if $\Hom_{\Kstar\Kstar_{\infty}}(\rho_{k},\pi)\neq 0$.

An automorphic representation $\pi$ of $\mathbf{G}^{\star}(\A_{\Q})$ decomposes as a completed tensor product $\widehat{\otimes}_{\nu}\pi_{\nu}$ over the places $\nu$ of $F$. For the archimedean places $v_{1},v_{2}$, the components $\pi_{v_{1}}$ and $\pi_{v_{2}}$ determine a pair of eigenvalues $(\lambda_{1},\lambda_{2})$, corresponding to eigenvalues of $\Delta_{\cpd_{1}}^{0,k}$, $\Delta_{\cpd_{2}}^{0,k}$ under the equivalences in Proposition \ref{prop:Dolbeault-Maass}; we say that $(\lambda_{1},\lambda_{2})$ is the pair of eigenvalues attached to $\pi$. In particular, we may decompose
\begin{displaymath}
	V_{(\lambda_{1},\lambda_{2})}^{0,k}=\overset{\perp}{\bigoplus}_{\pi}\Hom_{\Kstar\Kstar_{\infty}}(\rho_{k},\pi),
\end{displaymath}
where now the sum runs only over automorphic representations $\pi$ with attached pair of eigenvalues $(\lambda_{1},\lambda_{2})$.

We examine the consequences for holomorphic analytic torsion: by Corollary \ref{corollary:multiplicities}, the only possible representations contributing to the holomorphic analytic torsion will be those contributing to $A^{0,0}(\Mcal_{\Kstar}^{\star}(\C))$ or $A^{0,2}(\Mcal_{\Kstar}^{\star}(\C))$, and with attached pairs of eigenvalues of the form $(\lambda,0)$ or $(0,\lambda)$.

\begin{proposition}\label{theorem:multiplicities}
For $\lambda>0$, the eigenspaces $V_{(0,\lambda)}^{0,0}$ and $V_{(\lambda,0)}^{0,0}$ are trivial.
\end{proposition}
\begin{proof}
Suppose, by way of contradiction, that there is a non-zero function $f$ on $\Mcal_{\Kstar}^{\star}(\C)$ with attached pair of eigenvalues $(0,\lambda)$, which therefore generates an infinite dimensional representation $\pi = \pi_f$ contributing to $V^{0,0}_{(0,\lambda)}$; in particular
\[ \Hom_{\Kstar\Kstar_{\infty}}(\rho_0, \pi) \neq 0. \]
Decompose $\pi = \hat\otimes \pi_v$ in terms of local components, so that $\pi_{v_1}$ and $\pi_{v_2}$ are irreducible admissible $(\lie{gl}_2, SO(2))$-modules.  By the classification of such representations, cf.\ \cite[Chapter 2]{Bump} and especially Proposition 2.5.2 of \emph{loc. cit.}, the fact that the $\Delta^{0,0}_{\bar\partial,1 }$ eigenvalue of $f$ is zero and that $\pi_{v_1}$ has a weight zero vector implies that $\pi_{v_1}$ is trivial.

Now let $\pi'$ denote the automorphic representation of $GL_2(\A_F)$ corresponding to $\pi$ under the Jacquet-Langlands correspondence, so that in particular $ \pi'_{v_1} \simeq \pi_{v_1}$ is trivial.
From  $\pi'$ we can extract a non-zero classical cuspidal modular function $g \colon \lie H^2 \to \C$ that is invariant under the action of some $\Gamma' \subset SL_2(\OF)$,  and is \emph{constant} in the first variable. But the cuspidality condition then implies that $g \equiv 0$, a contradiction. 
The same argument implies $V_{(\lambda,0)}^{0,0}=0$.
\end{proof}
It remains to study the representations contributing to  $V^{0,2}_{(0,\lambda)}$ and  $V^{0,2}_{(\lambda,0)}$. If $\pi$ contributes to   $V^{0,2}_{(0,\lambda)}$, then  $\pi_{v_1}$ has laplace eigenvalue zero;  appealing again to \cite[Chapter 2]{Bump} we find that $\pi_{v_1}$ is necessarily discrete of weight $-2$. Similarly, a representation $\pi$ contributing to $V^{0,2}_{(\lambda, 0)}$ must have $\pi_{v_2} $ discrete of weight $-2$.
To summarize the discussion so far, we introduce  the \emph{spectral zeta function}
\begin{equation}\label{eq:spectral-xi}
	\xi_{\Kstar}(s):=\sum_{\pi}\frac{m_{\Kstar}(\pi)}{\lambda(\pi)^{s}},
\end{equation}
where $\pi$ runs over all the automorphic representations  contributing to $A^{0,2}(\Mcal_{\Kstar}^{\star}(\C))$ and such that $\pi_{v_1}$ or $\pi_{v_2}$ is discrete of lowest weight -2; here  $\lambda(\pi)$ is the corresponding eigenvalue under $\Delta^{0,2}_{\cpd}$, and  is counted with multiplicity 

\begin{displaymath}
 m_{\Kstar}(\pi) :=	\dim\Hom_{\Kstar \Kstar_{\infty} }(\rho_{2},\pi).
\end{displaymath}
Observe that it can happen that non-isomorphic representations $\pi$ and $\pi'$ contribute and share the eigenvalue $\lambda(\pi)=\lambda(\pi')$. 

The spectral zeta function $\xi_{\Kstar}(s)$ is a linear combination of the spectral zeta functions $\zeta_{\Ocal,1}(s)$ (for $\Delta_{\cpd}^{0,1}$) and $\zeta_{\Ocal,2}(s)$ (for $\Delta_{\cpd}^{0,2}$), and hence it is absolutely convergent for $\Real(s)\gg 0$ and has a meromorphic continuation to $\C$, regular at $s=0$.

The analysis in this section can be summarized as follows:
\begin{theorem}\label{thm:hol-an-tor-M_K}
The holomorphic analytic torsion of the trivial hermitian line bundle on $\Mcal_{\Kstar}^{\star}(\C)$, with respect to the invariant Poincar\'e metric, satisfies
\begin{displaymath}
	T(\Mcal_{\Kstar}^{\star}(\C),\ov{\Ocal})=-\xi^{\prime}_{\Kstar}(0).
\end{displaymath}
\end{theorem}
\subsection{Holomorphic analytic torsion and the Jacquet-Langlands correspondence}\label{subsec:hol-tor-JL}
We maintain the previous assumptions and notations regarding the quaternion algebra $B$, the Shimura variety $\Mcal_{\Kstar}^{\star}$, etc.

Recall from the introduction that we had also fixed a rational prime $\ell$ inert in the quadratic number field $F$ and coprime to $N$. In particular, $\ell$ is coprime to the discriminant $D_{B}$ of $B$. We denoted by $B_{1}$ the division quaternion algebra over $F$ ramified at the archimedean place $v_{1}$ and with discriminant $D_{B_{1}}=\ell D_{B}$. We similarly define $B_{2}$ as the algebra ramified at $v_{2}$ and with discriminant $D_{B_{2}}=\ell D_{B}$. 

For each rational prime $p \neq \ell$, fix once and for all isomorphisms
\[
	B_1 \otimes_{\Q} \Q_{p} \simeq B_2 \otimes_{\Q} \Q_{p} \simeq B\otimes_{\Q} \Q_{p}  
\]
as well as isomorphisms 
\[
	B\otimes \Q_{\ell} \simeq M_2(F_{\ell}) = M_2(\Q_{\ell^2}) \qquad \text{and} \qquad B_1 \otimes \Q_{\ell} \simeq B_2 \otimes \Q_{\ell}.
\]
These isomorphisms induce identifications 
\[\Gbf_1(\Q_{p}) \simeq \Gbf_2(\Q_{p}) \simeq \Gbf^{\star}(\Q_p) \text{  for all } p \neq \ell, \qquad \text{ and } \qquad\Gbf_1 (\Q_{\ell}) \simeq \Gbf_2(\Q_{\ell}) .\]
 
Given compact opens $\Kstar \subset \Gbf^{\star}(\A_{\Q,f})$ and $\Kstar_1 \subset \Gbf_1(\A_{\Q,f})$, we say $\Kstar$ and $\Kstar_1$ \emph{match} if the local components $\Kstar_{p}$ and $  \Kstar_{1,p}$ are identified for all $p \neq \ell$. Similarly, $\Kstar_1$ is said to match $\Kstar_2 \subset \Gbf_2(\A_{\Q,f})$ if $\Kstar_{1,p}$ and $  \Kstar_{2,p}$   are identified for all $p$. 

Finally, recall from the introduction that given a sufficiently small $\Kstar_1 \subset  \Gbf_1(\A_{\Q,f})$, we defined the complex Shimura curve
\[
	S_{1, \Kstar_1}  =   \Gbf_1(\Q) \big\backslash \lie H^{\pm} \times \Gbf_1(\A_{\Q,f}) / \Kstar_1.
\]
Similarly, for $\Kstar_2 \subset  \Gbf_2(\A_{\Q,f})$ which, say matches $\Kstar_1$, we may  define
\[
	S_{2, \Kstar_2} =     \Gbf_2(\Q) \big\backslash \lie H^{\pm} \times \Gbf_2(\A_{\Q,f}) / \Kstar_2.
\]
The curve $S_1 = S_{1, \Kstar_1}$  admits a canonical model $\Scal_{1}$ over $F$, such that $\Scal_{1,v_1}(\C) = S_1$; it follows immediately from definitions, canonicity, and the identifications above that $	\Scal_{1, v_2}(\C) \simeq S_2$.
In particular, by definition  
\[
	\Scal_1(\C)  :=  \coprod_{v \colon F \to \C} \Scal_{1,v}(\C) = S_1 \coprod S_2,
\]
and hence for the analytic torsion,  
\begin{equation} \label{eqn:AnTorScal1}
	T(\Scal_{1}(\C),\ov{\Ocal})=T(S_{1},\ov{\Ocal})+T(S_{2},\ov{\Ocal}).
\end{equation}
The aim of the following subsections is to relate this quantity to the analytic torsion $T(\Mcal_{\Kstar}^{\star}(\C))$ for a particular choice of $\Kstar$ and corresponding matching level structure for $\Scal_1$.

\subsubsection{The $\ell$-new holomorphic analytic torsion}\label{subsubsec:ell-new-torsion}
In the following discussion, we further exploit the automorphic theory to define and study a ``new part" of holomorphic analytic torsion. We use the automorphic variant of Atkin-Lehner theory explained by Casselman \cite{Casselman} and Miyake \cite{Miyake}, and in particular the notion of conductor of an automorphic representation. Although the authors only cover the $GL_{2}$ case, the quaternionic setting requires only standard modifications.

\paragraph{A general construction} Recall the description of the holomorphic analytic torsion $T(\Mcal_{\Kstar}^{\star}(\C),\ov{\Ocal})$ of Theorem \ref{thm:hol-an-tor-M_K}, in terms of the spectral zeta function \eqref{eq:spectral-xi}. We define the \emph{$\ell$-new spectral zeta function}
\begin{displaymath}
	\xi^{\ell-\mathrm{new}}_{\Kstar}(s)=\sum_{\ell\mid c(\pi)}\frac{m_{\Kstar}(\pi)}{\lambda(\pi)^{s}},
\end{displaymath}
where the automorphic representations now run over those contributing to $A^{0,2}(\Mcal_{\Kstar}^{\star}(\C))$, such that either $\pi_{v_1}$ or $\pi_{v_2}$ is discrete series of weight -2,  and the  conductor $c(\pi)$ divisible by $\ell$.
This spectral zeta function is absolutely convergent for $\Real(s)\gg 0$, since $\xi_{\Kstar}(s)$ is and we just restricted the summation set.


Provided that $\xi^{\ell-\mathrm{new}}_{\Kstar}(s)$ admits a holomorphic continuation to a neighborhood of $s=0$, one defines \emph{$\ell$-new holomorphic analytic torsion} by
\begin{equation}\label{eq:ell-new-torsion}
	T(\Mcal_{\Kstar}^{\star}(\C),\ov{\Ocal})^{\ell-\mathrm{new}}:=-\xi^{\ell-\mathrm{new}\,\prime}_{\Kstar}(0) = - \left. \frac{d}{ds}\xi^{\ell-\mathrm{new}}_{\Kstar}(s)\right|_{s=0} .
\end{equation}
For an appropriate choice of $\Kstar$, as in the next section, we will see that the $\ell$-new holomorphic analytic torsion can indeed be defined.

\paragraph{A choice of compact open subgroup $\KstarN$} 
For each prime $p \nmid D_B$, fix an isomorphism $B\otimes\Q_p \simeq M_2(F \otimes_{\Q} \Q_p)$.
Recall that $N$ is an integer divisible by $2\,\dF\,D_{B}$. Let $N'$ be its greatest divisor coprime to $6\,\dF\, D_{B}$, and assume that $N'>1$. 

In the remainder of this paper, we fix a subgroup 
\begin{equation} \label{eqn:KstarNDef}
	\KstarN \subset \Gbf^{\star}(\A_{\Q,f})
\end{equation}
satisfying the following conditions:
\begin{itemize}
	\item $\KstarN$ is maximal at places not dividing $N$ and maximal at places dividing $2\,D_{B}$;
	\item at primes $p$ dividing $N'$, we assume that $\KstarN$ is of type $\Gamma_{1}$. More precisely, if $\nu$ is a finite place of $F$ above $p$ and $\varpi_{\nu}$ is a uniformizer of $F_{\nu}$, then for an integer $n$ set
	\begin{displaymath}
		 K_{1}(\varpi_{\nu}^{n})=\left\lbrace\left(\begin{array}{cc} a  &b\\ c  &d\end{array}\right)\in GL_{2}(\Ocal_{F_{\nu}})\mid c,\,d-1\equiv 0\mod\varpi_{\nu}^{n}\right\rbrace;
	\end{displaymath}
	we then require that $\KstarN_p = \prod_{\nu | p} K_1(\varpi_{\nu}^{n_{\nu}})$ for some integers $n_{\nu}$. 
	\item at places dividing $\dF$, namely at ramified places, we require $\KstarN_{\nu}=K_{1}(\mathfrak{d}_{F_{\nu}})$, where $\mathfrak{d}_{F_{\nu}}$ is the different of $F_{\nu}$ relative to $\Q_{p}$ (here $p$ is the residual characteristic of $\nu$).
	\item finally, we require that $\KstarN$ is of Iwahori type at the inert prime $\ell$; i.e.\ 
	\[
		\KstarN_{\ell} = \left\lbrace\left(\begin{array}{cc} a	&b\\c  &d\end{array}\right)\in GL_{2}(\Ocal_{F_{\ell}})\mid c\equiv 0\mod\ell\right\rbrace  .
	\]
\end{itemize}
 Under these conditions, $\KstarN$ is small enough to guarantee that $\Mcal_{\KstarN}^{\star}(\C)$ is a smooth projective variety (see \cite[Lemme 1.4]{Dimitrov-Tilouine} for a discussion in the Hilbert modular case; the quaternionic case is dealt with similarly). 
\begin{lemma}
 Let $\Kstar = \KstarN$ as above, and let $\Kstar' \supset \Kstar$ coincide with $\Kstar$ at all finite places except for $\ell$, where $\Kstar^{\prime}_{\ell}$ is maximal. Then the $\ell$-new holomorphic analytic torsion is defined and
\begin{displaymath}
	T^{\ell-new}(\Mcal_{\Kstar}^{\star}(\C),\ov{\Ocal})=T(\Mcal_{\Kstar}^{\star}(\C),\ov{\Ocal})-2T(\Mcal_{\Kstar^{\prime}}^{\star}(\C),\ov{\Ocal}).
\end{displaymath}
\end{lemma}
\begin{proof}
We decompose the spectral zeta function $\xi_{\Kstar}(s)$ as 
\begin{equation}\label{eq:lemma-ell-new-0}
	\xi_{\Kstar}(s)=\sum_{\ell\mid c(\pi)}\frac{m_{\Kstar}(\pi)}{\lambda(\pi)^{s}}+\sum_{\ell\nmid c(\pi)}\frac{m_{\Kstar}(\pi)}{\lambda(\pi)^{s}},
\end{equation}
where the sums run over  representations $\pi$ contributing to $A^{0,2}(\Mcal_{\Kstar}^{\star}(\C))$ and such that $\pi_{v_1}$ or $\pi_{v_2}$ is discrete of lowest weight -2.
For such an automorphic representation $\pi$, if $\ell\mid c(\pi)$ then
\begin{equation}\label{eq:lemma-ell-new-1}
m_{\Kstar'}(\pi) =	\dim\Hom_{\Kstar '\Kstar_{\infty}}(\rho_{2},\pi)=0;
\end{equation}
i.e. the representation space has no ${\Kstar'}$ invariant vectors. Hence, it does not contribute to $T(\Mcal_{\Kstar^{\prime}}^{\star}(\C),\ov{\Ocal})$. By contrast, if $\ell\nmid c(\pi)$, then by \cite[Corollary to the Proof]{Casselman} and the discussion leading to Theorem 4 in \emph{loc. cit.},  
\begin{equation}\label{eq:lemma-ell-new-2}
m_{\Kstar}(\pi) = 	\dim\Hom_{\Kstar\Kstar_{\infty}}(\rho_{2},\pi)=2\dim\Hom_{\Kstar\prime\Kstar_{\infty}}(\rho_{2},\pi) = 2\, m_{\Kstar'}(\pi).
\end{equation}
From \eqref{eq:lemma-ell-new-0}--\eqref{eq:lemma-ell-new-2}, we see that
\begin{displaymath}
	\xi_{\Kstar}(s)=\xi_{\Kstar}^{\ell-new}(s)+2\xi_{\Kstar^{\prime}}(s).
\end{displaymath}
The lemma follows, since $\xi_{\Kstar}(s)$ and $\xi_{\Kstar^{\prime}}(s)$ have holomorphic continuations to a neighborhood of $s=0$, and by the very definition of the holomorphic analytic torsions.
\end{proof}

Before stating our main result, we note that the spectral zeta function $\zeta_{S_{1}}(s)$, of $\Delta^{0,1}_{\cpd}$ on the complex Shimura curve $S_{1} = \Scal_{1, v_1}(\C)$, can also be interpreted in terms of automorphic representations, as follows: let
\[ 
	\Kstar_{1, \infty} = \mathbb H^{\times,1} \times SO(2) \subset \Gbf_1(\R) ,
\]
and consider the representation $\rho'_2 = 1 \boxtimes w_2$. Then, arguing exactly as in Section \ref{subsection:holom-torsion} yields an identity
\[
	 	\zeta_{S_{1}}(s)=\sum_{\pi'}\frac{m_{\Kstar_1}(\pi')}{\lambda(\pi')^{s}}, \qquad \text{ where } m_{\Kstar_1}(\pi') := \dim \Hom_{\Kstar_1\Kstar_{1,\infty}}(\rho_2', \pi')
\]
and $\pi'$ runs over all the irreducible automorphic representations of $\mathbf{G}_{1}(\A_{\Q})$ contributing to $A^{0,1}(S_{1})$  with $\pi_{v_{1}}^{\prime} $ trivial. A similar expression holds for $\zeta_{S_2}(s)$, where $S_2 = \Scal_{1, v_2}(\C)$, by interchanging the roles of the infinite places.

We proceed to state and prove the main theorem of this section.
\begin{theorem}\label{thm:JL-an-torsion}
Let $\Kstar=\KstarN$, and  suppose $\Kstar_1 \subset \Gbf_1(\A_{\Q,f})$ matches $\Kstar$. Then, if $\Scal_1$ is the corresponding Shimura curve, there is an equality of holomorphic analytic torsions
\begin{displaymath}
	T^{\ell-\mathrm{new}}(\Mcal_{K}^{\star}(\C),\ov{\Ocal})=-T(\Scal_{1}(\C),\ov{\Ocal}).
\end{displaymath}
\end{theorem}
\begin{proof}
%
%

Let $\pi'$ be an automorphic representation of $\mathbf{G}_{1}(\A_{\Q})$ contributing to $A^{0,1}(S_{1})$, with $\pi_{v_1}'$ trivial, and corresponding eigenvalue $\lambda(\pi')$ for $\Delta_{\cpd}^{0,1}$. Let $\pi$ be its Jacquet-Langlands lift to an automorphic representation of $\mathbf{G}^{\star}(\A_{\Q})$. By the very construction of the local Jacquet-Langlands correspondence, the relation between the conductors is $c(\pi)=\ell c(\pi')$. Furthermore, $\pi_{v_1}$ is discrete series of weight -2. We deduce that $\pi$ contributes to $A^{0,2}(\Mcal_{\Kstar}^{\star}(\C))$, and its corresponding eigenvalue under $\Delta_{\cpd}^{0,2}$ is $\lambda(\pi)=\lambda(\pi')$. Finally, \cite[Corollary to the Proof]{Casselman} (and again an analogous discussion leading to Theorem 4 in \emph{loc. cit.}), one has
\begin{equation}\label{eq:matching-mult}
 	m_{\Kstar}(\pi) = 	\dim\Hom_{\Kstar\Kstar_{\infty}}(\rho_{2},\pi)=\dim\Hom_{\Kstar_1 \Kstar_{1,\infty}}(\rho_{2}^{\prime},\pi^{\prime}) = m_{\Kstar_1}(\pi');
\end{equation}
in fact, this multiplicity is equal to  the number of integral ideal divisors of $c(\pi')$. 
A symmetric discussion applies to the lifts of representations contributing to $A^{0,1}(S_{2})$ with $\pi_{v_2}$ trivial. 

Conversely, a representation $\pi$ contributing to $A^{0,2}(\Mcal_{\Kstar}^{\star}(\C))$ with $\ell\mid c(\pi)$ and $\pi_{v_1}$ discrete series of weight -2 arises as the Jacquet-Langlands lift of a representation $\pi'$ contributing to $A^{0,1}(S_{1})$, with conductor $c(\pi')=c(\pi)/\ell$, and $\pi'_{v_1}$ trivial. The corresponding eigenvalues, including multiplicities, match by the discussion above. Similarly, those representations with $\pi_{v_2}$ is discrete series are Jacquet-Langlands lifts from $A^{0,1}(S_2)$.

The discussion can be summarized in the following string of equalities:
\begin{displaymath}
	\begin{split}
		\xi_{\Kstar}^{\ell-new}(s)=&\sum_{\substack{\ell\mid c(\pi) \\ \pi_{v_1} = \text{d.s. wt. -2}}}\frac{m_{\Kstar}(\pi)}{\lambda(\pi)^{s}}+\sum_{\substack{\ell\mid c(\pi)\\ \pi_{v_{2}}=\text{d.s. wt. -2}}}\frac{m_{\Kstar}(\pi)}{\lambda(\pi)^{s}}\\
	=&\sum_{\substack{\pi'\\ \text{repr. of }\mathbf{G}_{1}(\A_{\Q})}}\frac{m_{\Kstar_1}(\pi')}{\lambda(\pi')^{s}}+\sum_{\substack{\pi''\\ \text{repr. of }\mathbf{G}_{2}(\A_{\Q})}}\frac{m_{\Kstar_2}(\pi'')}{\lambda(\pi'')^{s}}\\
	=& \zeta_{S_{1}}(s)+\zeta_{S_{2}}(s).
	\end{split}
\end{displaymath}
The theorem follows immediately from the definition of analytic torsion, cf.\ Definition \ref{def:holAnTor}, together with \eqref{eq:ell-new-torsion}  and \eqref{eqn:AnTorScal1}.
\end{proof}

\section{Arithmetic Grothendieck-Riemann-Roch and Jacquet-Langlands correspondence}

\subsection{Comparison of volumes}\label{subsec:comparison-volumes}
This subsection is preparatory for the forthcoming comparison of arithmetic intersection numbers of $\Mcal_{K}$ and $\Scal_{1}$. We compare classical ``volumes" of $\Mcal_{\Kstar}^{\star}(\C)$ and $\Scal_{1}(\C)$, by invoking the Riemann-Roch theorem and the Jacquet-Langlands correspondance. Riemann-Roch expresses volumes in terms of Euler-Poincar\'e characteristics; the latter are related by the Jacquet-Langlands correspondence. In the arithmetic setting, the strategy will be \emph{formally} the same: the arithmetic Riemann-Roch theorem relates arithmetic intersections to arithmetic degrees of determinants of cohomologies (with Quillen metrics); the Jacquet-Langlands correspondence (through Theorem \ref{thm:JL-an-torsion} and period computations by Shimura) relates these arithmetic degrees.

We continue with the notation from \S \ref{subsec:hol-tor-JL}, so that $\Kstar$ is a compact open subgroup of $\mathbf{G}^{\star}(\A_{\Q,f})$ and $\ell$ an auxiliary inert rational prime. 
\begin{definition} Suppose $\Kstar$ is Iwahori at $\ell$, and let $\Kstar' \supset \Kstar$ denote the subgroup that is maximal at $\ell$ and agrees with $\Kstar$ at all primes away from $\ell$. In this case, define
\begin{displaymath}
	\deg^{\ell-\mathrm{new}} c_{1}(\Omega_{\Mcal_{\Kstar}^{\star}(\C)})^{2}=
	\deg  c_{1}(\Omega_{\Mcal_{\Kstar}^{\star}(\C)})^{2}-2\deg c_{1}(\Omega_{\Mcal_{\Kstar^{\prime}}^{\star}(\C)})^{2},
\end{displaymath}
and similarly for $c_{2}(\Omega_{\Mcal_{\Kstar}^{\star}(\C)})$.
\end{definition}
\begin{proposition}\label{prop:geom-deg}
Suppose $\Kstar$ is Iwahori at $\ell$. Let $\Kstar_1 \subset \Gbf_1(\A_{\Q,f})$ denote the subgroup matching  $\Kstar$, with $\Kstar_1$ maximal at $\ell$, and $\Scal_1$ the corresponding Shimura curve over $F$. Then
\begin{displaymath}
	\deg^{\ell-\mathrm{new}}c_{1}(\Omega_{\Mcal_{\Kstar}^{\star}(\C)})^{2}=2\deg c_{1}(\Omega_{\Scal_{1}(\C)}).
\end{displaymath}
\end{proposition}
\begin{proof}
First of all, by the functoriality of characteristic classes under pull-back and the projection formula, the degrees in question are multiplicative under finite \'etale coverings: if we replace $\Kstar$ by a smaller normal compact open subgroup $\mathbb L \lhd \Kstar$ with $\mathbb L_{\ell} = \Kstar_{\ell}$, then both sides of the desired equality get multiplied by the index $[\Kstar\colon \mathbb L]$. 
Thus we may reduce to the case $\Kstar = \KstarN$ as in \S \ref{subsubsec:ell-new-torsion}. 

Defining a $\ell$-new Euler-Poincar\'e characteristic $\chi^{\ell-\mathrm{new}}$ by a similar rule as for $\deg^{\ell-\mathrm{new}}$, the Riemann-Roch theorem gives
\begin{equation}\label{eq:chi-new}
	\chi^{\ell-\mathrm{new}}(\Mcal_{\Kstar}^{\star}(\C),\Ocal)=\frac{1}{12}\left(\deg^{\ell-\mathrm{new}} c_{1}(\Omega_{\Mcal_{\Kstar}^{\star}(\C)})^{2}+\deg^{\ell-\mathrm{new}}c_{2}(\Omega_{\Mcal_{\Kstar}^{\star}(\C)})\right).
\end{equation}
Recall from Proposition \ref{prop:OmegaCpxDecom} that
\begin{displaymath}
	\deg^{\ell-\mathrm{new}}c_{1}(\Omega_{\Mcal_{\Kstar}^{\star}(\C)})^{2}=2\deg^{\ell-\mathrm{new}} c_{2}(\Omega_{\Mcal_{\Kstar}^{\star}(\C)}).
\end{displaymath}
Therefore \eqref{eq:chi-new} simplifies to
\begin{equation}\label{eq:chi-new-bis}
	\chi^{\ell-\mathrm{new}}(\Mcal_{\Kstar}^{\star}(\C),\Ocal)=\frac{1}{8}\deg^{\ell-\mathrm{new}} c_{1}(\Omega_{\Mcal_{\Kstar}^{\star}(\C)})^{2}.
\end{equation}
For the Shimura curve, Riemann-Roch reads
\begin{equation}\label{eq:chi-Sh-curve}
	\chi(\Scal_{1}(\C),\Ocal)=-\frac{1}{2}\deg c_{1}(\Omega_{\Scal_{1}(\C)}).
\end{equation}
Lew us now compare Euler-Poincar\'e characteristics. By \cite[Variante 2.5]{Deligne:Shimura}, the varieties $\Mcal_{\Kstar}^{\star}(\C)$, $\Mcal_{\Kstar^{\prime}}^{\star}(\C)$, $S_{1}$ and $S_{2}$ all have the same number of connected components. This implies that
\begin{equation}\label{eq:comp-H0}
	\dim H^{0}(\Mcal_{\Kstar}^{\star}(\C),\Ocal)-2\dim H^{0}(\Mcal_{\Kstar^{\prime}}^{\star}(\C),\Ocal)=-\frac{1}{2}\dim H^{0}(\Scal_{1}(\C),\Ocal).
\end{equation}
By Serre duality and the Jacquet-Langlands correspondence applied to holomorphic quaternionic modular forms of parallel weight 2 (similar argument leading to Theorem \ref{thm:JL-an-torsion}) we also find
\begin{equation}\label{eq:comp-H2}
	\dim H^{2}(\Mcal_{\Kstar}^{\star}(\C),\Ocal)-2\dim H^{2}(\Mcal_{\Kstar^{\prime}}^{\star}(\C),\Ocal)=\frac{1}{2}\dim H^{1}(\Scal_{1}(\C),\Ocal).
\end{equation}
Finally, we claim $H^{1}(\Mcal_{\Kstar}^{\star}(\C),\Ocal)$ vanishes. Indeed, by Serre duality, this space is dual to $H^0(\Mcal_{\Kstar}^{\star}(\C), \Omega)$. By a construction analogous to the discussion leading to \eqref{eq:L2-decomposition}, a non-zero element $\eta \in H^0(\Mcal_{\Kstar}^{\star}(\C), \Omega)$ would give rise to an automorphic representation $\pi$ that is annilihated by the laplacian, and such that
\[
	\Hom_{\Kstar \Kstar_{\infty}} \left( (1 \otimes w_2^{\vee}) \oplus (w_2^{\vee} \otimes 1) , \, \pi \right) \neq 0 .
\]
In particular, one of $\pi_{v_1}$ or $\pi_{v_2}$ is generated by a weight zero vector killed by the Laplacian, and so is the trivial representation; arguing as in the proof of Proposition \ref{theorem:multiplicities}, the Jacquet-Langlands correspondence produces a cusp form on $GL_2(\A_{F})$ that is constant in one of the variables, which necessarily vanishes. Thus $H^0(\Mcal_{\Kstar}^{\star}(\C), \Omega) = 0$ as required.

Hence, from \eqref{eq:comp-H0}--\eqref{eq:comp-H2}, we infer
\begin{equation}\label{eq:relation-chi}
	\chi^{\ell-\mathrm{new}}(\Mcal_{\Kstar}^{\star}(\C),\Ocal)=-\frac{1}{2}\chi(\Scal_{1}(\C),\Ocal).
\end{equation}
We conclude by combining \eqref{eq:chi-new-bis}, \eqref{eq:chi-Sh-curve} and \eqref{eq:relation-chi}.
\end{proof}
\subsection{Comparison of arithmetic intersections}
In this subsection we mimic the proof of Proposition \ref{prop:geom-deg} in the arithmetic setting to obtain our second main result, which we phrase in terms of the original PEL Shimura variety $\Mcal_K$ attached to the algebraic group $\Gbf$ as in the introduction.  
\begin{theorem}\label{theorem:comparison}
Let $K \subset \Gbf(\A_{\Q,f})$ and $K_1 \subset \Gbf_1(\A_{\Q,f})$ denote any sufficiently small compact open subgroups, and $\Mcal_K$ and $\Scal_1 = \Scal_{1, K_1}$ the corresponding Shimura varieties over $\Q$ and $F$ respectively. Then
\begin{equation}\label{eq:main-comparison}
	\frac{\ac_{1}(\ov{\Omega}_{\Mcal_{K}})\ac_{2}(\ov{\Omega}_{\Mcal_{K}})}{\deg c_{1}(\Omega_{\Mcal_{K}(\C)})^{2}}\equiv\frac{\ac_{1}(\ov{\Omega}_{\Scal_{1}})^{2}}{\deg c_{1}(\Omega_{\Scal_{1}(\C)})}\quad\text{in }\R/\log |\ov{\Q}^{\times}|,
\end{equation}
where $\ov{\Q}$ is the algebraic closure of $\Q$ in $\C$.
\end{theorem}
Note that since this equality takes place in $\R / \log|\bar\Q^{\times}|$, only the generic fibres of $\Mcal_K$ and $\Scal_1$ play a role.
An immediate consequence of the theorem and Theorem \ref{thm:c1c2thm} is:
\begin{corollary}
For any sufficiently small level structure $\Kstar_1 \subset \Gbf_1(\A_{\Q,f})$,
\begin{displaymath}
	\frac{\ac_{1}(\ov{\Omega}_{\Scal_{1}/F})^{2}}{\deg c_{1}(\Omega_{\Scal_{1}(\C)})}\equiv \  -4 \log \pi - 2\gamma + 1  +  \frac{\zeta'_F(2) }{\zeta_F(2)}  \in \R / \log|\ov\Q^{\times}|.
\end{displaymath}
\end{corollary}

The main step towards the comparison theorem is a consequence of the Jacquet-Langlands correspondence applied to determinants of cohomology for the coarse Hilbert modular surface $\Mcal^{\star}_{\Kstar}$. We continue with the notations from the previous sections; in particular, $\Kstar \subset \Gbf^{\star}(\A_{\Q, f})$ is sufficiently small and Iwahori at the fixed prime $\ell$, and $\Kstar'$ is maximal at $\ell$ and agrees with $\Kstar$ everywhere else. 

In analogy with $\deg^{\ell-\mathrm{new}}$ and $\chi^{\ell-\mathrm{new}}$ in \textsection \ref{subsec:comparison-volumes}, we introduce
\begin{align*}
	\widehat{\deg}\;^{\ell-\mathrm{new}}\det H^{\bullet}&(\Mcal_{\Kstar}^{\star},\Ocal)_{Q} := \\ 
	&\widehat{\deg}\;\det H^{\bullet}(\Mcal_{\Kstar}^{\star},\Ocal)_{Q} 
	-2\,\widehat{\deg}\;\det  H^{\bullet}(\Mcal_{\Kstar'}^{\star},\Ocal)_{Q}  
	 \ \in \ \R/\log|\Q^{\times}|.
\end{align*}
This quantity takes values in $\R/\log|\Q^{\times}|$, because $\Mcal_{\Kstar}^{\star}$ is just defined over $\Q$. However, the application of the Jacquet-Langlands correspondence will force us to work with its reduction modulo $\log |\ov{\Q}^{\times}|$.
\begin{proposition} Suppose $\Kstar$ is Iwahori at $\ell$, and $\Kstar_1 \subset \Gbf_1(\A_{\Q,f})$ matches $\Kstar$. Then
\begin{equation}\label{eq:equality-det-coh}
	\adeg^{\ell-\mathrm{new}}\det H^{\bullet}(\Mcal_{\Kstar}^{\star},\Ocal)_{Q} 
	\equiv -\adeg\det H^{\bullet}(\Scal_{1},\Ocal)_{Q}
\end{equation}
in $\R/\log |\ov{\Q}^{\times}|$.
\end{proposition}
\begin{proof}
First, note that the arithmetic Grothendieck-Riemann-Roch theorem implies that the arithmetic degrees in question are multiplicative for finite \'etale covers. Hence, arguing as in the proof of Proposition \ref{prop:geom-deg}, we may reduce to the case $\Kstar = \KstarN$ as in \eqref{eqn:KstarNDef}.

Second, recall that the Quillen metric is a rescaling of $L^{2}$ metric by the exponential of the holomorphic analytic torsion. By definition of $\adeg^{\ell-\mathrm{new}}$ and Theorem \ref{thm:JL-an-torsion}, we only need to prove the analogue of \eqref{eq:equality-det-coh} for the $L^{2}$ metric instead of the Quillen metric.

Recall that $H^{1}(\Mcal_{\Kstar}^{\star}(\C),\Ocal)$ and $H^{1}(\Mcal_{\Kstar'}^{\star}(\C),\Ocal)$ vanish, hence the $\Q$ vector spaces $H^{1}(\Mcal_{\Kstar}^{\star},\Ocal)$ and $H^{1}(\Mcal_{\Kstar'}^{\star},\Ocal)$ vanish as well. Consequently,
\begin{equation}\label{eq:JL-1}
	\adeg^{\ell-\mathrm{new}} \det H^{1}(\Mcal^{\star}_{\Kstar},\Ocal)_{L^{2}}=0\quad\text{in }\R/\log|\Q^{\times}|.
\end{equation}
The next easiest cohomology groups are in degree 0. We observe that
\begin{displaymath}
	\adeg \det H^{0}(\Mcal_{\Kstar}^{\star},\Ocal)_{L^{2}}=-\frac{1}{2}\sum_{C}\log\left(\int_{C}\frac{\omega^{2}}{2}\right),
\end{displaymath}
where $C$ runs over the connected components of $\Mcal_{\Kstar}^{\star}(\C)$ and $\omega$ is the normalized K\"ahler form used to define the $L^2$ metric. Similar statements hold for $\Mcal_{\Kstar'}^{\star}$ and $\Scal_{1}$. A straightforward computation shows that
\begin{displaymath}
	\omega=\frac{1}{8\pi^{2}}c_{1}(\ov{\Omega}_{\Mcal_{\Kstar}^{\star}(\C)}),
\end{displaymath}
and analogously for $\Mcal_{\Kstar'}^{\star}$ and $\Scal_{1}$. Because
\begin{displaymath}
	\int_{C}c_{1}(\ov{\Omega}_{\Mcal_{\Kstar}^{\star}(\C)})^{2}\in\Z,
\end{displaymath}
the definition of $\adeg^{\ell-\mathrm{new}}$ and the fact that $\pi_{0}(\Mcal_{\Kstar}^{\star}(\C))=\pi_{0}(\Mcal_{\Kstar'}^{\star}(\C))$ implies that
\begin{displaymath}
	\adeg^{\ell-\mathrm{new}}\det H^{0}(\Mcal_{\Kstar}^{\star},\Ocal_{L^{2}}) \equiv -2(\log\pi)\ \sharp\pi_{0}(\Mcal_{\Kstar}^{\star}(\C))\quad\text{in }\R/\log|\Q^{\times}|.
\end{displaymath}
Similarly,
\begin{displaymath}
	\adeg\det H^{0}(\Scal_{1},\Ocal)_{L^{2}}\equiv(\log\pi)\sharp\pi_{0}(\Scal_{1}(\C))\quad\text{in }\R/\log|\Q^{\times}|.
\end{displaymath}
Since 
\[\sharp\pi_{0}(\Scal_{1}(\C))= \# \pi_0(S_1)  +  \sharp \pi_0 (S_2) \ = \  2\ \sharp\pi_{0}(\Mcal_{\Kstar}^{\star}(\C)),\] 
we infer
\begin{equation}\label{eq:JL-2}
	\adeg^{\ell-\mathrm{new}} \det H^{0}(\Mcal_{\Kstar}^{\star},\Ocal)_{L^{2}}=-\adeg\det H^{0}(\Scal_{1},\Ocal)_{L^{2}}.
\end{equation}
For the remaining cohomology groups, Serre duality (and its compatibility with the $L^{2}$ metric) gives
\begin{equation}\label{eq:JL-3}
	\adeg^{\ell-\mathrm{new}}\det H^{2}(\Mcal_{\Kstar}^{\star},\Ocal)_{L^{2}}=-\adeg^{\ell-\mathrm{new}}\det H^{0}(\Mcal_{\Kstar}^{\star},\omega_{\Mcal_{\Kstar}^{\star}/\Q})_{L^{2}}
\end{equation}
and
\begin{equation}\label{eq:JL-4}
	\adeg \det H^{1}(\Scal_{1},\Ocal)_{L^{2}}=-\adeg \det H^{0}(\Scal_{1},\omega_{\Scal_{1}/F})_{L^{2}}.
\end{equation}
We are thus reduced to compare spaces of holomorphic quaternionic forms. Let us introduce the $\ell$-new quotient
\begin{displaymath}
	 \xymatrix{
	 	H^{0}(\Mcal_{\Kstar}^{\star},\omega_{\Mcal_{\Kstar}^{\star}/\Q})\ar@{->>}[r] 	& H^{0}(\Mcal_{\Kstar}^{\star},\omega_{\Mcal_{\Kstar}^{\star}/\Q})^{\ell-\mathrm{new}},
	}
\end{displaymath}
with the induced $L^{2}$ metric. Because $\ell$-old-forms are orthogonal to $\ell$-new-forms, we actually have
\begin{equation}\label{eq:JL-5}
	\adeg^{\ell-\mathrm{new}}\det H^{0}(\Mcal_{\Kstar}^{\star},\omega_{\Mcal_{\Kstar}^{\star}/\Q})_{L^{2}} \equiv \adeg \det H^{0}(\Mcal_{\Kstar}^{\star},\omega_{\Mcal_{\Kstar}^{\star}/\Q})^{\ell-\mathrm{new}}_{L^{2}}
\end{equation}
in $\R/\log |\Q^{\times}|$. By the Jacquet-Langlands correspondence, there are Hecke equivariant isomorphisms of $\C$-vector spaces
\begin{displaymath}
	H^{0}(\Mcal_{\Kstar}^{\star},\omega_{\Mcal_{\Kstar}^{\star}/\Q})^{\ell-\mathrm{new}}_{\C}\overset{\sim}{\longrightarrow} H^{0}(S_{j},\omega_{S_{j} }),\quad j=1,2.
\end{displaymath}
The isomorphism can be normalized so that it descends to $\ov{\Q}$ (the curves $S_{j} = \Scal_{1,v_j}(\C)$ are defined over $\ov{\Q}$). 
If $f\in H^{0}(\Mcal_{\Kstar}^{\star},\omega_{\Mcal_{\Kstar}^{\star}/\Q})^{\ell-\mathrm{new}}_{\ov{\Q}}$ is a Hecke eigenform that corresponds to a pair of Hecke eigenforms 
\[ g_{j}\in H^{0}(S_{j},\omega_{S_{j}/\ov{\Q}}), \qquad j=1,2 \]
 then, by \cite[Thm. 5.4]{Shimura}, there is a relation of $L^{2}$ norms
\[
	\|f\|_{L^{2}} =  C \cdot \|g_{1}\|_{L^{2}}\|g_{2}\|_{L^{2}}
\]
for some $C \in \ov{\Q}^{\times}$.
Now take a $\Q$-basis of $H^{0}(\Mcal_{\Kstar}^{\star},\omega_{\Mcal_{\Kstar}^{\star}/\Q})^{\ell-\mathrm{new}}$ (resp. an $F$-basis of $H^{0}(\Scal_{1},\omega_{\Scal_{1}/F})$) and express it in terms of a $\ov{\Q}$-basis of Hecke eigenforms. Recall that a basis of Hecke eigenforms is an orthogonal set. Recall as well that $\Scal_{1}(\C)=S_{1}\sqcup S_{2}$. We thus derive from Shimura's relation of norms that
\begin{equation}\label{eq:JL-6}
	\adeg \det H^{0}(\Mcal_{\Kstar}^{\star},\omega_{\Mcal^{\star}_{\Kstar}/\Q})^{\ell-\mathrm{new}}_{L^{2}}\equiv\adeg \det H^{0}(\Scal_{1},\omega_{\Scal_{1}/F})_{L^{2}}\quad \in  \ \R / \log |\ov{\Q}^{\times}|.
\end{equation}
Collecting equations \eqref{eq:JL-1}--\eqref{eq:JL-6}, and by the definition of determinant of cohomology, we conclude the proof.
\end{proof}
We are now in position to prove the comparison theorem.
\begin{proof}[Proof of Theorem \ref{theorem:comparison}]
As usual, by the projection formula and functoriality of Chern classes, the ratio of arithmetic and geometric degrees is invariant under finite \'etale covers. As a consequence, working in $\R / \log|\Q^{\times}|$, we have
\[
	\frac{\ac_{1}(\ov{\Omega}_{\Mcal_{K}})\ac_{2}(\ov{\Omega}_{\Mcal_{K}})}{\deg c_{1}(\Omega_{\Mcal_{\Kstar}(\C)})^{2}}\equiv 	\frac{\ac_{1}(\ov{\Omega}_{\Mcal_{K'}})\ac_{2}(\ov{\Omega}_{\Mcal_{K'}})}{\deg c_{1}(\Omega_{\Mcal_{K'}(\C)})^{2}} \in \R / \log|\Q^{\times}|
\]
for any sufficiently small $K, K' \subset \Gbf(\A_{\Q,f})$. Now consider the coarse Shimura variety $\Mcal^{\star}_{\Kstar}$ over $\Q$ attached to $\Kstar = \KstarN$ as in \eqref{eqn:KstarNDef}. By Section \ref{subsec:restriction-scalars}, there exists $K' \subset \Gbf(\A_{\Q,f})$ and a finite \'etale cover of $\Q$-schemes
\begin{displaymath}
	\pi\colon\coprod_{h}(\Mcal_{K'} )_{/ \Q} \longrightarrow\Mcal_{\Kstar}^{\star}
\end{displaymath}
such that  $\pi^{\ast}(\ov{\Omega}_{\Mcal_{K}^{\star}/\Q})=\ov{\Omega}_{\coprod \Mcal_{K'}/\Q}$; it follows that
\[
	\frac{\ac_{1}(\ov{\Omega}_{\Mcal_{K}})\ac_{2}(\ov{\Omega}_{\Mcal_{K}})}{\deg c_{1}(\Omega_{\Mcal_{K}(\C)})^{2}} \equiv 	\frac{\ac_{1}(\ov{\Omega}_{\Mcal_{K'}})\ac_{2}(\ov{\Omega}_{\Mcal_{K'}})}{\deg c_{1}(\Omega_{\Mcal_{K'}(\C)})^{2}} \equiv 	\frac{\ac_{1}(\ov{\Omega}_{\Mcal_{\Kstar}^{\star}})\ac_{2}(\ov{\Omega}_{\Mcal_{\Kstar}^{\star}})}{\deg c_{1}(\Omega_{\Mcal_{\Kstar}^{\star}(\C)})^{2}} \in \R / \log|\Q^{\times}|. 
\]
Analagous considerations show that for the Shimura curve $\Scal_1$, the ratio of arithmetic and geometric degrees is independent of the compact open $\Kstar_1 \subset \Gbf_1(\A_{\Q,f})$, and so we may assume that $\Kstar_1$ matches with $\Kstar=\KstarN$.
%

By the arithmetic Grothendieck-Riemann-Roch theorem and the previous definitions, we may write
\begin{align}\label{eq:ARR-HMS}
	\adeg^{\ell-\mathrm{new}} \det & H^{\bullet}(\Mcal_{\Kstar}^{\star},\Ocal)_{Q} \equiv \notag \\
		&-\frac{1}{24}\deg^{\ell-\mathrm{new}} \left(c_{1}(\Omega_{\Mcal_{\Kstar}^{\star}(\C)})^{2}\right) \,
	\frac{\ac_{1}(\ov{\Omega}_{\Mcal_{\Kstar}^{\star}/\Q})\ac_{2}(\ov{\Omega}_{\Mcal_{\Kstar}^{\star}/\Q})}{\deg c_{1}(\Omega_{\Mcal_{\Kstar}^{\star}(\C)})^{2}}  \notag \\
	& \qquad \qquad -\frac{1}{4}(2\zeta^{\prime}(-1)+\zeta(-1))\deg^{\ell-\mathrm{new}} \left(c_{1}(\Omega_{\Mcal_{\Kstar}^{\star}(\C)})^{2}\right).
\end{align}
Analogously, for the Shimura curve $\Scal_{1}$, the arithmetic Grothendieck-Riemann-Roch gives
\begin{equation}\label{eq:ARR-SC}
	\begin{split}
	\adeg\det H^{\bullet}(\Scal_{1},\Ocal)_{Q} \equiv &\frac{1}{12}\deg c_{1}(\Omega_{\Scal_{1}(\C)})\,
	\frac{\ac_{1}(\ov{\Omega}_{\Scal_{1}/F})^{2}}{\deg c_{1}(\Omega_{\Scal_{1}(\C)})}\\
	&+\frac{1}{2}(2\zeta^{\prime}(-1)+\zeta(-1))\deg c_{1}(\Omega_{\Scal_{1}(\C)}).
	\end{split}
\end{equation}
By the previous proposition, we have an equality \eqref{eq:ARR-HMS}$=-$\eqref{eq:ARR-SC}. Furthermore, by Proposition \ref{prop:geom-deg} there is a relation of degrees $\deg^{\ell-\mathrm{new}}c_{1}(\Omega_{\Mcal_{K}^{\star}(\C)})^{2}=2\,\deg c_{1}(\Omega_{\Scal_{1}(\C)})$. The theorem now follows from the last two facts.
\end{proof}

\end{document}